\documentclass[reqno, 11pt]{amsart} 
\usepackage{amsfonts, amsmath, amssymb, amsthm}
\usepackage[margin=2.75cm, heightrounded]{geometry}
\usepackage{fancybox}
\usepackage{hhline, float}
\usepackage{mathrsfs}
\usepackage{nccmath}
\usepackage[dvipsnames]{xcolor}
\usepackage[colorlinks=true]{hyperref}
\hypersetup{citecolor=red, linkcolor=blue}

\newtheorem{thm}{Theorem}[section]
\newtheorem{lemma}[thm]{Lemma}
\newtheorem{prop}[thm]{Proposition}

\theoremstyle{definition}

\newcommand{\pa}{\partial}

\newcommand{\N}{\mathbb{N}}
\newcommand{\R}{\mathbb{R}}
\renewcommand{\S}{\mathbb{S}}

\newcommand{\dx}{\textup{d}x}
\newcommand{\dy}{\textup{d}y}
\newcommand{\dz}{\textup{d}z}

\allowdisplaybreaks
\numberwithin{equation}{section}

\makeatletter
\@namedef{subjclassname@2020}{%
\textup{2020} Mathematics Subject Classification}
\makeatother

\begin{document}
\title[Construction of solutions for a critical elliptic system]
{Construction of solutions for a critical elliptic system of Hamiltonian type}

\author{Yuxia Guo}
\address[Yuxia Guo]{Department of Mathematics and Sciences, Tsinghua University}
\email{yguo@tsinghua.edu.cn}

\author{Congzheng Xuanyuan}
\address[Congzheng Xuanyuan]{Department of Mathematics and Sciences, Tsinghua University}
\email{xycz23@mails.tsinghua.edu.cn}

\author{Tingfeng Yuan}
\address[Tingfeng Yuan]{Department of Mathematics and Sciences, Tsinghua University}
\email{ytf22@mails.tsinghua.edu.cn}

\begin{abstract}
We consider the following nonlinear elliptic system of Hamiltonian type with critical exponents:
\begin{equation*}
    \begin{cases}
        -\Delta u + V(|y'|,y'')\, u = |v|^{p-1}v, & \text{in } \R^N,\\
        -\Delta v + V(|y'|,y'')\, v = |u|^{q-1}u, & \text{in } \R^N, \\
    \end{cases}
\end{equation*}
where $(y', y'') \in \R^2 \times \R^{N-2}$, $V(|y'|, y'') \not\equiv 0$ is a bounded, nonnegative function on $\R_+ \times \R^{N-2}$ and $p, q > 1$ lie on the critical hyperbola:
\[
    \frac{1}{p+1} + \frac{1}{q+1} = \frac{N-2}{N}.
\]
By applying the finite-dimensional reduction method and local Pohozaev identities combined with the Green representation formula and technical analysis, we show that, under the assumptions that $N \ge 5$, $(p,q)$ lies in a certain admissible range, and $r^2 V(r, y'')$ has a stable critical point, the above problem admits infinitely many solutions whose energy can be made arbitrarily large.
\end{abstract}


\maketitle


\section{Introduction}
In this paper, we will study the following nonlinear elliptic system of Hamiltonian type with critical growth:
\begin{equation}\label{Main}
    \begin{cases}
        -\Delta u + V(|y'|,y'') u = |v|^{p-1}v, \;\; &\text{in} \;\; \R^N,\\
        -\Delta v + V(|y'|,y'') v = |u|^{q-1}u, \;\; &\text{in} \;\; \R^N, \\
        (u,v) \in X : = (\dot{W}^{2,\frac{p+1}{p}} (\R^N) \cap L_V^2(\R^N)) \times (\dot{W}^{2,\frac{q+1}{q}} (\R^N) \cap L_V^2(\R^N)) ,
    \end{cases}
\end{equation}
where $(y', y'') \in \R^2 \times \R^{N-2}$,  $V(|y'|, y'')\not\equiv0$ is a bounded non-negative function in $\R_+\times \R^{N-2}$, $L^2_V(\R^N) = L^2(\R^N, V(|y'|,y'')\dy)$ and $p,q>1$ lie on the critical hyperbola:
    \begin{equation}\label{critical hyperbola}
        \dfrac{1}{p+1} + \dfrac{1}{q+1} = \dfrac{N-2}{N},	
	\end{equation}
Without loss of generality, we may assume that $p\leq \frac{N+2}{N-2}\leq q$.

Define
\[
    H(y,u,v) = -V(y)uv + \frac{|u|^{q+1}}{q+1} + \frac{|v|^{p+1}}{p+1}.
\]
Then equations in  \eqref{Main} can be rewritten in the following form:
\begin{equation*}
    \begin{cases}
        -\Delta u = H_v(y,u,v),\\[3pt]
        -\Delta v = H_u(y,u,v).
    \end{cases}
\end{equation*}
In the literature, systems of the above form are usually referred to as elliptic systems of Hamiltonian type. Such systems are also said to be strongly coupled, in the sense that $u \equiv 0$ if and only if $v \equiv 0$.
Elliptic systems of Hamiltonian type have been extensively studied in the literature. Many results concerning on the existence, nonexistence, multiplicity, symmetry, and other qualitative properties of their solutions have been obtained; see, for instance, \cite{Figueiredo2, BST, Figueiredo1,Figueiredo3,Ding, Ruf,Sirakov} and the references therein. In these works, most results are obtained by using variational methods such as the dual method, the generalized mountain pass theorem, and the generalized linking theorem. The systems considered there usually involve nonlinear terms $H(y,u,v)$ under {\it subcritical} growth conditions, namely:
\[
    \dfrac{1}{p+1} + \dfrac{1}{q+1} > \dfrac{N-2}{N}.
\]
However, when one considers the {\it critical} case, that is, when the growth exponents $p$ and $q$ lie on the critical hyperbola:
\[
\dfrac{1}{p+1} + \dfrac{1}{q+1} = \dfrac{N-2}{N},
\]
the classical variational methods cannot be applied directly, not only because of the loss of compactness in the Sobolev embedding, but also due to the strong indefiniteness of the corresponding Euler-Lagrange functional, in the sense that the functional is neither bounded from below nor from above.

\vskip8pt

The purpose of  the present paper is to apply Lyapunov-Schmidt reduction arguments and local Pohozaev identities to study the existence and the profile of the solutions for  system \eqref{Main}.

\vskip8pt

Note that when $p = q$ and $u = v>0$, system \eqref{Main} reduces to the following nonlinear elliptic problem with critical exponent:
\begin{equation}\label{single equation}
    -\Delta u + V(y)u = u^{\frac{N+2}{N-2}}, \quad u > 0, \quad u \in H^1(\R^N).
\end{equation}
Problem\eqref{single equation} is related to the well-known Brezis-Nirenberg problem on $\mathbb{S}^N$:
\begin{equation}\label{BN problem in SN}
    -\Delta_{\mathbb{S}^N} u = u^{\frac{N+2}{N-2}} + \lambda u,
    \quad u > 0 \;\; \text{ on } \;\; \mathbb{S}^N.
\end{equation}
Using stereographic projection, the problem \eqref{BN problem in SN} can be transformed into \eqref{single equation} with
\[
    V(y) = -\frac{-4\lambda - N(N-2)}{(1+|y|^2)^2}, \quad \text{and} \quad V(y) > 0 \;\;\text{ when } \;\;\lambda < -\frac{N(N-2)}{4}.
\]
The problem \eqref{BN problem in SN} has attracted a great deal of attention in recent decades and has been studied extensively, and we refer the interested reader to \cite{BandleWei, BrezisLi, BrezisPeletier, Druet, GidasSpruck} and the references therein.

\vskip8pt

    Recall that if $V \ge 0$ and $V \not\equiv 0$, then the mountain pass value for problem \eqref{single equation} is not a critical value of the corresponding functional. Therefore, all arguments based on the concentration compactness principle in \cite{Lions2, Lions1} cannot be applied to obtain the existence of solutions for \eqref{single equation}. As far as we know, the first existence result for~\eqref{single equation} is due to Benci and Cerami, see \cite{BenciCerami}. They showed that if $\|V\|_{L^{\frac{N}{2}}(\R^N)}$ is sufficiently small, then \eqref{single equation} admits a solution whose energy lies in the interval
    \[
        \left( \frac{1}{N} S^{\frac{N}{2}}, \frac{2}{N} S^{\frac{N}{2}} \right),
    \]
    where $S$ denotes the best Sobolev constant in the embedding $D^{1,2}(\R^N) \hookrightarrow L^{\frac{2N}{N-2}}(\R^N)$. Regarding multiplicity results, Chen, Wei, and Yan \cite{CWY} applied the reduction method to prove that~\eqref{single equation} has infinitely many non-radial solutions whose energy can be made arbitrarily large, provided that $N \ge 5$, $V(y)$ is radially symmetric, and $r^2V(r)$ has a local maximum point, or a local minimum point $r_0$ with $V(r_0) > 0$. Note that under the assumption that $V(y)$ is radially symmetric, this condition is also necessary for the existence of solutions, since by the following Pohozaev identity
    \begin{equation}
        \int_{\R^N} \big( V(|y|) + \tfrac{1}{2} |y| V'(|y|) \big) \, dy = 0,
    \end{equation}
    the problem \eqref{single equation} admits no solution if $r^2V(r)$ is monotone increasing or decreasing. Later, Peng, Wang, and Yan~\cite{PWY} considered a weaker symmetry condition on $V(y)$ and proved that problem \eqref{single equation} has infinitely many solutions whose energy can be made arbitrarily large.

     \vskip8pt

    To our best knowledge, very few results are available for Hamiltonian systems with critical growth. In this paper, we study the existence of infinitely many solutions for the Hamiltonian system \eqref{Main} under a weaker symmetry condition on $V(y)$. More precisely,  we consider the potential $V(y) = V(|y'|, y'')$ with $y = (y', y'') \in \R^2 \times \R^{N-2}$ and assume that $V(r, y'')$ satisfies the following condition:

    \vspace{0.5em}
    \noindent
    \textbf{(V):} Suppose that $r^2 V(r, y'')$ has a critical point $(r_0, y_0'')$ satisfying $r_0 > 0$, $V(r_0, y_0'') > 0$, and
    \[
        \deg\big( \nabla (r^2 V(r, y'')), (r_0, y_0'') \big) \ne 0.
    \]
    \vspace{0.5em}
    \noindent
    Here  is our main result.
    \begin{thm}\label{Thm1}
        Suppose that $V \in C^1(\R^N)$ is bounded, nonnegative, and satisfies condition~\textnormal{(V)}. If $N\geq 6$, $p\in(\frac{N}{N-2},\frac{N+2}{N-2})$ or $N=5$, $p\in (2,\frac{7}{3})$, and $p, q$ satisfy~\eqref{critical hyperbola}, then problem~\eqref{Main} admits infinitely many solutions whose energy can be made arbitrarily large.
    \end{thm}

    The proof of Theorem \ref{Thm1} is mainly based on a finite-dimensional reduction argument and local Pohozaev identities. To perform the reduction procedure, it is essential to determine the location of the bubbles. In \cite{CWY}, the location of the bubbles was found by using a minimization or maximization procedure, and the bubble solutions constructed there concentrate at a local maximum or minimum point of $r^2V(r)$. However, in our case, $(r_0,y''_0)$ may be a saddle point of $r^2V(r,y'')$, so the above argument may fail. As in \cite{Rey}, it is natural to locate the bubbles by calculating the derivatives of the reduced functional. However, as shown in Lemma \ref{expansion_2}, this approach cannot be applied here because it is difficult to obtain a sufficiently good estimate for the error term. Instead, following \cite{PWY}, we use the local Pohozaev identities to derive algebraic equations that determine the location of the bubbles. To overcome the lack of parameters, as in \cite{WY}, we use the number of bubbles in the solutions as a parameter to construct infinitely many positive bubbling solutions. This method has been successfully applied to construct infinitely many positive solutions for other non-compact elliptic problems; see \cite{CWY, DLY, GLP, GPY, PWY, WY1, WY2} and the references therein.

    To proceed, we first introduce some notation and definitions. Let $N \geq 3$, and let $(U,V)$ be a positive ground state solution of (see \cite{Lions1})
    \begin{equation}\label{lane-embden}
        \begin{cases}
            -\Delta U = |V|^{p-1}V, \quad \text{in } \R^N,\\
            -\Delta V = |U|^{q-1}U, \quad \text{in } \R^N,\\
            (U,V) \in \dot{W}^{2,\frac{p+1}{p}}(\R^N) \times \dot{W}^{2,\frac{q+1}{q}}(\R^N).
        \end{cases}
    \end{equation}
    It is known that $(U,V)$ is radially symmetric and decreasing after a suitable translation (see \cite{Lions2}). Moreover, the results of Wang \cite{Wang} and Hulshof and Van der Vorst \cite{HV} show that there exists a positive ground state solution $(U_{0,1}, V_{0,1})$ such that $U_{0,1}(0) = 1$, which is unique up to translation and scaling. The family of functions $\{(U_{x,\lambda}, V_{x,\lambda})\}$ defined by
    \begin{equation}
        (U_{x,\lambda}, V_{x,\lambda}) = \big(\lambda^{\frac{N}{q+1}} U_{0,1}(\lambda(y - x)), \ \lambda^{\frac{N}{p+1}} V_{0,1}(\lambda(y - x))\big), \quad \text{for any } x \in \R^N \text{ and } \lambda > 0,
    \end{equation}
    exhausts all the positive ground state solutions of \eqref{lane-embden}.

   \vskip8pt

   We define the function spaces $L_s$ and $H_s$ by
   \[
       L_s = \left\{ (u_1, u_2)  \ \middle| \
       \begin{aligned}
       & u_1, u_2 \ \ \text{are measurable functions in} \ \ \R^N, \\
       &u_i (r \cos \theta, r \sin \theta, y'')
        = u_i \!\left(r \cos \!\left(\theta + \tfrac{2\pi j}{m}\right),\,
        r \sin \!\left(\theta + \tfrac{2\pi j}{m}\right),\, y''\right), \\[2mm]
        &u_i (y_1, -y_2, y'') = u_i (y_1, y_2, y''),
        \quad i = 1, 2,\ \ j = 1, 2, \ldots, m
       \end{aligned}
        \right\},
    \]
    and $H_s = L_s \cap X$, where $X$ is given in \eqref{Main}. We also define
    \[
        x_j = \left(\bar{r} \cos \frac{2(j-1)\pi}{m},\ \bar{r} \sin \frac{2(j-1)\pi}{m},\ \bar{y}''\right),\quad j = 1, 2, \ldots, m,
    \]
    where $\bar{y}'' \in \R^{N-2}$ and $\bar{r} \in \R_+$ will be determined later. By the weak symmetry of $V(y)$, we observe that $V(x_j) = V(\bar{r}, \bar{y}'')$ for all $j = 1, \ldots, m$.

    \vskip8pt

    Let $\delta > 0$ be a small constant such that
    \[
        r^2 V(r, y'') > 0 \;\; \text{if} \;\; |(r, y'') - (r_0, y_0'')| \leq 10\delta.
    \]
    Let $\xi(y) = \xi(|y'|, y'')$ be a smooth function satisfying $0 \leq \xi \leq 1$ and
    \[
        \xi = \begin{cases}
            1, \;\; \text{if} \;\; |(r, y'') - (r_0, y_0'')| \leq \delta, \\
            0, \;\;  \text{if} \;\; |(r, y'') - (r_0, y_0'')| \geq 2\delta.
        \end{cases}
    \]
    We denote
    \begin{align}\label{ansatzes_1}
        (Y_{\bar{r}, \bar{y}'', \lambda}, Z_{\bar{r}, \bar{y}'', \lambda})= (\xi Y_{\bar{r}, \bar{y}'', \lambda}^*,\, \xi Z_{\bar{r}, \bar{y}'', \lambda}^*),
    \end{align}
    where $Z_{\bar{r}, \bar{y}'', \lambda}^* = \sum\limits_{j=1}^{m} V_{x_j,\lambda}$, and $Y_{\bar{r}, \bar{y}'', \lambda}^*$ is defined as the solution of
    \begin{align}\label{ansatzes_2}
        -\Delta Y_{\bar{r}, \bar{y}'', \lambda}^* = (Z_{\bar{r}, \bar{y}'', \lambda}^*)^p \quad \text{in } \R^N, \quad Y_{\bar{r}, \bar{y}'', \lambda}^* \in \dot{W}^{2,\frac{p+1}{p}}(\R^N).
    \end{align}
    We will use $(Y_{\bar{r}, \bar{y}'', \lambda}(y), Z_{\bar{r}, \bar{y}'', \lambda}(y))$ as an approximation solution of \eqref{Main}. Note that here we need to apply a cut-off on $(Y_{\bar{r}, \bar{y}'', \lambda}^*,\, Z_{\bar{r}, \bar{y}'', \lambda}^*)$ because we need to deal with the slow decay of it when $N$ is not large.

    One of the difficulties in this paper is to estimate $Y_{\bar{r}, \bar{y}'', \lambda}^*$. In Appendix B, we successfully derive an estimate for $Y_{\bar{r}, \bar{y}'', \lambda}^*$ by applying the Green’s representation formula together with a careful analytical study. Define $$\varphi = Y_{\bar{r}, \bar{y}'', \lambda}^* - \sum\limits_{j=1}^m U_{x_j,\lambda}.$$
    It turns out that the function $\varphi$ plays an important role in the energy expansion, as it carries the leading term of that expansion, see Appendix C for details. To better capture the main contribution, we further decompose $\varphi$ by introducing a new function $w$ that is defined as the solution to \eqref{def-w}.

    \vspace{0.5em}

    Note that in \cite{PWY}, the approximate solution is obtained by taking the sum of the cut-offs for Talanti bubbles. If we follow the method of \cite{PWY} directly, the approximate solution will only be accurate when $p$ lies in a narrow range of $\frac{N+2}{N-2}$ (more precisely, in the interval $\frac{N+1}{N-2} \leq p < \frac{N+2}{N-2}$). As $p$ decreases, this approximate solution will become less accurate and the error term will become too large for the method to remain valid. We believe that this is a technical difficulty. Therefore, inspired by the method used in \cite{KP21} for handling systems (an approach also used in \cite{JK23} and more recently in \cite{GKPY}), we perform a projection to obtain a more accurate approximate solution $(Y_{\bar{r}, \bar{y}'', \lambda}, Z_{\bar{r}, \bar{y}'', \lambda})$. Theorem \ref{Thm1} demonstrates that with this approximate solution, we can find the solutions to system \eqref{Main} for a wider range $\frac{N}{N-2}< p < \frac{N+2}{N-2}$ when $N \geq 6$. It is worth mentioning that this paper deals only with the case $p > \frac{N}{N-2}$ since in this range the decay of $U_{0,1}(y)$ is uniform ($U_{0,1} (y) \approx 1/(1+|y|)^{N-2}$ when $|y| \to +\infty$ when $\frac{N}{N-2} < p \leq \frac{N+2}{N-2}$), which is convenient to analyze. When $p \leq \frac{N}{N-2}$, the decay of $U_{\lambda,x}(y)$ at infinity will depend on $p$, gradually changing from $\log |y|/(1+|y|)^{N-2}$ when $p = \frac{N}{N-2}$ to $1/(1+|y|)^{N-4}$ when $p =1$. We believe that our method can be extended in this range as well, which will be the subject of our future work.

    \vspace{0.5em}

    In this paper, we always assume that $m > 0$ is a large integer, $\lambda \in [L_0 m^{\frac{N-2}{N-4}},\, L_1 m^{\frac{N-2}{N-4}}]$ for some constants $L_1 > L_0 > 0$, and for a small constant $\vartheta > 0$, we let
    \begin{equation}\label{condition_ry}
        |(\bar{r}, \bar{y}'') - (r_0, y_0'')| \leq \vartheta,
    \end{equation}

    In order to prove Theorem \ref{Thm1}, we will show the following result.
    \begin{thm}\label{Thm2}
        Under the assumptions of Theorem \ref{Thm1}, there exists a positive integer $m_0 > 0$, such that for any integer $m \geq m_0$, \eqref{Main} has a solution $(u_m,v_m)$ of the form
        \[
            (u_m,v_m) = (Y_{\bar{r}_m, \bar{y}_m'', \lambda_m} + \phi_m,Z_{\bar{r}_m, \bar{y}_m'', \lambda_m} + \psi_m )\approx (\sum_{j=1}^m \xi U_{x_j, \lambda_m} + \phi_m,\sum_{j=1}^m \xi V_{x_j, \lambda_m} + \psi_m),
        \]
        where $(\phi_m,\psi_m) \in H_s$. Moreover, as $m \to +\infty$, $\lambda_m \in [L_0 m^{\frac{N-2}{N-4}}, L_1 m^{\frac{N-2}{N-4}}]$, $(\bar{r}_m, \bar{y}_m'') \to (r_0, y_0'')$, and $(\lambda_m^{-\frac{N}{q+1}} \| \phi_m \|_{L^\infty},\lambda_m^{-\frac{N}{p+1}} \| \psi_m \|_{L^\infty}) \to \bf{0}$.
    \end{thm}

    Now we outline the main idea of the proof of Theorem \ref{Thm2}. The corresponding functional for \eqref{Main} is defined by
    \[
        I(u,v) = \int_{\R^N} (\nabla u \cdot \nabla v + V(|y'|,y'')\, uv ) \, \mathrm{d}y - \frac{1}{p+1} \int_{\R^N} |v|^{p+1} \, \mathrm{d}y - \frac{1}{q+1} \int_{\R^N} |u|^{q+1} \, \mathrm{d}y,
    \]
    which is well-defined in $X$ since
    \begin{equation*}
        \begin{cases}
            \dot{W}^{2,\frac{p+1}{p}}(\R^N) \hookrightarrow \dot{W}^{1,p^*}(\R^N) \hookrightarrow L^{q+1}(\R^N),\\
            \dot{W}^{2,\frac{q+1}{q}}(\R^N) \hookrightarrow \dot{W}^{1,q^*}(\R^N) \hookrightarrow L^{p+1}(\R^N),
        \end{cases}
    \end{equation*}
    with
    \[
        \frac{1}{p^*} = \frac{p}{p+1} - \frac{1}{N} = \frac{1}{q+1} + \frac{1}{N}, \quad \frac{1}{q^*} = \frac{q}{q+1} - \frac{1}{N} = \frac{1}{p+1} + \frac{1}{N}.
    \]
    Based on a finite-dimensional reduction argument, finding a critical point of $I(u,v)$ of the form in Theorem \ref{Thm2} can be reduced to finding a critical point of
    \[
        F(\bar{r}_m, \bar{y}_m'', \lambda_m) = I\big((Y_{\bar{r}_m, \bar{y}_m'', \lambda_m} + \phi_m,\, Z_{\bar{r}_m, \bar{y}_m'', \lambda_m} + \psi_m)\big),
    \]
    where $\lambda_m \in [L_0 m^{\frac{N-2}{N-4}}, L_1 m^{\frac{N-2}{N-4}}]$ and $(\bar{r}_m, \bar{y}_m'')$ satisfy \eqref{condition_ry}. To handle the large number of bubbles in the solution, motivated by \cite{WY}, we carry out the reduction procedure in a space equipped with a weighted maximum norm instead of the standard Sobolev space. We then turn to solving the corresponding finite-dimensional problem. We avoid directly calculating the partial derivatives of $F(\bar{r}_m, \bar{y}_m'', \lambda_m)$, since it is difficult to obtain sharp estimates for the error terms. In fact, from \eqref{partial derivative_1} and \eqref{partial derivative_2}, one has
    \[
        \frac{\partial F}{\partial \bar{r}} = m \Bigg( \frac{B_1}{\lambda^2} \frac{\partial V(\bar{r},\bar{y}'')}{\partial \bar{r}} + \sum_{j=2}^m \frac{B_2}{\bar{r} \lambda^{N-1} |x_1 - x_j|^{N-2}} + O\Big(\frac{1}{\lambda^{1+\varepsilon}}\Big) \Bigg),
    \]
   and
    \[
        \frac{\partial F}{\partial \bar{y}_k''} = m \Bigg( \frac{B_1}{\lambda^2} \frac{\partial V(\bar{r},\bar{y}'')}{\partial \bar{y}_k''} + O\Big(\frac{1}{\lambda^{1+\varepsilon}}\Big) \Bigg), \quad k = 3, \ldots, N.
    \]
    It is shown that the error terms can dominate the main terms in the above expansions. Instead, inspired by \cite{PWY}, we employ local Pohozaev identities to determine the location of the bubbles. More details are provided in Section 3.

    \vskip8pt
    Our paper is organized as follows. In Section 2, we perform the finite-dimensional reduction. We define suitable weighted functional spaces, in which we study the linearized equation and prove the invertibility of this linearized operator on the orthogonal complete of its kernel. In Section 3, we study the reduced finite-dimensional problem and prove Theorem \ref{Thm2}. In the Appendix, we collect all the  essential technical estimates needed throughout the paper. This includes the asymptotic behavior and non-degeneracy of the $(U_{0,1},V_{0,1})$ (Appendix A), sharp estimates of the approximate solution $Y_{\bar{r}, \bar{y}'',\lambda}$ (Appendix B), the energy expansion (Appendix C),
    and various other integral and norm estimates (Appendix D).

    \medskip

    In the following of the  paper, we use $C$ to denote a positive constant which may vary from line to line. In addition, we sometimes use $\int_A f$ to denote $\int_{A}f(x) \ \dx$ if there is no confusion.

\section{Finite Dimensional Reduction}
Define
    $$\|u\|_{*,1}=\sup\limits_{y\in\R^N}\Big(\sum\limits_{j=1}^{m}\frac{1}{(1+\lambda|y-x_{j}|)^{\frac{N}{q+1}+\tau}}\Big)^{-1} \lambda^{-\frac{N}{q+1}}|u(y)|,$$
    $$\|v\|_{*,2}=\sup\limits_{y\in\R^N}\Big(\sum\limits_{j=1}^{m}\frac{1}{(1+\lambda|y-x_{j}|)^{\frac{N}{p+1}+\tau}}\Big)^{-1} \lambda^{-\frac{N}{p+1}}|v(y)|,$$
and
    $$\|u\|_{**,1}=\sup\limits_{y\in \R^N}\Big(\sum\limits_{j=1}^{m}\frac{1}{(1+\lambda|y-x_{j}|)^{\frac{N}{q+1}+2+\tau}}\Big)^{-1}\lambda^{-\frac{N}{q+1}-2}|u(y)|,$$
    $$\|v\|_{**,2}=\sup\limits_{y\in \R^N}\Big(\sum\limits_{j=1}^{m}\frac{1}{(1+\lambda|y-x_{j}|)^{\frac{N}{p+1}+2+\tau}}\Big)^{-1}\lambda^{-\frac{N}{q+1}-2}|v(y)|,$$
where we choose $\tau = \frac{N-4}{N-2}$. Define
\begin{equation*}
    \|(u,v)\|_* = \|u\|_{*,1}+\|v\|_{*,2}, \quad \|(u,v)\|_{**} = \|u\|_{**,1}+\|v\|_{**,2}.
\end{equation*}
Then we set two Banach spaces
\begin{equation}\label{Xs}
    X_s=:\{(u,v)\in C(\R^N)\times C(\R^N): ||(u,v)||_{*}<+\infty\}\cap L_s,
\end{equation}
and
\begin{equation}\label{Ys}
    Y_s=:\{(f,g)\in C(\R^N)\times C(\R^N): ||(f,g)||_{**}<+\infty\}\cap L_s.
\end{equation}
equipped with the norms $||\cdot||_{*}$ and $||\cdot||_{**}$ respectively.
Denote
\[
    Y_{x_j, \lambda}(y) = \xi U_{x_j, \lambda}(y),\quad Z_{x_j, \lambda}(y) = \xi V_{x_j, \lambda}(y),
\]
and
\[
    (Y_{j,l},Z_{j,l})= \left(\frac{\partial Y_{x_j,\lambda}}{\partial \Box_l},\frac{\partial Z_{x_j,\lambda}}{\partial \Box_l}\right),
    \quad
    (\bar{Y}_{j,l},\bar{Z}_{j,l})=\left((-\Delta+V)^{-1}\frac{\partial (Z_{x_j,\lambda}^p)}{\partial \Box_l},(-\Delta+V)^{-1}\frac{\partial (Y_{x_j,\lambda}^q)}{\partial \Box_l}\right),
\]
for $j=1,2,\cdots,m$, $l=1,2,\cdots,N$, where
\[
    \Box_l =\lambda, \ \ \text{if} \ \ l=1, \ \ \ \Box_l = \bar{r}, \ \ \text{if} \ \ l=2, \ \ \ \Box_l = \bar{y}_l'', \ \ \text{if}  \ \ l=3,\cdots,N.
\]
We will use this notation repeatedly in this article.

     Let $n_l = -1$ for $l=1$ and $n_l = 1$ for $l=2,3,\cdots,N$, then it holds that
    $$\frac{\partial U_{x_j,\lambda}}{\partial \Box_l}=O(\lambda^{n_l}U_{x_j,\lambda}),\quad \frac{\partial V_{x_j,\lambda}}{\partial \Box_l}=O(\lambda^{n_l}V_{x_j,\lambda}).$$
Define
\begin{equation}\label{E}
    E=:\left\{ (u,v)\in X_s:  \left\langle \left(\sum\limits_{j=1}^{m} q Y_{x_j,\lambda}^{q-1} Y_{j,l},  \sum\limits_{j=1}^{m} p Z_{x_j,\lambda}^{p-1} Z_{j,l} \right), (u,v)  \right\rangle = 0, \;\; l=1,2,\ldots,N\right\}
\end{equation}
and
\begin{equation}\label{F}
    F=:\left\{(f,g)\in Y_s: \left\langle \left( \sum\limits_{j=1}^{m}\bar{Z}_{j,l}, \sum\limits_{j=1}^{m}\bar{Y}_{j,l}\right), (f,g)  \right\rangle = 0, \;\; l=1,2,\ldots,N\right\}
\end{equation}
where we set
$$\langle (\phi_1,\phi_2),(\psi_1,\psi_2) \rangle:= \langle \phi_1, \psi_1 \rangle + \langle \phi_2, \psi_2 \rangle: = \int_{\R^N} \phi_1\psi_1 + \phi_2 \psi_2.$$

We will perform the reduction procedure in $E$. For this purpose, we need to find a solution of the form $(Y_{\bar{r}, \bar{y}'', \lambda} + \phi, Z_{\bar{r}, \bar{y}'', \lambda} +\psi )$ for \eqref{Main} such that $(\phi,\psi)\in E$. By direct computation, we can verify that $(\phi, \psi)$ satisfies the following equation:
\begin{equation}\label{op-1}
    L(\phi,\psi) =l +N(\phi,\psi),
\end{equation}
where
\begin{equation}\label{op-2}
    \begin{split}
        L(\phi,\psi) =& \Big(L_1 (\phi,\psi), L_2 (\phi, \psi)\Big) \\
        =&\Big(   -\Delta \phi + V\phi -p(Z_{\bar{r}, \bar{y}'', \lambda})^{p-1} \psi,  -\Delta \psi + V\psi -q(Y_{\bar{r}, \bar{y}'', \lambda})^{q-1} \phi \Big),
    \end{split}
\end{equation}
\begin{equation}\label{op-3}
    \begin{split}
        l= \big(l_1,l_2\big)= \big( & ( Z_{\bar{r}, \bar{y}'', \lambda}^p - \xi (Z_{\bar{r}, \bar{y}'', \lambda}^*)^p) - VY_{\bar{r}, \bar{y}'', \lambda} + Y_{\bar{r}, \bar{y}'', \lambda}^* \Delta \xi + 2\nabla \xi \nabla Y_{\bar{r}, \bar{y}'', \lambda}^*, \\
        & (Y_{\bar{r}, \bar{y}'', \lambda}^q - \xi \sum\limits_{j=1}^m U_{x_j,\lambda}^q) - VZ_{\bar{r}, \bar{y}'', \lambda} + Z_{\bar{r}, \bar{y}'', \lambda}^* \Delta \xi + 2\nabla \xi \nabla Z_{\bar{r}, \bar{y}'', \lambda}^*  \big)
    \end{split}
\end{equation}
and
\begin{equation}\label{op-4}
    N(\phi,\psi)= \Big(N_1(\psi),N_2(\phi)\Big) ,
\end{equation}
with
\begin{equation*}
    \begin{split}
        N_1(\psi) = & |Z_{\bar{r}, \bar{y}'', \lambda} + \psi|^{p-1}(Z_{\bar{r}, \bar{y}'', \lambda} + \psi) - Z_{\bar{r}, \bar{y}'', \lambda}^p - pZ_{\bar{r}, \bar{y}'', \lambda}^{p-1}\psi  , \\
        N_2(\phi) = & | Y_{\bar{r}, \bar{y}'', \lambda} + \phi|^{q-1}(Y_{\bar{r}, \bar{y}'', \lambda} + \phi) - Y_{\bar{r}, \bar{y}'', \lambda}^q - q Y_{\bar{r}, \bar{y}'', \lambda}^{q-1}\phi  .
    \end{split}
\end{equation*}
This motivates us to consider the following problem:
\begin{equation}\label{linear-equation}
   \begin{cases}
      L(\phi,\psi) = (h_1,h_2) + \sum\limits_{l=1}^{N} c_l \left( \sum\limits_{j=1}^{m} pZ_{x_j,\lambda}^{p-1} Z_{j,l}, \sum\limits_{j=1}^{m} qY_{x_j,\lambda}^{q-1} Y_{j,l} \right)\;\;\; \hbox{in} \;\;\; \R^N,\\
      (\phi, \psi) \in E.
   \end{cases}
\end{equation}

\begin{lemma}\label{blowup}
    Suppose $(\phi_m, \psi_m)$ solves \eqref{linear-equation} with $(h_1,h_2) = h_m:=(h_{m,1},h_{m,2})$. If $||h_m||_{**} \rightarrow{ 0}$, then $||(\phi_m, \psi_m)||_{*} \rightarrow 0$.
\end{lemma}
\begin{proof}
      Suppose in contrast that $(\phi_{m}, \psi_{m})$ solves equation \eqref{linear-equation} with
      \[
          (h_1,h_2)=h_m  = (h_{m,1}, h_{m,2}), \ \ \bar{r} = \bar{r}_m, \ \ \lambda = \lambda_m \ \ \hbox{ and } \ \ \bar{y}'' = \bar{y}_m'',
      \]
      and
      \[
          ||h_m||_{**} \rightarrow 0, \ \ \bar{r}_m \rightarrow r_0 \ \  \hbox{ and } \ \ \bar{y}_m'' \rightarrow \bar{y}_0'',
      \]
      and
      \[
          \lambda_m \in [L_0 m^{\frac{N-2}{N-4}}, L_1 m^{\frac{N-2}{N-4}}], \ \  ||(\phi_{m}, \psi_{m})||_{*} \geq c > 0  \ \ \hbox{ as } \ \ k \rightarrow \infty.
      \]
        Without loss of generality, we may assume $||(\phi_{m}, \psi_{m})||_{*} = 1 $. For simplicity, we drop the subscript $m$ and write $h=(h_1,h_2)$.

  According to the definition of $L(\phi,\psi)$, we have
 \begin{equation}\label{omega1}
       -\Delta \phi + V(|y'|,y'')\phi - p(Z_{\bar{r},\bar{y}'',\lambda})^{p-1}\psi = h_1 + \sum\limits_{l=1}^{N}c_l \sum\limits_{j=1}^m p Z_{x_j,\lambda}^{p-1}Z_{j,l},
 \end{equation}
 and
\begin{equation}\label{omega2}
       -\Delta \psi + V(|y'|,y'')\psi - q(Y_{\bar{r},\bar{y}'',\lambda})^{q-1}\phi = h_2 + \sum\limits_{l=1}^{N}c_l \sum\limits_{j=1}^m q Y_{x_j,\lambda}^{q-1}Y_{j,l}.
 \end{equation}
 By Green's representation formula for $-\Delta+V$, we have
    \begin{align}\label{esti-omega1}
          |\phi(y)| \leq C \int_{\R^N} \dfrac{1}{|y-z|^{N-2}} \left( Z_{\bar{r},\bar{y}'',\lambda}^{p-1}(z)|\psi(z)| + |h_{1}(z)| + \left| \sum\limits_{l=1}^{N}c_l \sum\limits_{j=1}^m p Z_{x_j,\lambda}^{p-1}(z)Z_{j,l}(z) \right|\right) \ \dz,
    \end{align}
    and
    \begin{align}
          |\psi(y)| \leq C \int_{\R^N} \dfrac{1}{|y-z|^{N-2}} \left(   |Y_{\bar{r},\bar{y}'',\lambda}^{q-1}(z)|| \phi(z) | + |h_2(z) | + \left| \sum\limits_{l=1}^{N}c_l \sum\limits_{j=1}^m q Y_{x_j,\lambda}^{q-1}(z)Y_{j,l}(z) \right|  \right) \ \dz.
    \end{align}
    Next, we will prove that the following function
    \begin{equation}
        \left[\sum_{j=1}^{m}\frac{\lambda^{\frac{N}{q+1}}}{(1+\lambda |y-x_j|)^{\frac{N}{q+1}+\tau}}\right]^{-1}|\phi(y)| +
        \left[\sum_{j=1}^{m}\frac{\lambda^{\frac{N}{p+1}}}{(1+\lambda |y-x_j|)^{\frac{N}{p+1}+\tau}}\right]^{-1}|\psi(y)|
    \end{equation}
    can only achieve its maximum in $\cup_{i=1}^m \overline{B(x_i,M\lambda^{-1})}$ for some large constant $M>1$. For this purpose, we will estimate the above function when $y \in \R^N\backslash \cup_{i=1}^m {B(x_i,M\lambda^{-1})}$.

\textbf{Estimate $\phi$ and $\psi$.} To begin with, we split the first integral on the right hand side of (\ref{esti-omega1}) into several parts
\begin{align*}
    &\int_{\R^N} \dfrac{1}{|y-z|^{N-2}}Z_{\bar{r},\bar{y}'',\lambda}^{p-1}(z)|\psi(z)| dz\\
    &\leq C||\psi||_{*,2}\left(\int_{\R^N\backslash \cup_{i=1}^m {B(x_i,M\lambda^{-1})}}+\sum_{k=1}^m \int_{B(x_k,M\lambda^{-1})} \right)\frac{\lambda^{\frac{pN}{p+1}}}{|y-z|^{N-2}}\\ &\quad\times\left[\sum_{j=1}^{m}\frac{1}{(1+\lambda|z-x_i|)^{N-2}}\right]^{p-1} \sum_{j=1}^{m}\frac{\dz}{(1+\lambda|z-x_j|)^{\frac{N}{p+1}+\tau}}
    :=I_0+\sum_{k=1}^mI_k.
\end{align*}
Note that $N-2>\frac{N}{p+1}+\tau$ when $p> \frac{N}{N-2}$. Therefore, for some small constant $\sigma>0$, we have
    \begin{align}
        I_0
        &\leq \frac {C||\psi||_{*,2}}{M^{\sigma}}\int_{\R^N\backslash \cup_{i=1}^m {B(x_i,M\lambda^{-1})}} \frac{\lambda^{\frac{pN}{p+1}}}{|y-z|^{N-2}} \left[\sum_{j=1}^{m}\frac{1}{(1+\lambda|z-x_j|)^{\frac{N}{p+1}+\tau}}\right]^p \ \dz \nonumber\\
        &\leq \frac {C||\psi||_{*,2}}{M^{\sigma}}\int_{\R^N\backslash \cup_{i=1}^m {B(x_i,M\lambda^{-1})}}\frac{\lambda^{\frac{pN}{p+1}}}{|y-z|^{N-2}}\sum_{j=1}^{m}\frac{1}{(1+\lambda|z-x_j|)^{\frac{pN}{p+1}+\tau}}\ \dz \label{I0}\\
        &\leq \frac {C||\psi||_{*,2}}{M^{\sigma}}\sum_{j=1}^{m}\frac{\lambda^{\frac{N}{q+1}}}{(1+\lambda|y-x_j|)^{\frac{N}{q+1}+\tau}}, \nonumber
    \end{align}
 where the second inequality follows from Lemma \ref{B1_1} and the last one follows from Lemma \ref{B3} and $\frac{pN}{p+1}=2+\frac{N}{q+1}$.

Next, we estimate $I_k$ for $k=1,2,\cdots,m$. Note that for any $\alpha\geq \tau$ and $z \in {B(x_k,M\lambda^{-1})}$, it holds that
\begin{equation}\label{Ik-0}
    \sum_{j\neq k}\frac{1}{(1+\lambda|z-x_j|)^{\alpha}}\leq C\leq \frac{C}{(1+\lambda|z-x_k|)^{\alpha}}.
\end{equation}
Therefore, we can derive from Lemma \ref{B3} and \eqref{Ik-0} that, for $y\in \R^N\backslash \cup_{i=1}^m {B(x_i,M\lambda^{-1})}$,
\begin{equation}\label{Ik}
    \begin{split}
        I_k&\leq C||\psi||_{*,2}\int_{ {B(x_k,M\lambda^{-1})}} \frac{1}{|y-z|^{N-2}}\frac{\lambda^{\frac{pN}{p+1}}}{(1+\lambda|z-x_k|)^{(N-2)(p-1)+\frac{N}{p+1}+\tau}}\ \dz\\
        &\leq C||\psi||_{*,2} \frac{\lambda^{\frac{N}{q+1}}}{(1+\lambda|y-x_k|)^{\min\{(N-2)(p-1)+\frac{N}{p+1}+\tau-2,N-2-\theta\}}}\\
        &\leq \frac {C||\psi||_{*,2}}{M^{\sigma}}\frac{\lambda^{\frac{N}{q+1}}}{(1+\lambda|y-x_k|)^{\frac{N}{q+1}+\tau}}.
    \end{split}
\end{equation}
Combine \eqref{I0} and \eqref{Ik}, it holds that for any $y\in \R^N\backslash \cup_{i=1}^m {B(x_i,M\lambda^{-1})}$,
\begin{align}\label{omega1-1}
    \int_{\R^N} \dfrac{1}{|y-z|^{N-2}}Z_{\bar{r},\bar{y}'',\lambda}^{p-1}(z)|\psi(z)| \ \dz\leq \frac {C||\psi||_{*,2}}{M^{\sigma}}\sum_{j=1}^{m}\frac{\lambda^{\frac{N}{q+1}}}{(1+\lambda|y-x_j|)^{\frac{N}{q+1}+\tau}}.
\end{align}
Similarly, we can employ Lemma \ref{lemma:U} to obtain for any $y\in \R^N\backslash \cup_{i=1}^m {B(x_i,M\lambda^{-1})}$,
\begin{align}\label{omega2-1}
    \int_{\R^N} \dfrac{1}{|y-z|^{N-2}} | Y_{\bar{r},\bar{y}'',\lambda}^{q-1}(z)| |\phi(z) | \ \dz\leq \frac {C||\phi||_{*,2}}{M^{\sigma}}\sum_{j=1}^{m}\frac{\lambda^{\frac{N}{p+1}}}{(1+\lambda|y-x_j|)^{\frac{N}{p+1}+\tau}}.
\end{align}

Next, it follows from Lemma \ref{B3} and the fact $\frac{N}{q+1}+\tau < N-2$, that
\begin{equation}\label{omega1-2}
   \begin{split}
     \int_{\R^N} \dfrac{1}{|z-y|^{N-2}} |h_1(z)|\  \dz
       &\leq ||h_1||_{**,1} \int_{\R^N} \dfrac{1}{|z-y|^{N-2}} \sum\limits_{j=1}^{m}\dfrac{\lambda^{\frac{N}{q+1}+2}}{(1+\lambda|z-x_j|)^{\frac{N}{q+1}+2+\tau}}\ \dz \\
      &\leq ||h_1||_{**,1} \sum\limits_{j=1}^{m}\dfrac{\lambda^{\frac{N}{q+1}}}{(1+\lambda|y-x_j|)^{\frac{N}{q+1}+\tau}}.
   \end{split}
\end{equation}

Next we consider the last term on the right side of (\ref{esti-omega1}). By the definitions of $Z_{x_j,\lambda}$ and $Z_{j,l}$, we have
\[
    \begin{split}
        |Z_{x_j,\lambda}^{p-1}Z_{j,l}| & \leq \left|V_{x_{j}, \lambda}^{p-1} Z_{j,l} \right| \leq C\lambda^{n_l} |V_{x_j,\lambda}|^p
        \leq C\dfrac{\lambda^{n_l+ \frac{pN}{p+1}}}{(1+\lambda|z-x_j|)^{(N- 2)p}},
    \end{split}
\]
where $n_l = -1$ for $l=1$ and $n_l = 1$ for $j=2,3,\cdots,N$. Thus, it follows that
   \begin{align}
       \int_{\R^N} \dfrac{1}{|y-z|^{N-2}} \big| \sum\limits_{j=1}^m Z_{x_j,\lambda}^{p-1}Z_{j,l} \big| \ \dz
      &\leq C \int_{\R^N} \dfrac{1}{|y-z|^{N-2}} \sum\limits_{j=1}^{m}  \dfrac{\lambda^{n_l+ \frac{pN}{p+1}}}{(1+\lambda|z-x_j|)^{(N-2)p}}\ \dz  \label{omega1-3} \\
       &\leq C \sum\limits_{j=1}^{m} \dfrac{\lambda^{n_l + \frac{N}{q+1}}}{(1+ \lambda|y-x_j|)^{N-2}}
      \leq C \sum\limits_{j=1}^{m} \dfrac{\lambda^{n_l + \frac{N}{q+1}}}{(1+\lambda|y-x_j|)^{\frac{N}{q+1}+\tau}}. \nonumber
\end{align}
In addition, we have similar results for $\psi$:
\begin{equation}\label{omega2-2}
    \int_{\R^N} \dfrac{1}{|y-z|^{N-2}} |h_{2}(z)| \ \dz \leq C||h_2||_{**,2}\sum\limits_{j=1}^{m}\dfrac{\lambda^{\frac{N}{p+1}}}{(1+\lambda|y-x_j|)^{\frac{N}{p+1} + \tau}} ,
\end{equation}
and
\begin{equation}\label{omega2-3}
    \int_{\R^N} \dfrac{1}{|y-z|^{N-2}} \left| \sum\limits_{j=1}^m q Y_{x_j,\lambda}^{q-1}Y_{j,l} \right|  \ \dz \leq C \sum\limits_{j=1}^{m}\dfrac{\lambda^{\frac{N}{p+1} + n_l}}{(1+\lambda|y-x_j|)^{\frac{N}{p+1} + \tau}} .
\end{equation}

\textbf{Estimate $c_l$ for $l=1,2,\cdots,N$.} Multiplying $Z_{1,h}$ on both sides of equation (\ref{omega1}) and integrating over $\R^N$, we have
\begin{equation}\label{c1}
       \sum\limits_{l=1}^{N}c_l \sum\limits_{j=1}^m p\left\langle Z_{x_j,\lambda}^{p-1}Z_{j,l}   , Z_{1,h}  \right\rangle = \left\langle -\Delta \phi+ V(|y'|,y'') \phi- p(Z_{\bar{r}, \bar{y}'',\lambda})^{p-1}\psi , Z_{1,h} \right\rangle - \left\langle h_1 , Z_{1,h} \right\rangle.
\end{equation}
Multiplying $Y_{1,h}$ on both sides of equation (\ref{omega2}) and integrating over $\R^N$, we have
\begin{equation}\label{c2}
     \sum\limits_{l=1}^{N}c_l \sum\limits_{j=1}^m q \left\langle Y_{x_j,\lambda}^{q-1}Y_{j,l}, Y_{1,h}  \right\rangle = \left\langle -\Delta \psi+ V(|y'|,y'') \psi - q(Y_{\bar{r}, \bar{y}'',\lambda})^{q-1}\phi , Y_{1,h} \right\rangle - \left\langle h_2 , Y_{1,h} \right\rangle.
\end{equation}
By adding equations (\ref{c1}) and (\ref{c2}), we obtain
   \begin{align}
        & \quad \sum\limits_{l=1}^{N}c_l \sum\limits_{j=1}^m \left( p\left\langle Z_{x_j,\lambda}^{p-1}Z_{j,l}   , Z_{1,h}   \right\rangle +  q\left\langle Y_{x_j,\lambda}^{q-1}Y_{j,l}, Y_{1,h} \right\rangle \right) \nonumber  \\
        &= -\left\langle h_1 , Z_{1,h} \right\rangle - \left\langle h_2 , Y_{1,h} \right\rangle + \left\langle -\Delta Z_{1,h} +  V(|y'|,y'') Z_{1,h} - q (Y_{\bar{r}, \bar{y}'',\lambda})^{q-1}Y_{1,h}, \phi \right\rangle \label{c3} \\
        & \quad + \left\langle -\Delta Y_{1,h} +  V(|y'|,y'') Y_{1,h} - p (Z_{\bar{r}, \bar{y}'',\lambda})^{p-1}Z_{1,h}, \psi \right\rangle . \nonumber
   \end{align}
We estimate each term on both sides of \eqref{c3}.

First, using Lemma \ref{B2}, we have
\begin{align}
   | \left\langle h_1 , Z_{1,h} \right\rangle |
   & \leq C ||h_{1}||_{**,1} \int_{\R^N} \dfrac{\lambda^{n_h+ \frac{N}{p+1}}}{(1+\lambda|y-x_1|)^{N-2}}\sum\limits_{j=1}^{m} \dfrac{\lambda^{\frac{N}{q+1}+2}}{(1+ \lambda|y-x_j|)^{\frac{N}{q+1}+ 2+\tau}} \ \dy \nonumber \\
   & \leq C \lambda^{n_h} ||h_1||_{**,1} \int_{\R^N} \dfrac{\lambda^N}{(1+\lambda|y-x_1|)^{N + \frac{N}{q+1} + \tau}} \ \dy  \label{ch1} \\
   & \quad + C \lambda^{n_h} ||h_1||_{**,1}\sum\limits_{j=2}^m \dfrac{1}{(\lambda |x_j - x_1|)^{\tau}}\int_{\R^N}  \dfrac{\lambda^{N}}{(1+ \lambda|y-x_j|)^{N+\frac{N}{q+1} }}  \ \dy \nonumber \\
   & \leq C \lambda^{n_h} ||h_1||_{**,1} \left(1 + \sum\limits_{j=2}^m \dfrac{1}{(\lambda |x_j - x_1|)^{\tau}} \right) \leq C\lambda^{n_h} ||h_1||_{**,1}.\nonumber
\end{align}
Similarly, we have
\begin{equation}\label{ch2}
    | \left\langle h_2 , Y_{1,h} \right\rangle | \leq C\lambda^{n_h} ||h_2||_{**,2}.
\end{equation}
Next, we estimate the last two terms on the right hand side of \eqref{c3}. Since $$\min\left\{\frac{N}{p+1}+N-2,\frac{N}{q+1}+N-2 \right\}=\frac{N}{q+1}+N-2>N-1,$$ we can use Lemma \ref{B1} to obtain
\begin{align}
     &\quad| \left\langle V(|y'|, y'') \psi , Y_{1,h} \right\rangle |   \leq \int_{\R^N} V(|y'|, y'') |\psi(y)| |Y_{1,h}(y)| \ \dy \nonumber \\
     & \leq C ||\psi||_{*,2} \int_{\R^N} \dfrac{\lambda^{n_h+ \frac{N}{q+1}}}{(1+\lambda|y-x_1|)^{N-2}}\sum\limits_{j=1}^{m} \dfrac{\lambda^{\frac{N}{p+1}}\xi(y)}{(1+ \lambda|y-x_j|)^{\frac{N}{p+1}+\tau}} \ \dy  \nonumber \\
     & \leq C \lambda^{n_h} ||\psi||_{*,2} \int_{\R^N} \dfrac{\lambda^{N-2}\xi(y)}{(1+\lambda|y-x_1|)^{N -2 + \frac{N}{p+1} + \tau}} \ \dy \label{cV1} \\
     &\quad + C \lambda^{n_h} ||\psi||_{*,2} \sum\limits_{j=2}^m \dfrac{1}{(\lambda |x_j - x_1|)^{\tau}}\int_{\R^N}  \dfrac{\lambda^{N-2}\xi(y)}{(1+ \lambda|y-x_j|)^{N-2+\frac{N}{p+1} }}  \ \dy \nonumber \\
     & \leq C \lambda^{n_h-1 - \varepsilon} ||\psi||_{*,2} \int_{\R^N} \dfrac{\lambda^{N}\xi(y)}{(1+\lambda|y-x_1|)^{N -2 + \frac{N}{p+1} +1 - \varepsilon}} \leq \dfrac{C\lambda^{n_h} ||\psi||_{*,2}}{\lambda^{1+\varepsilon}}.\nonumber
\end{align}
Here, we have used the fact that
$$\frac{1}{\lambda}\leq \frac{C}{1+\lambda|y-x_j|}, \ \ \text{if} \ \ y\in \text{supp} (\xi).$$
Similarly, we have
\begin{equation}\label{cV2}
   | \left\langle V(|y'|, y'') \psi , Y_{1,h} \right\rangle | \leq \dfrac{C\lambda^{n_h} ||\phi||_{*,1}}{\lambda^{1+\varepsilon}}.
\end{equation}
On the other hand, by Lemma \ref{B7}, we have
    \begin{align}\label{estimate_1}
        \left| \left\langle -\Delta Z_{1,h}  - q (Y_{\bar{r}, \bar{y}'',\lambda})^{q-1}Y_{1,h}, \phi \right\rangle\right|=O\left(\frac{\lambda^{n_h}||(\phi,\psi)||_{*}}{\lambda^{1+\varepsilon}} \right),
    \end{align}
and
    \begin{align}\label{estimate_2}
    \left|\left\langle -\Delta Y_{1,h}  - p (Z_{\bar{r}, \bar{y}'',\lambda})^{p-1}Z_{1,h}, \psi \right\rangle \right|=O\left( \frac{\lambda^{n_h}||(\phi,\psi)||_{*}}{\lambda^{\varepsilon}} \right).
    \end{align}
Combine \eqref{ch1} \eqref{ch2} \eqref{cV1} \eqref{cV2} \eqref{estimate_1} and \eqref{estimate_2}, we have
\begin{equation}\label{RHS}
    R.H.S. \ of \ \eqref{c3} = \lambda^{n_h}(o(||(\phi,\psi)||_{*})+ O(||(h_1,h_2)||_{**})).
\end{equation}
Besides, we have
\begin{equation}\label{LHS}
    \sum\limits_{j=1}^m \left( p\left\langle Z_{x_j,\lambda}^{p-1}Z_{j,l}   , Z_{1,h}   \right\rangle +  q\left\langle Y_{x_j,\lambda}^{q-1}Y_{j,l}, Y_{1,h} \right\rangle \right) = \delta_{hl}\lambda^{2n_h}a_h+o(\lambda^{n_l+n_h-2}),
\end{equation}
for some $a_h>0$. Inserting \eqref{RHS} \eqref{LHS} into \eqref{c3}, we have
\begin{equation}\label{c}
    c_h = \dfrac{1}{\lambda^{n_h}}(o(||(\phi,\psi)||_{*})+ O(||(h_1,h_2)||_{**})), \;\; h=1,2,\cdots,N.
\end{equation}

\textbf{Derive a contradiction.} Combine \eqref{omega1}, \eqref{omega2}, \eqref{omega1-1}, \eqref{omega2-1}, \eqref{omega1-2}, \eqref{omega1-3}, \eqref{omega2-2}, \eqref{omega2-3} and \eqref{c}, we get that for $y \in \R^N \setminus \cup_{i=1}^m B_{M\lambda^{-1}}(x_i)$,
\begin{equation}\label{21-1-6-21}
    \left[\sum_{j=1}^{m}\frac{\lambda^{\frac{N}{q+1}}}{(1+\lambda |y-x_j|)^{\frac{N}{q+1}+\tau}}\right]^{-1}|\phi_m(y)| \le C\left(\frac{1}{M^{\sigma}}+o(1)\right)\|(\phi_m, \psi_m)\|_{*} + C\|(h_{1,m}, h_{2,m})\|_{**}.
\end{equation}
and
\begin{equation}\label{n21-1-6-21}
    \left[\sum_{j=1}^{m}\frac{\lambda^{\frac{N}{p+1}}}{(1+\lambda |y-x_j|)^{\frac{N}{p+1}+\tau}}\right]^{-1}|\psi_m(y)| \le C\left( \frac{1}{M^{\sigma}}+o(1)\right)\|(\phi_m, \psi_m)\|_{*} + C\|(h_{1,m}, h_{2,m})\|_{**}.
\end{equation}
Since $||(\phi_m,\psi_m)||_{*} = 1$, we deduce that there is a $R>0$ and $ c_0 > 0$ such that
\[
    ||\lambda^{-\frac{N}{q+1}}\phi_m||_{L^\infty(B_{R/\lambda}(x_i))} + ||\lambda^{-\frac{N}{p+1}}\psi_m||_{L^\infty(B_{R/\lambda}(x_i))} \geq c_0 > 0,
\]
for some $i$. Define $(\bar{\phi}_m(y), \bar{\psi}_m(y))=(\lambda^{-\frac{N}{q+1}}\phi_m(\lambda^{-1}y+x_i), \lambda^{-\frac{N}{p+1}}\psi_m(\lambda^{-1}y+x_i))$, then $(\bar{\phi}_m, \bar{\psi}_m)$ converges uniformly in any compact set of $\R^{N}$ to a solution $(\Phi, \Psi)$ of
\begin{equation}\label{linear}
   \begin{cases}
   -\Delta \Phi = p V_{0,\Lambda}^{p-1} \Psi \;\;\;  \hbox{in } \R^N,\\
   -\Delta \Psi = q U_{0,\Lambda}^{q-1} \Phi  \;\;\;  \hbox{in } \R^N,
   \end{cases}
\end{equation}
for some $\Lambda \in [L_1,L_2]$. However, since $(\phi_m, \psi_m) \in E$, $(\Phi, \Psi)$ is perpendicular to the kernel of equation (\ref{linear}). So $\Phi=0, \Psi=0$. This is a contradiction.
\end{proof}

    Next, we consider the existence of solution for system \eqref{linear-equation}. For this purposes, we use a different approach. For any $(\phi,\psi) \in E$, define $(-\Delta+V)^{-1}(\phi,\psi)$ as
    $$
        (-\Delta+V)^{-1}(\phi,\psi)=\left(\int_{\R^N}G(x-y)\phi(y)\dy,\int_{\R^N}G(x-y)\psi(y)\dy \right),
    $$
    where $G$ is the Green function of the operator $-\Delta+V$ in $\R^N$. We can derive from \eqref{omega1-2} and the definition of $(\bar{Y}_{j,l},\bar{Z}_{j,l})$ that $(-\Delta+V)^{-1}$ is a bounded linear operator from $F$ to $E$, where $E$ and $F$ are defined in \eqref{E} and \eqref{F}.

    The following lemma is essential to the reduction argument.
    \begin{lemma}\label{prop1}
    Suppose $Y_s$ is defined as \eqref{Ys} and $F$ is defined as \eqref{F}. Then $F$ and
    \[
        \overline{F}:=\text{span}\left\{ \left( \sum\limits_{j=1}^{m} pZ_{x_j,\lambda}^{p-1} Z_{j,l}, \sum\limits_{j=1}^{m} qY_{x_j,\lambda}^{q-1} Y_{j,l} \right),\ l=1,2\cdots,N\right\}
    \]
    are topological complements to each other, and $Y_s=F\oplus \bar{F}$. Moreover, define the projection \textbf{P} from $Y_s$ to $F$ as follows:
    $$
        \textbf{P}(f,g)= \left(f-\sum\limits_{l=1}^{N} c_l\sum\limits_{j=1}^{m} pZ_{x_j,\lambda}^{p-1} Z_{j,l},g-\sum\limits_{l=1}^{N} c_l \sum\limits_{j=1}^{m} qY_{x_j,\lambda}^{q-1} Y_{j,l} \right),\,\, (f,g)\in Y_s
    $$
    where $\{c_l\}$ are chosen so that $\textbf{P}(f,g)\in F$.
    Then \textbf{P} is a linear bounded operator from $Y_s$ to $F$.
    \end{lemma}

    \begin{proof}
    It is sufficient to show that for any $(f,g)\in Y_s$, there exist $c_1, c_2, \cdots, c_N$ and a unique pair $(f_0,g_0)\in F$, such that
    $$(f,g)=(f_0,g_0)+\sum\limits_{l=1}^{N}c_l\left(\sum\limits_{j=1}^{m} pZ_{x_j,\lambda}^{p-1} Z_{j,l} , \sum\limits_{j=1}^{m} qY_{x_j,\lambda}^{q-1} Y_{j,l}\right),$$
    which is equivalent to solving the following equations involved $c_l$:
    \begin{equation}\label{c_l_0}
        \begin{gathered}
            \sum\limits_{l=1}^{N}c_l \left\langle \left(\sum\limits_{j=1}^{m} q Y_{x_j,\lambda}^{q-1} Y_{j,l},  \sum\limits_{j=1}^{m} p Z_{x_j,\lambda}^{p-1} Z_{j,l} \right), \left(\sum\limits_{j=1}^{m}\bar{Y}_{j,h},\sum\limits_{j=1}^{m}\bar{Z}_{j,h} \right)  \right\rangle= \left\langle \left(\sum\limits_{j=1}^{m}  \bar{Z}_{j,h},  \sum\limits_{j=1}^{m} \bar{Y}_{j,h} \right), (f,g) \right\rangle,
            \\  h=1,2,\cdots,N.
        \end{gathered}
    \end{equation}
     \textbf{Estimate the left hand side of \eqref{c_l_0}}. Since
    \begin{align*}
            -\Delta \frac{\partial ( \xi U_{x_j,\lambda})}{\partial \Box_h}=\frac{\partial }{\partial \Box_h}\left(  - U_{x_j, \lambda}\Delta \xi - 2 \nabla \xi \nabla U_{x_j, \lambda} + \xi V_{x_j,\lambda}^p \right), \\
           -\Delta \frac{\partial (\xi V_{x_j,\lambda})}{\partial \Box_h}= \frac{\partial }{\partial \Box_h}\left(  - V_{x_j, \lambda}\Delta \xi - 2 \nabla \xi \nabla V_{x_j, \lambda} + \xi U_{x_j,\lambda}^p \right),
    \end{align*}
    where $\Box_h$ denotes $\lambda$ if $h=1$, $\bar{r}$ if $h=2$ and $\bar{y}_ k''$ if $h\geq3$, we have
    \[
        (-\Delta+V) \left(\bar{Y}_{j,h}-\frac{\partial (\xi U_{x_j,\lambda})}{\partial \Box_h} \right)=\frac{\partial}{\partial \Box_h}\left( \xi^p V_{x_j,\lambda}^p - \xi V_{x_j,\lambda}^p + U_{x_j,\lambda}\Delta \xi + 2 \nabla \xi \nabla U_{x_j, \lambda} \right) - V \xi \frac{\partial U_{x_j,\lambda}}{\partial \Box_h},
    \]
    and
    \[
        (-\Delta+V) \left( \bar{Z}_{j,h}-\frac{\partial (\xi V_{x_j,\lambda})}{\partial \Box_h} \right)=\frac{\partial}{\partial \Box_h}(\xi^q U_{x_j,\lambda}^q - \xi U_{x_j,\lambda}^q + V_{x_j,\lambda}\Delta \xi + 2 \nabla \xi \nabla V_{x_j, \lambda}) -V\xi\frac{\partial V_{x_j,\lambda}}{\partial \Box_h}.
    \]
    Denote $A := \{ y = (y',y''): \delta < |(y',y'') - (r_0,y_0'')| <2\delta \}$ and $\textbf{B}_{\delta}:=\{ y = (y',y''): |(y',y'') - (r_0,y_0'')| <\delta \}$
    Then, we have
    \begin{equation*}
        \begin{aligned}
            \left|\bar{Y}_{j,h}(y)-\frac{\partial (\xi  U_{x_j,\lambda})}{\partial \Box_h}(y)\right|&\leq
            \int_{\R^N}\frac{C\lambda^{n_h}}{|y-z|^{N-2}}\left(  \frac{\lambda^{\frac{pN}{p+1}}|\xi^p-\xi|}{(1+\lambda|z-x_j|)^{p(N-2)}}+\frac{\chi_{\textbf{B}_{2\delta}}(z)\lambda^{\frac{N}{q+1}}}{(1+\lambda|z-x_j|)^{N-2}} \right)\dz,
        \end{aligned}
    \end{equation*}
    and
    \begin{equation*}
        \begin{aligned}
            \left|\bar{Z}_{j,h}(y)-\frac{\partial (\xi  V_{x_j,\lambda})}{\partial \Box_h}(y)\right|&\leq
            \int_{\R^N}\frac{C\lambda^{n_h}}{|y-z|^{N-2}}\left(  \frac{\lambda^{\frac{qN}{q+1}}|\xi^q-\xi|}{(1+\lambda|z-x_j|)^{q(N-2)}}+\frac{\chi_{\textbf{B}_{2\delta}}(z)\lambda^{\frac{N}{p+1}}}{(1+\lambda|z-x_j|)^{N-2}} \right)\dz.
        \end{aligned}
    \end{equation*}
    If $y\in \textbf{B}_{\delta/2}$, we have $|y - z| \approx |y-x_j| \approx 1+|y-x_j| $, therefore
    \begin{equation*}
        \begin{aligned}
            \int_{\R^N}\frac{C\lambda^{n_h}}{|y-z|^{N-2}}\frac{\lambda^{\frac{pN}{p+1}}|\xi^p-\xi|}{(1+\lambda|z-x_j|)^{p(N-2)}}\dz&\leq \frac{C\lambda^{n_h+ \frac{pN}{p+1}}}{\lambda^{N}(1+|y-x_j|)^{N-2}} \int_{\lambda (A + \{x_j\})}\frac{\dz}{(1+|z|)^{p(N-2)}}\\
            & \leq \frac{C\lambda^{n_h}}{\lambda^{\frac{N}{p+1}}(1+|y-x_j|)^{N-2}}
        \end{aligned}
    \end{equation*}
    If $y\in \R^N\backslash \textbf{B}_{\delta/2}$, then, from Lemma \ref{B3},
    \[
        \int_{\R^N}\frac{C\lambda^{n_h}}{|y-z|^{N-2}}\frac{\lambda^{\frac{pN}{p+1}}|\xi^p-\xi|}{(1+\lambda|z-x_j|)^{p(N-2)}}\dz\leq \frac{C\lambda^{n_h+\frac{N}{q+1}}}{(1+\lambda|y-x_j|)^{N-2}}.
    \]
    Therefore,
    \[
        \int_{\R^N}\frac{C\lambda^{n_h}}{|y-z|^{N-2}}\frac{\lambda^{\frac{pN}{p+1}}|\xi^p-\xi|}{(1+\lambda|z-x_j|)^{p(N-2)}}\dz \leq  C\lambda^{n_h}\left(\frac{\lambda^{-\frac{N} {p+1}} \chi_{\textbf{B}_{\delta/2}}}{(1+|y-x_j|)^{N-2}}+\frac{\lambda^{\frac{N}{q+1}}\chi_{\textbf{B}_{\delta/2}^c}}{(1+\lambda|y-x_j|)^{N-2}}\right).
    \]
    Similarly,
    \[
        \int_{\R^N}\frac{C\lambda^{n_h}}{|y-z|^{N-2}}\frac{\chi_{\textbf{B}_{2\delta}}(z)\lambda^{\frac{N}{q+1}}}{(1+\lambda|z-x_j|)^{N-2}}\dz\leq C\lambda^{n_h}\left(\frac{\lambda^{-\frac{N}{p+1}}\chi_{\textbf{B}_{4\delta}^c}}{(1+|y-x_j|)^{N-2}}+\frac{\lambda^{\frac{N}{q+1}-2}\chi_{\textbf{B}_{4\delta}}}{(1+\lambda|y-x_j|)^{N-4}}\right).
    \]
    Thus, we can obtain
    \begin{equation}\label{barY}
        \begin{aligned}
            \quad\left|\bar{Y}_{j,h}(y)-\frac{\partial (\xi  U_{x_j,\lambda})}{\partial \Box_h}(y)\right| \leq C\lambda^{n_h}\left( \frac{\lambda^{\frac{N}{q+1}}\chi_{\textbf{B}_{\delta/2}^c}}{(1+\lambda|y-x_j|)^{N-2}} +\frac{\lambda^{\frac{N}{q+1}-2}\chi_{\textbf{B}_{4\delta}}}{(1+\lambda|y-x_j|)^{N-4}}\right).
        \end{aligned}
    \end{equation}
    Similarly, we have
    \begin{equation}\label{barZ}
        \begin{aligned}
            \left|\bar{Z}_{j,h}(y)-\frac{\partial (\xi V_{x_j,\lambda})}{\partial \Box_h}(y) \right| \leq  C\lambda^{n_h}\left( \frac{\lambda^{\frac{N}{p+1}}\chi_{\textbf{B}_{\delta/2}^c}}{(1+\lambda|y-x_j|)^{N-2}}+\frac{\lambda^{\frac{N}{p+1}-2}\chi_{\textbf{B}_{4\delta}}}{(1+\lambda|y-x_j|)^{N-4}}\right).
        \end{aligned}
    \end{equation}

    Then, using Lemma \ref{B1} and \ref{B2} and \eqref{barY}, \eqref{barZ}, we have
    \begin{align}
            &\quad \left\langle \left( \sum\limits_{j=1}^{m} q Y_{x_j,\lambda}^{q-1} Y_{j,l},  \sum\limits_{j=1}^{m} p Z_{x_j,\lambda}^{p-1} Z_{j,l}), (\sum\limits_{j=1}^{m}\bar{Y}_{j,h},\sum\limits_{j=1}^{m}\bar{Z}_{j,h} \right) \right\rangle  \nonumber \\
            &= \int_{\R^N} \left( \sum_{j=1}^m  \frac{\pa (\xi^q U_{x_j,\lambda}^q)}{\pa \Box_l} \cdot \sum_{j=1}^m\left( \frac{\pa(\xi U_{x_j,\lambda})}{\pa \Box_h}+\bar{Y}_{j,h}- \frac{\pa(\xi U_{x_j,\lambda})}{\pa \Box_h} \right) \right) \nonumber\\
            & \quad + \int_{\R^N}  \left( \sum_{j=1}^m  \frac{\pa (\xi^p V_{x_j,\lambda}^p)}{\pa \Box_l} \cdot \sum_{j=1}^m\left( \frac{\pa(\xi V_{x_j,\lambda})}{\pa \Box_h}+\bar{Z}_{j,h}- \frac{\pa(\xi V_{x_j,\lambda})}{\pa \Box_h} \right) \right) \nonumber\\
            &=\int\limits_{\R^N} \left(\sum\limits_{j=1}^{m} \frac{\partial (V_{x_j,\lambda}^p)}{\partial \Box_l} \cdot \sum\limits_{j=1}^{m}\frac{\partial V_{x_j,\lambda}}{\partial \Box_h} +\int\limits_{\R^N} \sum\limits_{j=1}^{m} \frac{\partial (U_{x_j,\lambda}^q)}{\partial \Box_l} \cdot \sum\limits_{j=1}^{m}\frac{\partial U_{x_j,\lambda}}{\partial \Box_h}\right) \label{LHS^1} \\
            & \quad + Cm \lambda^{n_h + h_l} \sum_{i \neq j} \int_{\textbf{B}_{2\delta} \backslash \textbf{B}_{\delta/2}} \frac{\lambda^N}{(1+\lambda|y-x_j|)^{(N-2)p}} \frac{1}{(1+\lambda|y-x_i|)^{N-2} } \ \dy \nonumber \\
            & \quad + Cm \lambda^{n_h + h_l} \sum_{i \neq j}  \int_{\textbf{B}_{2\delta} } \frac{\lambda^{N-2}}{(1+\lambda|y-x_j|)^{(N-2)p}} \frac{1}{(1+\lambda|y-x_i|)^{N-4} } \ \dy \nonumber\\
            & \quad + Cm \lambda^{n_h + h_l} \sum_{i \neq j}  \int_{\textbf{B}_{\delta}^c } \frac{\lambda^{N}}{(1+\lambda|y-x_j|)^{(N-2)p}} \frac{1}{(1+\lambda|y-x_j|)^{N-2}} \ \dy \nonumber \\
            &=\delta_{hl}m\lambda^{2n_h}a_h+o(m\lambda^{n_l+n_h-2}), \nonumber
        \end{align}
    for some $a_h>0$, $h=1,2,\cdots,N$.

    \textbf{Estimate the right hand side of \eqref{c_l_0}}. Using Lemma \ref{B2} and \eqref{barY}, we have
    \begin{equation*}
        \begin{aligned}
            &\quad \left|\int_{\R^N}\sum\limits_{j=1}^{m} \bar{Y}_{j,h}\cdot g \right|\\
            &\leq C m\lambda^{n_h}||g||_{**,2} \int_{\R^N}\left(\frac{\lambda^N}{(1+\lambda|y-x_{1}|)^{N-2}}+\frac{\lambda^{N}\chi_{\textbf{B}_{4\delta}}}{\lambda^{\varepsilon}(1+\lambda|y-x_{1}|)^{N-2-\varepsilon}}\right)
            \cdot \sum\limits_{j=1}^{m}\frac{\dy}{(1+\lambda|y-x_{j}|)^{\frac{N}{q+1}+2+\tau}} \\
            &\leq Cm\lambda^{n_h}||g||_{**,2}
        \end{aligned}
    \end{equation*}
    where we have used the fact that
    \[
        \frac{1}{\lambda}\leq C\frac{1}{1+\lambda|y-x_1|},\quad\text{ in}\,\,\textbf{B}_{4\delta}.
    \]
   Similarly, we have
   \[
       \left| \int_{\R^N}\sum\limits_{j=1}^{m} \bar{Z}_{j,h}\cdot f \right| \leq Cm\lambda^{n_h}||f||_{**,1}.
    \]
   Thus, we have
   \begin{align}\label{RHS_1}
   \left|\left\langle (\sum\limits_{j=1}^{m}  \bar{Y}_{j,h},  \sum\limits_{j=1}^{m} \bar{Z}_{j,h}), (f,g)\right\rangle\right|\leq Cm\lambda^{n_h}||(f,g)||_{**}.
   \end{align}
    Inserting \eqref{LHS^1} and \eqref{RHS_1} into \eqref{c_l_0}, we conclude that \eqref{c_l_0} is solvable and
   $$c_h=O\left(\frac{||(f,g)||_{**}}{\lambda^{n_h}}\right).$$

    \textbf{Prove that P is bounded from $Y_s$ to $F$.} Note that
        \begin{equation*}
            \begin{aligned}
                \left|\left|\sum\limits_{j=1}^{m} pZ_{x_j,\lambda}^{p-1} Z_{j,l}\right|\right|_{**,1}&\leq
                C\lambda^{n_l}\sup\limits_{y\in \R^N}\left(\sum\limits_{j=1}^{m}\frac{1}{(1+\lambda|y-x_{j}|)^{\frac{N}{q+1}+2+\tau}}\right)^{-1}\cdot
                \sum\limits_{j=1}^{m}\frac{1}{(1+\lambda|y-x_{j}|)^{(N-2)p}}.
            \end{aligned}
        \end{equation*}
        Define
        \begin{equation}\label{Omega}
            \Omega_j = \left\{ y = (y',y'') \in \R^2 \times \R^{N-2}, \langle \dfrac{y'}{|y'|} , \dfrac{x_j'}{|x_j'|} \rangle \geq \cos \dfrac{\pi}{m}\right\}, \;\; j=1,2,\cdots,N.
        \end{equation}
        Then, for any $y \in \Omega_i$, and $\zeta \geq \tau$, we have
        \begin{equation}\label{I_32_2}
            \begin{aligned}
                \sum\limits_{j=1}^m \frac{1}{(1+\lambda|y-x_j|)^{\zeta}}
                &\leq \frac{1}{(1+\lambda|y-x_1|)^{\zeta}}+\frac{1}{(1+\lambda|y-x_1|)^{\zeta-\tau}}\sum\limits_{j=2}^m \frac{1}{(\lambda|x_1-x_j|)^{\tau}}\\
                &\leq  \frac{C}{(1+\lambda|y-x_1|)^{\zeta-\tau}}.
            \end{aligned}
        \end{equation}
       Hence
        \[
            \left|\left|\sum\limits_{j=1}^{m} pZ_{x_j,\lambda}^{p-1} Z_{j,l}\right|\right|_{**,1} \leq C\lambda^{n_l}\max\limits_{1\leq i\leq m}\sup\limits_{y\in \Omega_i} \left(\sum\limits_{j=1}^{m}\frac{(1+\lambda|y-x_{i}|)^{(N-2)p-\tau}}{(1+\lambda|y-x_{j}|)^{\frac{N}{q+1}+2+\tau}}\right)^{-1} \leq C \lambda^{n_l},
        \]
        where in the second inequality, we have used $(N-2)p-\tau>\frac{N}{q+1}+2+\tau$ when $p>\frac{N}{N-2}$ and $N \geq 6$, or $p \geq 2$ and $N=5$. Similarly, we have
        \[
            \left| \left|\sum\limits_{j=1}^{m} qY_{x_j,\lambda}^{q-1} Y_{j,l}\right|\right|_{**,2}\leq C \lambda^{n_l}.
        \]

        As a result, we obtain
        \begin{equation*}
            \begin{aligned}
                 \left|\left|\sum\limits_{l=1}^{N} c_l\big(\sum\limits_{j=1}^{m} pZ_{x_j,\lambda}^{p-1} Z_{j,l}, \sum\limits_{l=1}^{N} c_l \sum\limits_{j=1}^{m} qY_{x_j,\lambda}^{q-1} Y_{j,l}\big)\right|\right|_{**} &\leq \sum\limits_{l=1}^{N} c_l \left(\left| \left| \sum\limits_{j=1}^{m} pZ_{x_j,\lambda}^{p-1} Z_{j,l} \right| \right|_{**,1} + \left| \left|\sum\limits_{j=1}^{m} qY_{x_j,\lambda}^{q-1} Y_{j,l} \right| \right|_{**,2}\right)\\
                 &\leq O(||(f,g)||_{**})
            \end{aligned}
        \end{equation*}
        which implies that
        $$||\textbf{P}(f,g)||_{**}\leq C ||(f,g)||_{**}$$

    \end{proof}

\begin{prop}\label{existence1}
    There is a $m_0>0$ such that for any $m > m_0$, and all $(h_1,h_2) \in Y_s$, the problem \eqref{linear-equation} has a unique solution $(\phi,\psi) : = L_m(h_1,h_2)$, which satisfies
    \[
        ||L_m(h_1,h_2)||_{*} \leq C||(h_1,h_2)||_{**}, \;\; |c_{l}| \leq \dfrac{C}{\lambda^{n_l}} ||(h_1,h_2)||_{**}.
    \]
\end{prop}

\begin{proof}
    Since $(\phi, \psi) \in E$, we know that $(-\Delta \phi + V \phi, -\Delta \psi + V\psi)$ satisfies
    \[
        \begin{split}
            \left\langle (-\Delta \phi + V \phi, -\Delta \psi + V\psi) , \sum_{j=1}^m(\bar{Z}_{j,l}, \bar{Y}_{j,l}) \right\rangle & = \sum_{j=1}^m \int_{\R^N} \left( \phi (-\Delta + V) \bar{Z}_{j,l} + \psi (-\Delta + V) \bar{Y}_{j,l} \right) \ \dx \\
            & = \sum_{j=1}^m \int_{\R^N} (qY_{x_j,\lambda}^{q-1}Y_{j,l} \phi + p Z_{x_j,\lambda}^{p-1}Z_{j,l} \psi) \ \dx =0.
        \end{split}
    \]
    Therefore $(-\Delta \phi + V \phi, -\Delta \psi + V\psi) \in F$. Applying \textbf{P} on both sides of \eqref{linear-equation}, we obtain
    \[
        (\phi,\psi)= (-\Delta+V)^{-1}\cdot\textbf{P}(p(Z_{\bar{r}, \bar{y}'', \lambda})^{p-1} \psi,  q(Y_{\bar{r}, \bar{y}'', \lambda})^{q-1} \phi)+ (-\Delta+V)^{-1}\cdot\textbf{P}(h_1,h_2),\quad (\phi,\psi)\in E.
    \]

    On the other hand, we can deduce from Lemma \ref{nnl2-19-1} and Lemma \ref{B1_1} that
    \[
        \left|q(Y_{\bar{r}, \bar{y}'', \lambda})^{q-1} \phi\right|\leq C\lambda^{\frac{qN}{q+1}}||\phi||_{*,1}\cdot
         \left(\sum_{j=1}^k \frac{1}{(1+\lambda|y-x_j|)^{\frac{N}{q+1}+\tau}}\right)^{q} \le C||\phi||_{*,1}\sum_{j=1}^k \frac{ \lambda^{\frac{N}{p+1}+2} }{ (1+\lambda |y-x_j| )^{  \frac{N}{p+1}+2+\tau } },
    \]
    and similarly,
    \[
         \left|p(Z_{\bar{r}, \bar{y}'', \lambda})^{p-1} \psi\right|\leq
          C||\psi||_{*,2}\sum_{j=1}^k \frac{ \lambda^{\frac{N}{q+1}+2} }{ (1+\lambda |y-x_j| )^{  \frac{N}{q+1}+2+\tau } },
    \]
    which implies that
    \[
        \left|\left|\left(p(Z_{\bar{r}, \bar{y}'', \lambda})^{p-1} \psi,  q(Y_{\bar{r}, \bar{y}'', \lambda})^{q-1} \phi\right)\right|\right|_{**}\leq C||(\phi,\psi)||_{*}.
    \]
    Define the operator
    \[
        T(\phi,\psi)=(-\Delta+V)^{-1}\cdot\textbf{P}(p(Z_{\bar{r}, \bar{y}'', \lambda})^{p-1} \psi,  q(Y_{\bar{r}, \bar{y}'', \lambda})^{q-1} \phi).
    \]
    Then, $T$ is a bounded linear operator from $E$ to itself. Moreover, we can derive from the fact that $|Y_{\bar{r}, \bar{y}'', \lambda}(y)|+|Z_{\bar{r}, \bar{y}'', \lambda}(y)|\rightarrow 0$ as $|y|\rightarrow+\infty$ and  the Arzel\`a-Ascoli Theorem that $T$ is compact.

    On the other hand, Lemma \ref{blowup} guarantees that the homogeneous equation
    $$(\phi,\psi)=T(\phi,\psi)$$
    has only the zero solution in $E$, which implies that $I-T$ is an injection. Thus, Fredholm's alternative gives that $I-T$ is a bijection in $E$. Therefore, there is a unique solution $(\phi, \psi)$ of \eqref{linear-equation}. Furthermore, using Lemma \ref{blowup} and \eqref{c}, we have
    $$||(\phi,\psi)||_{*} \leq C||(h_1,h_2)||_{**}, \ \ |c_{l}| \leq \dfrac{C}{\lambda^{n_l}} ||(h_1,h_2)||_{**}.$$
\end{proof}

\vskip8pt

Now we consider the following linearized problem:
\begin{equation}\label{linear-equation2}
   \begin{cases}
      L(\phi,\psi) = l+ N(\phi,\psi) + \sum\limits_{l=1}^{N} c_l ( \sum\limits_{j=1}^{m} pZ_{x_j,\lambda}^{p-1} Z_{j,l}, \sum\limits_{j=1}^{m} qY_{x_j,\lambda}^{q-1} Y_{j,l})\;\;\; \hbox{in} \;\;\; \R^N,\\
      (\phi, \psi) \in E ,
   \end{cases}
\end{equation}
where $l$ and $N(\phi,\psi)$ are defined as \eqref{op-3} and \eqref{op-4}. We have

\begin{prop}\label{existence2}
    There is a $m_0>0$ such that if $m>m_0$, $\lambda \in [L_0m^{\frac{N-2}{N-4}}, L_1m^{\frac{N-2}{N-4}}]$, $\bar{r} \in [r_0-\theta, r_0 + \theta]$ and $\bar{y}''\in B_{\theta}(y_0'')$, where $\theta >0$ is small, \eqref{linear-equation2} has a unique solution $(\phi,\psi)=(\phi_{\bar{r},\bar{y}'',\lambda}, \psi_{\bar{r},\bar{y}'',\lambda})$ satisfying
    \begin{align}\label{phi_psi_c}
         ||(\phi,\psi)||_{*} \leq \dfrac{C}{\lambda^{1+\varepsilon}}, \;\; |c_l| \leq \dfrac{C}{\lambda^{n_l+1+\varepsilon}},
    \end{align}
    where $\varepsilon>0$ is a small constant.
\end{prop}

To prove this proposition, we need to estimate $l$ and $N(\phi,\psi)$ respectively. For later use, we define the following set:
\begin{equation}\label{eq:S}
    S := B_{r_0^{-1}m^{-1}}(x_1) = \left\{y \in \R^N: |y-x_1| < r_0^{-1}m^{-1}\right\}.
\end{equation}
We choose $r_0>0$ to be a large constant such that $S \in \Omega_1$. In the proof of Lemma \ref{l1-23-4}, we obtain $\sum\limits_{j=2}^m U_{x_j,\lambda} + \varphi \le CU_{x_1,\lambda}$ in $S$, where $\varphi$ is defined as $\varphi = Y_{\bar{r}, \bar{y}'', \lambda}^* - \sum\limits_{j=1}^m U_{x_j,\lambda}$.

\begin{lemma}\label{N}
    Suppose $N\geq 6$, $p\in(\frac{N}{N-2},\frac{N+2}{N-2})$ or $N=5$, $p\in (2,\frac{7}{3})$ and $(p,q)$ satisfies condition \eqref{critical hyperbola}, then for $||(\phi,\psi)||_*$ small enough, it holds that
    \[
        ||N(\phi,\psi)||_{**} \leq C||(\phi,\psi)||_{*}^{\min\{p,2\}}
    \]
\end{lemma}

\begin{proof}
   Recall that $N_2(\phi) = |Y_{\bar{r},\bar{y}'',\lambda} + \phi|^{q-1}(Y_{\bar{r},\bar{y}'',\lambda} + \phi) - (Y_{\bar{r},\bar{y}'',\lambda})^q - q (Y_{\bar{r},\bar{y}'',\lambda})^{q-1}\phi$. Thus, we have
\[ |N_2(\psi)| \leq
\begin{cases}
   C|\phi|^q, \;\;\; \hbox{if}\;\; q \leq 2 ;\\
   C(Y_{\bar{r},\bar{y}'',\lambda})^{q-2}\phi^2 + C|\phi|^q \;\; \hbox{if}\;\; q > 2.
\end{cases}
\]
Note that the case $q>2$ occurs only if $N\leq 8$.

First, we can derive from Lemma \ref{B1_1} that
\[
\begin{split}
    |\phi|^{q}&\leq \lambda^{\frac{qN}{q+1}} ||\phi||_{*,1}^{q} \left(\sum\limits_{j=1}^{m} \dfrac{1}{(1+\lambda|y-x_j|)^{\frac{N}{q+1}+\tau}}\right)^{q}\\
    &\leq C ||\phi||_{*,1}^{q} \sum\limits_{j=1}^{m}\dfrac{\lambda^{\frac{N}{p+1}+2}}{(1+\lambda|y-x_j|)^{\frac{N}{p+1}+2+\tau}}.
\end{split}
\]
If $q>2$, from Lemma \ref{B1_1} and Lemma \ref{l1-23-4}, we estimate that
\[
\begin{split}
    \big|(Y_{\bar{r},\bar{y}'',\lambda})^{q-2}\phi^2\big|&\leq C||\phi||_{*,1}^{2}
    \left(\sum\limits_{j=1}^{m} \dfrac{\lambda^{\frac{N}{q+1}}}{(1+\lambda|y-x_j|)^{\frac{N}{q+1}+\tau+\theta}} \right)^{q-2}\left(\sum\limits_{j=1}^{m} \dfrac{\lambda^{\frac{N}{q+1}}}{(1+\lambda|y-x_j|)^{\frac{N}{q+1}+\tau}}\right)^2\\
    &\leq C||\phi||_{*,1}^{2}\left(\sum\limits_{j=1}^{m} \dfrac{\lambda^{\frac{N}{q+1}}}{(1+\lambda|y-x_j|)^{\frac{N}{q+1}+\tau}}\right)^q\\
    &\leq  C||\phi||_{*,1}^{2} \sum\limits_{j=1}^{m}\dfrac{\lambda^{\frac{N}{p+1}+2}}{(1+\lambda|y-x_j|)^{\frac{N}{p+1}+2+\tau}}.
\end{split}
\]
Hence, we have
$$\big|N_2(\phi)\big|\leq C||\phi||_{*,1}^{\min \{q,2\}} \sum\limits_{j=1}^{m}\dfrac{\lambda^{\frac{N}{p+1}+2}}{(1+\lambda|y-x_j|)^{\frac{N}{p+1}+2+\tau}}.$$
Similarly,
$$  
    |N_1(\psi)| \leq \begin{cases}
    C|\psi|^p \;\; p \leq 2 ;\\
    C |Z_{\bar{r},\bar{y}'',\lambda}|^{p-2}|\psi|^2 + C|\psi|^p \;\; \hbox{if}\;\; p > 2.
\end{cases}
$$
Note that the case $p>2$ occurs only when $N=5$ and $p \in (2, \frac{7}{3} ]$.
Thus, we have
$$\big|N_1(\psi)\big|\leq C||\psi||_{*,2}^{\min \{p,2\}} \sum\limits_{j=1}^{m}\dfrac{\lambda^{\frac{N}{q+1}+2}}{(1+\lambda|y-x_j|)^{\frac{N}{q+1}+2+\tau}}.$$
Since $q \geq p$ and we assume that $||(\phi, \psi)||_{*}$ is small enough, we have
\[
   ||N(\phi, \psi)||_{**} \leq C||(\phi,\psi)||_{*}^{min(p,2)}.
\]
\end{proof}

\begin{lemma}\label{l}
   Under the same condition as in Lemma \ref{N}, we have
   \[
       ||(l_1, l_2)||_{**} \leq \dfrac{C}{\lambda^{1+\varepsilon}}.
   \]
where $\varepsilon > 0$ is a fixed small constant.
\end{lemma}

\begin{proof}
   Recall that
   \begin{equation*}
    \begin{split}
        l= \big(l_1,l_2\big)=  & \Big( ( Z_{\bar{r}, \bar{y}'', \lambda}^p - \xi (Z_{\bar{r}, \bar{y}'', \lambda}^*)^p ) - VY_{\bar{r}, \bar{y}'', \lambda} + ( Y_{\bar{r}, \bar{y}'', \lambda}^* \Delta \xi + 2\nabla \xi \nabla Y_{\bar{r}, \bar{y}'', \lambda}^* ) , \\
        & (Y_{\bar{r}, \bar{y}'', \lambda}^q - \xi \sum\limits_{j=1}^m U_{x_j,\lambda}^q) - VZ_{\bar{r}, \bar{y}'', \lambda} + (Z_{\bar{r}, \bar{y}'', \lambda}^* \Delta \xi + 2\nabla \xi \nabla Z_{\bar{r}, \bar{y}'', \lambda}^*) \Big)\\
        & : = (H_{1}+H_{2}+H_{3}, K_{1}+K_{2}+K_{3})
    \end{split}
    \end{equation*}
    We will estimate $H_1,H_2,H_3,K_1,K_2,K_3$ respectively. We use the notation in Lemma \ref{prop1}, i.e. $A = \{ y = (y',y''): \delta < |(y',y'') - (r_0,y_0'')| <2\delta \}$, and $\text{supp}(\xi)=\textbf{B}_{2\delta} =\{(r,y'')|(r, y'') - (r_0, y_0')| \leq 2\delta\}$. Since $l \equiv 0$ in $\textbf{B}_{2 \delta}^c$, it remains to estimate $(l_1,l_2)$ in $\textbf{B}_{2 \delta}$. Without loss of generality, we may assume $y \in \Omega_1$. Then, we have $|y-x_j| \geq |y-x_1|$.

    \textbf{Estimate $H_1$.} If $N\geq 6$, $p\in(\frac{N}{N-2},\frac{N+2}{N-2})$ or $N=5$, $p\in (2,\frac{7}{3})$, we have
    \[
        (N-2-\tau)p-\frac{N}{q+1}-2-\tau>1.
    \]
    Then we can use $\frac{1}{1+\lambda|y-x_1|}\leq \frac{C}{\lambda}$ in $A$ and Lemma \ref{B2}, \ref{B3} to obtain that
    \begin{align}
        |H_1|&\leq C\chi_{A}\left( \sum\limits_{j=1}^{m}\dfrac{\lambda^{\frac{N}{p+1}} }{(1+\lambda|y-x_j|)^{N-2}}\right)^p
        \leq C\chi_{A}\dfrac{\lambda^{\frac{pN}{p+1}} }{(1+\lambda|y-x_1|)^{(N-2-\tau)p}} \label{H1-1} \\
        & \leq \dfrac{C\chi_{A}}{\lambda^{(N-2-\tau)p - \frac{N}{q+1} - 2 - \tau}} \dfrac{\lambda^{\frac{pN}{p+1}}}{(1+\lambda|y-x_1|)^{\frac{N}{q+1} + 2 + \tau}}=O\left(\frac{1}{\lambda^{1+\varepsilon}} \right) \sum\limits_{j=1}^m \dfrac{\lambda^{\frac{N}{q+1}+2}}{(1+\lambda|y-x_j|)^{\frac{N}{q+1} + 2+\tau}} .\nonumber
    \end{align}
   Therefore, for some small constant $\varepsilon>0$, we have
   \[
       ||H_1||_{**,1} =  O\left(\frac{1}{\lambda^{1+\varepsilon}} \right) .
    \]
  \textbf{Estimate $H_2$ and $K_2$}. From Lemma \ref{lemma:U}, we know for some small $\theta>0$ that
    \begin{align}
        |H_2| &\leq C\sum_{j=1}^m \frac{\lambda^{\frac{N}{q+1}} \xi(y)}{(1+\lambda|y-x_j|)^{N-2}} +C\left(\frac{m}{\lambda} \right)^{p(N-2)-2} \sum_{j=1}^m \frac{\lambda^{\frac{N}{q+1}} \xi(y)}{(1+m|y-x_j|)^{\min \{N-2,p(N-3-\theta)-2\}}} \nonumber \\
        & :=H_{21}+H_{22}. \label{H2}
    \end{align}

Since $N-1>\min\{\frac{N}{p+1},\frac{N}{q+1}\}+\tau + 2$ when $N\geq 6$, $p\in(\frac{N}{N-2},\frac{N+2}{N-2}]$ or $N=5$, $p\in (2,\frac{7}{3}) $, we have
\begin{equation}\label{H_21}
    \begin{aligned}
        H_{21}\leq C \sum\limits_{j=1}^m \dfrac{\lambda^{\frac{N}{q+1}+2} \xi(y)}{\lambda^{1+\varepsilon}(1+\lambda|y-x_j|)^{N-1-\varepsilon}} \leq \dfrac{C}{\lambda^{1+\varepsilon}} \sum\limits_{j=1}^m \dfrac{\lambda^{\frac{N}{q+1}+2}}{(1+\lambda|y-x_j|)^{\frac{N}{q+1}+\tau + 2}},
    \end{aligned}
\end{equation}
where in the first inequality, we have used
 \[
  \dfrac{1}{\lambda} \leq  \dfrac{C}{1+\lambda|y-x_1|} \;\; \text{, when} \;\;
  y \in \textbf{B}_{2\delta}.
\]
Hence, we have
\[
    ||H_{21}||_{**,1} = O\left(\frac{1}{\lambda^{1+\varepsilon}} \right) .
\]

For $H_{22}$, if $p(N-3)>N$, then $H_{22}$ can be controlled by $H_{21}$. Next, we consider the case $p(N-3)\leq N$. In this case,
\[
    H_{22}= C\left( \frac{m}{\lambda} \right)^{p(N-2)-2} \sum_{j=1}^m \frac{\lambda^{\frac{N}{q+1}} \xi(y)}{(1+m|y-x_j|)^{ p(N-3-\theta)-2}}.
\]
We estimate $H_{22}$ for $y \in S$ and $y \not\in S$. For any $y\in S$ and any $\sigma>1$, we have
$$\sum\limits_{j=2}^m \dfrac{1}{(1+m|y-x_j|)^{\sigma}}\leq \sum\limits_{j=2}^m \dfrac{C}{(m|x_1-x_j|)^{\sigma}}\leq C\leq \dfrac{C}{(1+m|y-x_1|)^{\sigma}}, \ \ \ \frac{m}{\lambda} \leq \frac{1}{1+\lambda|y-x_1|} .$$
Then, we have
\begin{equation}\label{H22S}
    \begin{aligned}
        H_{22}&\leq \left(\frac{m}{\lambda} \right)^{p(N-2)-2} \frac{\lambda^{\frac{N}{q+1}}}{(1+m|y-x_1|)^{ p(N-3-\theta)-2}}\leq C\left( \frac{m}{\lambda} \right)^{p(N-2)-2}\cdot \lambda^{\frac{N}{q+1}}\\
        &\leq \frac{C}{\lambda^2}\left(\frac{m}{\lambda} \right)^{p(N-2)-\frac{N}{q+1}-4-\tau} \left(\lambda^{\frac{N}{q+1}+2}\cdot \left(\frac{m}{\lambda} \right)^{\frac{N}{q+1}+2+\tau} \right)\\
        &\leq \dfrac{C}{\lambda^{1+\varepsilon}}\sum\limits_{j=1}^m \dfrac{\lambda^{\frac{N}{q+1}+2}}{(1+\lambda|y-x_j|)^{\frac{N}{q+1}+\tau + 2}}.
    \end{aligned}
\end{equation}
The last inequality holds since $m \approx \lambda^{\tau}$ and
\[
    1+(1-\tau)\left(p(N-2)-2-\frac{N}{q+1}-2-\tau\right)
    =1+(1-\tau)\left(\frac{pN}{q+1}-2-\tau \right)>0.
\]

For any $y\in \big(\textbf{B}_{2\delta}\ \backslash S\big)\cap \Omega_1$, we have $1 + \lambda|y-x_1| \leq C\lambda |y-x_1|$. Similarly to the computation in $H_{21}$ and using Lemma \ref{B2}, we have
we have
     \begin{align}
         H_{22}&\leq C\left( \frac{m}{\lambda} \right)^{p(N-2) -p(N-3-\theta)} \sum_{j=1}^m \frac{\lambda^{\frac{N}{q+1}} \xi(y)}{(\lambda|y-x_j|)^{ p(N-3-\theta)-2}} =C\left( \frac{m}{\lambda} \right)^{p(1+\theta)} \sum_{j=1}^m \frac{\lambda^{\frac{N}{q+1}} \xi(y)}{(\lambda|y-x_j|)^{ p(N-3-\theta)-2}} \nonumber \\
         & \leq C \left( \frac{m}{\lambda} \right)^{p(1+\theta)} \frac{\lambda^{\frac{N}{q+1}} \xi(y)}{(\lambda|y-x_1|)^{ p(N-3-\theta)-2 - \tau}} \leq C \left( \frac{m}{\lambda} \right)^{p(1+\theta)} \frac{\lambda^{\frac{N}{q+1}+2} \xi(y)}{\lambda^2(1+\lambda|y-x_1|)^{ p(N-3-\theta)-2 - \tau}} \label{H22Sc} \\
         &\leq C \dfrac{\lambda^{\frac{N}{q+1}+2} \xi(y)}{\lambda^{1+\varepsilon}(1+\lambda|y-x_1|)^{p(N-3-\theta)-1-\varepsilon+(1-\tau)p(1+\theta)-\tau}} \leq \dfrac{C}{\lambda^{1+\varepsilon}}\sum\limits_{j=1}^m \dfrac{\lambda^{\frac{N}{q+1}+2}}{(1+\lambda|y-x_j|)^{\frac{N}{q+1}+\tau + 2}},\nonumber
     \end{align}
 where we have used
 \[
    p(N-3)-1+p(1-\tau) - \tau=p(N-2-\tau)-1 - \tau>\frac{N}{q+1}+2+\tau,
\]
and
 $$\frac{1}{\lambda}\leq \frac{C}{1+\lambda|y-x_1|}, \quad\text{in } \ \ \textbf{B}_{2\delta},$$
 Combine \eqref{H22S} and \eqref{H22Sc}, it follows that
 \begin{equation}\label{H_2}
     ||H_2||_{**,1}\leq ||H_{21}||_{**,1} + ||H_{22}||_{**,1} \leq  \frac{C}{\lambda^{1+\varepsilon}}.
 \end{equation}
Similarly, we have
\begin{equation}\label{K_2}
    ||K_2||_{**,2}\leq \left| \left|  \sum_{j=1}^m \frac{\lambda^{\frac{N}{p+1}} \xi(y)}{(1+\lambda|y-x_j|)^{N-2}} \right| \right|_{**,2} \leq \frac{C}{\lambda^{1+\varepsilon}}.
\end{equation}

\textbf{Estimate $H_3$ and $K_3$}.
In $A$, it holds that $\nabla Y^*_{\bar{r}, \bar{y}'', \lambda}(y)=O\big(Y^*_{\bar{r}, \bar{y}'', \lambda}(y)\big)$, $\nabla Z^*_{\bar{r}, \bar{y}'', \lambda}(y)=O\big(Z^*_{\bar{r}, \bar{y}'', \lambda}(y)\big)$ and
$$\frac{C_1}{\lambda}\leq \frac{1}{1+\lambda|y-x_j|}\leq \frac{C_2}{\lambda}\quad j=1,\cdots,N$$
for some $C_1,\,C_2>0$. Therefore, we have
$$\sum\limits_{j=1}^m \dfrac{\lambda^{\frac{N}{q+1}+2}}{(1+\lambda|y-x_j|)^{\frac{N}{q+1}+\tau + 2}}\geq C.$$
 When $N\geq 6$, $p\in(\frac{N}{N-2},\frac{N+2}{N-2}]$ and $N=5$, $p\in (2,\frac{7}{3}) $, it holds
 $$\min\left\{\frac{N}{p+1}-\tau,\frac{N}{q+1}-\tau\right\}>1,$$
 Therefore, from Lemma \ref{B2} and \ref{B3}, we have
\begin{equation}\label{H3}
     \begin{split}
         |H_3| & \leq C \chi_{A}\sum\limits_{j=1}^m \dfrac{\lambda^{\frac{N}{q+1}}}{(1+\lambda|y-x_j|)^{N-2}} \leq \frac{C}{\lambda^{\frac{N}{p+1}-\tau}}
         \leq \dfrac{C}{\lambda^{1+\varepsilon}} \sum\limits_{j=1}^m \dfrac{\lambda^{\frac{N}{q+1}+2}}{(1+\lambda|y-x_j|)^{\frac{N}{q+1}+\tau + 2}},
     \end{split}
 \end{equation}
which implies
\begin{equation}\label{H_32}
    ||H_3||_{**,1}\leq \frac{C}{\lambda^{1+\varepsilon}}.
\end{equation}
Similarly, we have
\begin{equation}\label{K_3}
    ||K_3||_{**,2}\leq \frac{C}{\lambda^{1+\varepsilon}}.
\end{equation}

\textbf{Estimate $K_1$.} We first estimate $K_1$ in $S$. For any $y\in S$, we have $\sum\limits_{j=2}^m U_{x_j,\lambda} + \varphi \le C \le C U_{x_1,\lambda}$ and $\xi(y) =1$, thus we have
\begin{align*}
        Y_{\bar{r}, \bar{y}'', \lambda}^q - \xi U_{x_1,\lambda}^q & =(Y_{\bar{r}, \bar{y}'', \lambda}^*)^q -  U_{x_1,\lambda}^q = (U_{x_1,\lambda} + \sum\limits_{j=2}^m U_{x_j,\lambda} + \varphi)^q - U_{x_1,\lambda}^q \\
        &\leq C U_{x_1,\lambda}^{q-1}(\sum\limits_{j=2}^m U_{x_j,\lambda} + \varphi)\leq CU_{x_1,\lambda}^{q-1},
\end{align*}
and
\[
    \xi \sum\limits_{j=2}^mU_{x_j,\lambda}^q\leq U_{x_1,\lambda}^{q-1}\sum\limits_{j=2}^m U_{x_j,\lambda}\leq CU_{x_1,\lambda}^{q-1}.
\]
Therefore, we can compute that
\begin{equation}\label{K1_1}
    \begin{aligned}
        |K_1| & \leq C U_{x_1,\lambda}^{q-1} \leq C \frac{\lambda^{\frac{(q-1)N}{q+1}}}{(1+\lambda|y-x_1|)^{(q-1)(N-2)}} = C \frac{\lambda^{\frac{N}{p+1}+2}}{\lambda^{\frac{N}{q+1}}(1+\lambda|y-x_1|)^{(q-1)(N-2)}} \\
        &\leq  \dfrac{C}{\lambda^{1+\varepsilon}}\sum\limits_{j=1}^m\dfrac{\lambda^{\frac{N}{p+1} + 2}}{(1+\lambda|y-x_j|)^{(q-1)(N-2)+\frac{N}{q+1}- 1 -\varepsilon}} \leq
        \dfrac{C}{\lambda^{1+\varepsilon}}\sum\limits_{j=1}^m\dfrac{\lambda^{\frac{N}{p+1} + 2}}{(1+\lambda|y-x_j|)^{\frac{N}{p+1} + 2 + \tau}} ,
    \end{aligned}
\end{equation}
where we have used
\[
    (q-1)(N-2)+\frac{N}{q+1} - 1 > \frac{N}{p+1} + 2 + \tau.
\]

For any $y\in (\textbf{B}_{2 \delta}\backslash S) \cap \Omega_1$, from Lemma \ref{lemma:U}, we compute
    \begin{equation}\label{K1_2}
        \begin{aligned}
            &|K_1|\leq  C\left(\sum_{j=1}^m \frac{\lambda^{\frac{N}{q+1}}}{(1+\lambda|y-x_j|)^{N-2}}\right)^{q}+C\left( \left(\frac{m}{\lambda} \right)^{p(1+\theta)} \sum_{j=1}^m \frac{\lambda^{\frac{N}{q+1}}\chi_{\{p(N-3)\leq N\}}}{(1+\lambda|y-x_j|)^{p(N-3-\theta)-2}}\right)^{q} \\
            & : = K_{11} + K_{12}.
        \end{aligned}
    \end{equation}
In $\textbf{B}_{2 \delta}\backslash S$, we have $(1+\lambda|y-x_1|) \geq \lambda/(m r_0) \geq c \lambda^{1-\tau}$. Hence, the first term of the right hand side of \eqref{K1_2} can be bounded by
\begin{equation*}
    \begin{aligned}
        &K_{11} \leq \left(\sum_{j=1}^m \frac{1}{(\lambda|x_1-x_j|)^{\tau}} \cdot\frac{\lambda^{\frac{N}{q+1}}}{(1+\lambda|y-x_1|)^{N-2-\tau}}\right)^{q}\leq C\frac{\lambda^{\frac{qN}{q+1}}}{(1+\lambda|y-x_1|)^{q(N-2-\tau)}}\\
        &\leq C\left(\frac{m}{\lambda}\right)^{q(N-2-\tau)-\frac{N}{p+1}-2-\tau}\frac{\lambda^{\frac{qN}{q+1}}}{(1+\lambda|y-x_1|)^{\frac{N}{p+1}+2+\tau}}= O\left(\dfrac{1}{\lambda^{1+\varepsilon}}\right) \sum\limits_{j=1}^m\dfrac{\lambda^{\frac{N}{p+1} + 2}}{(1+\lambda|y-x_j|)^{\frac{N}{p+1} + 2 + \tau}} .
    \end{aligned}
\end{equation*}
where in the last inequality, we have used
\[
    q(N-2-\tau)-\frac{N}{p+1}-2-\tau=\frac{qN}{(p+1)}-(q+1)\tau>\frac{N-2}{2} = \frac{1}{1-\tau}.
\]
Regarding $K_{12}$, we can use Lemma \ref{B2} to get
\begin{align*}
        K_{12} &\leq \left( \left(\frac{m}{\lambda}\right)^{p(1+\theta)}\cdot\frac{\lambda^{\frac{N}{q+1}}}{(1+\lambda|y-x_1|)^{p(N-3-\theta)-2-\tau}}\right)^{q} \\
        &\leq C\left(\frac{m}{\lambda}\right)^{q(p(N-2)-2-\tau)-\frac{N}{p+1}-2-\tau}\frac{\lambda^{\frac{qN}{q+1}}}{(1+\lambda|y-x_1|)^{\frac{N}{p+1}+2+\tau}} \\
        & =O\left(\dfrac{1}{\lambda^{1+\varepsilon}} \right) \sum\limits_{j=1}^m\dfrac{\lambda^{\frac{N}{p+1} + 2}}{(1+\lambda|y-x_j|)^{\frac{N}{p+1} + 2 + \tau}},
\end{align*}
where we have used
\[
    q[p(N-3) - 2 -\tau] > \frac{N}{p+1} + 2 +\tau,
\]
and
$$q(p(N-2)-2-\tau)-\frac{N}{p+1}-2-\tau=\frac{pqN}{q+1}-(q+1)\tau>\frac{N-2}{2} = \frac{1}{1-\tau}.$$
From \eqref{K1_1} and \eqref{K1_2}, we have
\begin{equation}\label{K_1}
    |K_1|\leq \dfrac{C}{\lambda^{1+\varepsilon}}\sum\limits_{j=1}^m\dfrac{\lambda^{\frac{N}{p+1} + 2}}{(1+\lambda|y-x_j|)^{\frac{N}{p+1} + 2 + \tau}} .
\end{equation}
Combine \eqref{H1-1}, \eqref{H2}, \eqref{K_2}, \eqref{H3}, \eqref{K_3} and \eqref{K_1}, we obtain
\[
    ||(l_1, l_2)||_{**} \leq \dfrac{C}{\lambda^{1+\varepsilon}}.
\]

\end{proof}

\begin{proof}[Proof of Proposition \ref{existence2}]

   Define
  $$ \begin{array}{ll}
      \bar{E}=\displaystyle\left\{ (u,v) \in E \cap X\ | \ \ ||(u,v)||_{*} \leq \dfrac{1}{\lambda^{1+\varepsilon_0}}\right\},
   \end{array}$$
   where $\varepsilon_0 > 0$ is slightly less than $\varepsilon$, which appears in Lemma \ref{l}.

   We will find a solution of (\ref{linear-equation2}) in $E$ which is equivalent to
   \[
      (\phi, \psi) = A(\phi,\psi) : = L_m(N_1(\psi),N_2(\phi)) + L_m(l_1,l_2), \;\; \hbox{for} \;\;  (\phi,\psi) \in E.
   \]

  First, we prove that $A$ maps $ \bar{E}$ into $ \bar{E}$. For any $(u,v) \in  \bar{E}$, by Lemma \ref{existence1}, Lemma \ref{N} and Lemma \ref{l}, we have
   \begin{align}
      ||A(u,v)||_{*} & \leq ||L_m(N_1(v),N_2(u))||_* + ||L_m(l_1,l_2)||_*
       \leq C ||(N_1(v),N_2(u))||_{**} + C ||(l_1,l_2)||_{**} \nonumber \\
      & \leq C ||(u , v)||_{*}^{min(p,2)} + \dfrac{C}{\lambda^{1+\varepsilon}}
       \leq \bigg( \dfrac{C}{\lambda^{1+\varepsilon_0}} \bigg)^{\min\{p,2\}}  + \dfrac{C}{\lambda^{1+\varepsilon} } \leq  \dfrac{1}{\lambda^{1+\varepsilon_0}}, \label{A}
     \end{align}
  Here we choose $m$ large enough such that $A(u,v) \in \bar{E}$.

   Secondly, we prove $A$ is a contraction map. For any $(\omega_1, \omega_2)$ and $(\phi_1,\phi_2) \in \bar{E}$, we have
   \[
      \begin{split}
         ||A(\omega_1,\omega_2) - A(\phi_1,\phi_2)||_* & = ||L_m(N_1(\omega_2),N_2(\omega_1)) - L_m(N_1(\phi_2),N_2(\phi_1))||_*\\
         & \leq C||(N_1(\omega_2) - N_1(\phi_2) ,  N_2(\omega_1) - N_2(\phi_1))||_{**} .
      \end{split}
   \]

It holds that
   \[
      |N_2^{\prime}(t)| \leq \begin{cases}
          C |t|^{q-1}, \;\;\; \hbox{if}\;\; q \leq 2 ;\\
   C(Y_{\bar{r},\bar{y}'',\lambda})^{q-2}t + C|t|^{q-1} \;\; \hbox{if}\;\; q > 2.
      \end{cases}
   \]
   We only present the argument for the case $q>2$, since the other case is similar. Similar to the proof of Lemma \ref{N}, we can derive from Lemma \ref{l}, Lemma \ref{B1_1} and Lemma \ref{l1-23-4} that
   \begin{align*}
         &\quad  |N_2(\omega_1) - N_2(\phi_1)|   \leq C \left( (Y_{\bar{r},\bar{y}'',\lambda})^{q-2}(| \omega_1 | + | \phi_1 | )  + |\omega_1|^{q-1} + |\phi_1|^{q-1}\right)|\omega_1 - \phi_1| \\
         & \leq C (||\omega_1||_{*,1}^{q-1}+||\phi_1||_{*,1}^{q-1}) ||\omega_1 - \phi_1||_{*,1} \left( \sum\limits_{j=1}^{m} \dfrac{\lambda^{\frac{N}{q+1}}}{(1+\lambda|y-x_j|)^{\frac{N}{q+1}+\tau}}  \right)^{q} \\
         & \quad+ C (||\omega_1||_{*,1}+||\phi_1||_{*,1}) ||\omega_1 - \phi_1||_{*,1}\left(  \sum\limits_{j=1}^{m} \dfrac{\lambda^{\frac{N}{q+1}}}{(1+\lambda|y-x_j|)^{\frac{N}{q+1}+\tau+\theta}} \right)^{q-2} \left( \sum\limits_{j=1}^{m} \dfrac{\lambda^{\frac{N}{q+1}}}{(1+\lambda|y-x_j|)^{\frac{N}{q+1}+\tau}}  \right)^{2} \\
         & \leq C \lambda^{-\varepsilon} ||\omega_1 - \phi_1||_{*,1} \sum\limits_{j=1}^m \dfrac{\lambda^{\frac{N}{p+1}+2}}{(1+\lambda|y-x_j|)^{\frac{N}{p+1}+2+\tau}}.
     \end{align*}
   where $\varepsilon>0$ is a small constant.

  Hence, we have
   \[
      ||N_2(\omega_1) - N_2(\phi_1)||_{**,2} \leq C \lambda^{-\varepsilon}||\omega_1-\phi_1||_{*,1}.
   \]
   Similarly, we also have
   \[
      ||N_1(\omega_2) - N_1(\phi_2)||_{**,1} \leq \lambda^{-\varepsilon}||\omega_2-\phi_2||_{*,2}.
   \]
   Thus, we obtain
   \[
   \begin{split}
      ||A(\omega_1,\omega_2) - A(\phi_1,\phi_2)||_* & \leq C||(N_1(\omega_2) - N_1(\phi_2) ,  N_2(\omega_1) - N_2(\phi_1))||_{**} \\
      & \leq C\lambda^{-\varepsilon} ||(\omega_1 - \phi_1, \omega_2-\phi_2)||_{*} \leq \dfrac{1}{2} ||(\omega_1, \omega_2) - (\phi_1, \phi_2)||_{*}.
   \end{split}
   \]
   The last inequality holds if we choose $m$ large enough.

   As a result, there is a $m_0>0$ such that for any $m \geq m_0$, $A$ is a contraction map from $\bar{E}$ to $\bar{E}$ and it follows that there is a unique solution $(\phi.\psi) \in \bar{E}$ of equation (\ref{linear-equation2}). Moreover, by (\ref{A}), we have
   \[
   ||(\phi,\psi)||_{*} = ||A(\phi,\psi)||_{*} \leq \dfrac{C}{\lambda^{1+\varepsilon}}.
   \]
   Finally, estimate of $c_l$ follows from Lemma \ref{existence1}.
\end{proof}

\section{Proof of Main Result}

Recall that the energy functional of \eqref{Main} is defined as
\[
    I(u,v) = \int_{\R^N}(\nabla u \nabla v + V(|y'|,y'')uv ) \ \dy - \dfrac{1}{p+1} \int_{\R^N} |v|^{p+1} \ \dy - \dfrac{1}{q+1} \int_{\R^N} |u|^{q+1} \ \dy.
\]
In this section, we will choose $\bar{r}, \bar{y}''$ and $\lambda$ such that $c_l = 0, l =1,2,\ldots,N$ in \eqref{linear-equation}.
\begin{prop}\label{prop31}
    Suppose that $(\bar{r}, \bar{y}'',\lambda)$ satisfies
   \begin{align}
       \int_{\textbf{B}_\rho}\left( -\Delta u_m+V(r,y'') u_m-|v_m|^{p-1}v_m\right)\langle y, \nabla v_m\rangle+
       \left(-\Delta v_m+V(r,y'') v_m-|u_m|^{q-1}u_m \right)\langle y,\nabla u_m\rangle=0,\label{31}\\
       \int_{\textbf{B}_\rho}\left( -\Delta u_m+V(r,y'') u_m-|v_m|^{p-1}v_m\right)\frac{\partial v_m}{\partial y_i}+
       \left(-\Delta v_m+V(r,y'') v_m-|u_m|^{q-1}u_m \right)\frac{\partial u_m}{\partial y_i}=0,\,\,\, i=3,\cdots,N, \label{32}\\
       \int_{\mathbb{R}^N}\left( -\Delta u_m+V(r,y'') u_m-|v_m|^{p-1}v_m\right)\frac{\partial Z_{\bar{r}, \bar{y}'', \lambda}}{\partial \lambda}+\left(-\Delta v_m+V(r,y'') v_m-|u_m|^{q-1}u_m \right)\frac{\partial Y_{\bar{r}, \bar{y}'', \lambda}}{\partial \lambda}=0,\label{33}
   \end{align}
   where $(u_m,v_m)=(Y_{\bar{r}, \bar{y}'', \lambda}+\phi_{\bar{r}, \bar{y}'', \lambda},Z_{\bar{r}, \bar{y}'', \lambda}+\psi_{\bar{r}, \bar{y}'', \lambda})$ and $\textbf{B}_\rho=\{(r,y''):\,\,|(r,y'')-(r_0,y_0'')|\leq \rho\}$ with $\rho\in (2\delta,5\delta)$, then $c_l=0,\,\,l=1,\cdots N$.
\end{prop}
\begin{proof}
    Since $(Y_{\bar{r}, \bar{y}'', \lambda},Z_{\bar{r}, \bar{y}'', \lambda})= (\xi Y_{\bar{r}, \bar{y}'', \lambda}^*, \xi Z_{\bar{r}, \bar{y}'', \lambda}^*)=(0,0)$ in $\mathbb{R}^N\backslash \textbf{B}_\rho$, we see that if (\ref{31})-(\ref{33}) hold, then
    \begin{equation}\label{P31_1}
    \sum_{l=1}^{N}c_l\sigma_{kl}=0,\,\,k=1,\cdots,N, \ \ \text{where} \ \
        \sigma_{kl}=\int_{\R^N}\left(\sum\limits_{j=1}^{m} q Y_{x_j,\lambda}^{q-1} Y_{j,l},  \sum\limits_{j=1}^{m} p Z_{x_j,\lambda}^{p-1} Z_{j,l}\right)\cdot W_k,
    \end{equation}
    for $k,l=1,\cdots,N$ with
    \[
        W_k=\begin{cases}
            \left( \frac{\partial Y_{\bar{r}, \bar{y}'', \lambda}}{\partial \lambda},\frac{\partial Z_{\bar{r}, \bar{y}'', \lambda}}{\partial \lambda}\right),\,\,\,k=1,\\
            \left(\langle y,\nabla u_m\rangle,\langle y,\nabla v_m\rangle\right),\,\,\,k=2,\\
            \left( \frac{\partial u_m}{\partial y_k},\frac{\partial v_m}{\partial y_k}\right),\,\,\,k=3,\cdots,N.
        \end{cases}
    \]
    By direct calculation, we can prove
    \begin{equation}\label{P31_2}
        \begin{aligned}
            \int_{\R^N}\left(\sum\limits_{j=1}^{m} q Y_{x_j,\lambda}^{q-1} Y_{j,1},  \sum\limits_{j=1}^{m} p Z_{x_j,\lambda}^{p-1} Z_{j,1}\right)\cdot \left(\frac{\partial Y_{\bar{r}, \bar{y}'', \lambda}}{\partial \lambda},\frac{\partial Z_{\bar{r}, \bar{y}'', \lambda}}{\partial \lambda}\right)=\frac{m}{\lambda^2}(a_1+o(1)),\\
            \int_{\R^N}\left(\sum\limits_{j=1}^{m} q Y_{x_j,\lambda}^{q-1} Y_{j,2},  \sum\limits_{j=1}^{m} p Z_{x_j,\lambda}^{p-1} Z_{j,2}\right)\cdot \left(\langle y',\nabla_{y'}Y_{\bar{r}, \bar{y}'', \lambda} \rangle,\langle y',\nabla_{y'} Z_{\bar{r}, \bar{y}'', \lambda} \rangle \right)=m\lambda^2(a_2+o(1)),\\
            \int_{\R^N}\left(\sum\limits_{j=1}^{m} q Y_{x_j,\lambda}^{q-1} Y_{j,i},  \sum\limits_{j=1}^{m} p Z_{x_j,\lambda}^{p-1} Z_{j,i}\right)\cdot \left(\frac{\partial Y_{\bar{r}, \bar{y}'', \lambda}}{\partial y_i},\frac{\partial Z_{\bar{r}, \bar{y}'', \lambda}}{\partial y_i}\right)=m\lambda^2(a_i+o(1)),\,\,\,i=3,\cdots,N,
        \end{aligned}
    \end{equation}
    and
    \begin{align}\label{P31_21}
        \int_{\R^N}\left(\sum\limits_{j=1}^{m} q Y_{x_j,\lambda}^{q-1} Y_{j,s},  \sum\limits_{j=1}^{m} p Z_{x_j,\lambda}^{p-1} Z_{j,s}\right)\cdot \left(\langle y_s,\nabla_{y_s} Y_{\bar{r}, \bar{y}'', \lambda}\rangle,\langle y_s,\nabla_{y_s} Z_{\bar{r}, \bar{y}'', \lambda}\rangle\right)=m\lambda^2(a_{2s}+o(1)),\nonumber\\
        s=3,\cdots,N,
    \end{align}
    for some constants $a_i\neq 0,\, i=1,2,\cdots,N$, $a_{2s}\neq0,\, s=3,\cdots,N$ and $o(1)\rightarrow0$, as $\lambda\rightarrow +\infty$.

     Next, we estimate the terms involving the derivative of $(\phi,\psi)$ in $W_k$, $k=2,\cdots,N$.
     \begin{equation}\label{P31_3}
         \begin{aligned}
             &\left| \int_{\R^N} \sum\limits_{j=1}^{m} p Z_{x_j,\lambda}^{p-1} Z_{j,l}\cdot \frac{\partial \psi}{\partial y_k}\right| = \left| \int_{\R^N} \frac{\partial}{\partial y_k}\big(\sum\limits_{j=1}^{m} p Z_{x_j,\lambda}^{p-1} Z_{j,l}\big)\cdot \psi \right|\\
             &\leq C \lambda^{1+n_l} ||\psi||_{*,2} \int_{\R^N} \sum\limits_{j=1}^{m} \frac{\lambda^N}{(1+\lambda|y-x_j|)^{(N-2)p}}\cdot\sum\limits_{j=1}^{m}\frac{1}{(1+\lambda|y-x_j|)^{\frac{N}{p+1}+\tau}}=o(m\lambda^{1+n_l}).
         \end{aligned}
     \end{equation}
     Similarly, we can prove
    \begin{equation}\label{P31_4}
        \begin{gathered}
            \int_{\R^N}\left(\sum\limits_{j=1}^{m} q Y_{x_j,\lambda}^{q-1} Y_{j,l},  \sum\limits_{j=1}^{m} p Z_{x_j,\lambda}^{p-1} Z_{j,l}\right)\cdot \left(\langle y,\nabla \phi_{\bar{r}, \bar{y}'', \lambda}\rangle,\langle y,\nabla \psi_{\bar{r}, \bar{y}'', \lambda}\rangle\right)=o(m\lambda^{n_l+1}), \\
            \int_{\R^N}\left(\sum\limits_{j=1}^{m} q Y_{x_j,\lambda}^{q-1} Y_{j,l},  \sum\limits_{j=1}^{m} p Z_{x_j,\lambda}^{p-1} Z_{j,l}\right)\cdot \left( \frac{\partial \phi_{\bar{r}, \bar{y}'', \lambda}}{\partial y_k},\frac{\partial \psi_{\bar{r}, \bar{y}'', \lambda}}{\partial y_k}\right)=o(m\lambda^{n_l+1}), \ \ k=3,\cdots,N.
            \end{gathered}
    \end{equation}
    Next, we estimate $\sigma_{kl}$. For $k=2$, combine \eqref{P31_2}, \eqref{P31_21}, \eqref{P31_4} and
    \[
        \left(\langle y,\nabla Y_{\bar{r}, \bar{y}'', \lambda}\rangle,\langle y,\nabla Z_{\bar{r}, \bar{y}'', \lambda}\rangle\right)=\left(\langle y',\nabla_{y'} Y_{\bar{r}, \bar{y}'', \lambda}\rangle,\langle y',\nabla_{y'} Z_{\bar{r}, \bar{y}'', \lambda}\rangle\right)+\left(\langle y'',\nabla_{y''} Y_{\bar{r}, \bar{y}'', \lambda}\rangle,\langle y'',\nabla_{y''} Z_{\bar{r}, \bar{y}'', \lambda}\rangle\right),
    \]
    we have
    \begin{equation}\label{P31_c1}
        \begin{aligned}
            \sigma_{2l}&
            =\int_{\R^N}\big(\sum\limits_{j=1}^{m} q Y_{x_j,\lambda}^{q-1} Y_{j,l},  \sum\limits_{j=1}^{m} p Z_{x_j,\lambda}^{p-1} Z_{j,l}\big)\cdot
            \left(\langle y',\nabla_{y'}Y_{\bar{r}, \bar{y}'', \lambda} \rangle,\langle y',\nabla_{y'} Z_{\bar{r}, \bar{y}'', \lambda} \rangle\right)\\
            &\hspace{1em}+\sum\limits_{s=3}^{N}\int_{\R^N}\big(\sum\limits_{j=1}^{m} q Y_{x_j,\lambda}^{q-1} Y_{j,l},  \sum\limits_{j=1}^{m} p Z_{x_j,\lambda}^{p-1} Z_{j,l}\big)\cdot
            \left(\langle y_s,\nabla_{y_s}Y_{\bar{r}, \bar{y}'', \lambda} \rangle,\langle y_s,\nabla_{y_s} Z_{\bar{r}, \bar{y}'', \lambda} \rangle\right)+o(m\lambda^{1+n_l})\\
            &=\delta_{2l} m\lambda^2(a_2+o(1))+\sum_{s=3}^{N}\delta_{sl}m\lambda^2(a_{2s}+o(1))+o(m\lambda^{1+n_l}).
        \end{aligned}
    \end{equation}
    For $k\neq 2$, combine (\ref{P31_2}) and (\ref{P31_4}), we have
    \begin{equation}\label{P31_c2}
        \begin{aligned}
            \sigma_{kl}=\delta_{kl}m\lambda^{n_k+n_l}(a_k+o(1))+o(m\lambda^{n_k+n_l}).
        \end{aligned}
    \end{equation}
    Based on (\ref{P31_c1}) and (\ref{P31_c2}), the algebra equations (\ref{P31_1}) governing {$c_l$} can be transformed into
    \begin{align}\label{P31_e}
       \left( \Lambda^T P \Lambda \right)\textbf{c}=\textbf{0},
    \end{align}
    where $\textbf{c}=(c_1,\cdots,c_N)^T$, $\Lambda=\text{diag}(\frac{1}{\lambda},\lambda,\cdots,\lambda)\in \mathbf{M}_{N}(\mathbb{R})$, $\textbf{0}$ is the zero vector in $\R^N$ and
    \[
        P =
        \begin{pmatrix}
            a_{1}+o(1) & o(1) & o(1) & \cdots & o(1)\\
            o(1)      & a_{2}+o(1) & a_{23} & \cdots & a_{2n}\\
            o(1)     & o(1)      & a_{3}+o(1) & \cdots & o(1)\\
            \vdots & \vdots & \vdots & \ddots & \vdots\\
            o(1)      & o(1)      & o(1)      & \cdots & a_{n}+o(1)
        \end{pmatrix}.
    \]
    Since $P$ is invertible when $\lambda$ is sufficiently large, we can deduce from (\ref{P31_e}) that
   $$\textbf{c}=\textbf{0}.$$
   This completes the proof.
\end{proof}

Next, we need to compute the left hand side of \eqref{31}-\eqref{33}.

\begin{lemma}
     We have
     \begin{equation}\label{partial derivative_3}
         \begin{aligned}
               &\int_{\mathbb{R}^N}\left( -\Delta u_m+V(r,y'') u_m-|v_m|^{p-1}v_m\right)\frac{\partial Z_{\bar{r}, \bar{y}'', \lambda}}{\partial \lambda}+\left(-\Delta v_m+V(r,y'') v_m-|u_m|^{q-1}u_m \right)\frac{\partial Y_{\bar{r}, \bar{y}'', \lambda}}{\partial \lambda}\\
               &=m\left(-\frac{B_1}{\lambda^3}V(\bar{r},\bar{y}'')+\sum\limits^{m}_{j=2}\frac{B_2(N-2)}{\lambda^{N-1}|x_1-x_j|^{N-2}}+O\left(\frac{1}{\lambda^{3+\varepsilon}}\right)\right),\\
               &=m\left(-\frac{B_1}{\lambda^3}V(\bar{r},\bar{y}'')+\frac{B_3m^{N-2}}{\lambda^{N-1}}+O\left(\frac{1}{\lambda^{3+\varepsilon}}\right)\right).
         \end{aligned}
     \end{equation}
     for some positive constants $B_i,\,i=1,2,3$.
\end{lemma}
\begin{proof}
    Using $(u_m,v_m)=(Y_{\bar{r}, \bar{y}'', \lambda}+\phi_{\bar{r}, \bar{y}'', \lambda},Z_{\bar{r}, \bar{y}'', \lambda}+\psi_{\bar{r}, \bar{y}'', \lambda})$, we have
    \begin{align*}
        &\int_{\mathbb{R}^N}\left( -\Delta u_m+V(r,y'') u_m-|v_m|^{p-1}v_m\right)\frac{\partial Z_{\bar{r}, \bar{y}'', \lambda}}{\partial \lambda}+\left(-\Delta v_m+V(r,y'') v_m-|u_m|^{q-1}u_m \right)\frac{\partial Y_{\bar{r}, \bar{y}'', \lambda}}{\partial \lambda}\\
        =&\left\langle I_u(Y_{\bar{r}, \bar{y}'', \lambda},Z_{\bar{r}, \bar{y}'', \lambda}),\frac{\partial Y_{\bar{r}, \bar{y}'', \lambda}}{\partial \lambda}\right\rangle +\left\langle I_v(Y_{\bar{r}, \bar{y}'', \lambda},Z_{\bar{r}, \bar{y}'', \lambda}),\frac{\partial Z_{\bar{r}, \bar{y}'', \lambda}}{\partial \lambda}\right\rangle\\
        &+\left\{\left\langle-\Delta\phi+V(r,y'')\phi-pZ^{p-1}_{\bar{r}, \bar{y}'', \lambda}\psi,\frac{\partial Z_{\bar{r}, \bar{y}'', \lambda}}{\partial \lambda}\right\rangle +\left\langle-\Delta\psi+V(r,y'')\psi-q Y^{q-1}_{\bar{r}, \bar{y}'', \lambda}\phi,\frac{\partial Y_{\bar{r}, \bar{y}'', \lambda}}{\partial \lambda}\right\rangle\right\}\\
        &-\bigg\{\int\limits_{\R^N}\left(( Z_{\bar{r}, \bar{y}'', \lambda}+\psi)^{p}- Z_{\bar{r}, \bar{y}'', \lambda}^p-pZ^{p-1}_{\bar{r}, \bar{y}'', \lambda}\psi \right)\frac{\partial Z_{\bar{r}, \bar{y}'', \lambda}}{\partial \lambda}\\
        &\hspace{8em}+\left(| Y_{\bar{r}, \bar{y}'', \lambda}+\phi|^{q-1}(Y_{\bar{r}, \bar{y}'', \lambda}+\phi)- Y_{\bar{r}, \bar{y}'', \lambda}^q-qY^{q-1}_{\bar{r}, \bar{y}'', \lambda}\phi \right)\frac{\partial Y_{\bar{r}, \bar{y}'', \lambda}}{\partial \lambda} \bigg\}\\
        :=&\frac{\partial I(Y_{\bar{r}, \bar{y}'', \lambda},Z_{\bar{r}, \bar{y}'', \lambda})}{\partial \lambda}+I_1-I_2.
    \end{align*}
     
    Using \eqref{cV1}, \eqref{cV2}, Lemma \ref{B7} and the fact that
    $$-\Delta \frac{\partial Y_{\bar{r}, \bar{y}'', \lambda}}{\partial \lambda}  - p (Z_{\bar{r}, \bar{y}'',\lambda})^{p-1}\frac{\partial Z_{\bar{r}, \bar{y}'', \lambda}}{\partial \lambda}=0,$$
    we have
    \begin{equation*}
        \begin{aligned}
            I_1&= \left\langle -\Delta \frac{\partial Z_{\bar{r}, \bar{y}'', \lambda}}{\partial \lambda}  - q (Y_{\bar{r}, \bar{y}'',\lambda})^{q-1}\frac{\partial Y_{\bar{r}, \bar{y}'', \lambda}}{\partial \lambda}, \phi \right\rangle+\left\langle V(r,y'')\phi,\frac{\partial Z_{\bar{r}, \bar{y}'', \lambda}}{\partial \lambda}\right\rangle+\left\langle V(r,y'')\psi,\frac{\partial Y_{\bar{r}, \bar{y}'', \lambda}}{\partial \lambda}\right\rangle\\
            &=m\left\langle -\Delta Z_{1,h}  - q (Y_{\bar{r}, \bar{y}'',\lambda})^{q-1}Y_{1,h}, \phi \right\rangle+m\left\langle V(r, y'') \psi , Y_{1,h} \right\rangle+m\left\langle V(r, y'') \phi , Z_{1,h} \right\rangle+O\left(\frac{m}{\lambda^{3+\epsilon}}\right)\\
            &=O\left(\frac{m}{\lambda^{3+\epsilon}}\right).
        \end{aligned}
    \end{equation*}
    
    On the other hand, for $1<q<2$, by the formula
    \[
        (1+t)^q=1+qt+O(t^2),
    \]
    we can use a similar method as in Lemma \ref{N} to compute
    \begin{equation*}
        \begin{aligned}
           & \quad \left| \int\limits_{\R^N}\left(| Y_{\bar{r}, \bar{y}'', \lambda}+\phi|^{q-1}(Y_{\bar{r}, \bar{y}'', \lambda}+\phi)- Y_{\bar{r}, \bar{y}'', \lambda}^q-qY^{q-1}_{\bar{r}, \bar{y}'', \lambda}\phi
            \right)\frac{\partial Y_{\bar{r}, \bar{y}'', \lambda}}{\partial \lambda}\right|\\
            & \leq \frac{C}{\lambda} \int\limits_{\R^N} |\phi|^2Y_{\bar{r}, \bar{y}'', \lambda}^{q-2}\left|\frac{\partial Y_{\bar{r}, \bar{y}'', \lambda}}{\partial \lambda}\right|
            \leq \frac{C}{\lambda} \int\limits_{\R^N} Y_{\bar{r}, \bar{y}'', \lambda}^{q-1}|\phi|^2
            =O\left(\frac{m||\phi||^2_{*,1}}{\lambda}\right)=O\left(\frac{m}{\lambda^{3+\varepsilon}} \right).
        \end{aligned}
    \end{equation*}
    Similarly, for $q\geq 2$, we have
        \begin{equation*}
        \begin{aligned}
            &\quad \left| \int\limits_{\R^N}\left(| Y_{\bar{r}, \bar{y}'', \lambda}+\phi|^{q-1}(Y_{\bar{r}, \bar{y}'', \lambda}+\phi)- Y_{\bar{r}, \bar{y}'', \lambda}^q-qY^{q-1}_{\bar{r}, \bar{y}'', \lambda}\phi
            \right)\frac{\partial Y_{\bar{r}, \bar{y}'', \lambda}}{\partial \lambda}\right|
            \\
            &\leq  \frac{C}{\lambda} \int\limits_{\R^N} Y_{\bar{r}, \bar{y}'', \lambda}^{q-1}|\phi|^2 + Y_{\bar{r}, \bar{y}'', \lambda}|\phi|^{q}=O\left(\frac{m}{\lambda^{3+\varepsilon}} \right).
        \end{aligned}
    \end{equation*}
    Similarly, we can prove
    \[
        \left| \int\limits_{\R^N}\left(| Z_{\bar{r}, \bar{y}'', \lambda}+\psi|^{p-1}(Z_{\bar{r}, \bar{y}'', \lambda}+\psi)- Z_{\bar{r}, \bar{y}'', \lambda}^p-pZ^{p-1}_{\bar{r}, \bar{y}'', \lambda}\psi
        \right)\frac{\partial Z_{\bar{r}, \bar{y}'', \lambda}}{\partial \lambda}\right| =O\left(\frac{m}{\lambda^{3+\varepsilon}}\right).
    \]
    Therefore,
    $$|I_2|=O\left(\frac{m}{\lambda^{3+\varepsilon}}\right).$$
    Therefore, we have
    \begin{equation*}
        \begin{aligned}
             \frac{\partial I(u_m,v_m)}{\partial \lambda}=\frac{\partial I(Y_{\bar{r}, \bar{y}'', \lambda},Z_{\bar{r}, \bar{y}'', \lambda})}{\partial \lambda}+O\bigg(\frac{m}{\lambda^{3+\varepsilon}}\bigg).
        \end{aligned}
    \end{equation*}
    Using Lemma \ref{expansion_lambda}, we obtain the desired result.
\end{proof}

Using the same arguments as above and Lemma \ref{expansion_2}, we can also prove
\begin{align}
        &\left\langle I_u(u_m,v_m),\frac{\partial Y_{\bar{r}, \bar{y}'', \lambda}}{\partial \bar{r}}\right\rangle+\left\langle I_v(u_m,v_m),\frac{\partial Z_{\bar{r}, \bar{y}'', \lambda}}{\partial \bar{r}}\right\rangle \label{partial derivative_1}\\
        &=m\left( \frac{B_1}{\lambda^2}\frac{\partial V(\bar{r},\bar{y}'')}{\partial \bar{r}}+\sum\limits_{j=2}^m\frac{B_2}{\bar{r}\lambda^{N-1}|x_1-x_j|^{N-2}}+O\left(\frac{1}{\lambda^{1+\varepsilon}}\right)\right). \nonumber
\end{align}
and
\begin{equation}\label{partial derivative_2}
    \begin{aligned}
        \big\langle I_u(u_m,v_m),\frac{\partial Y_{\bar{r}, \bar{y}'', \lambda}}{\partial \bar{y}''_k}\big\rangle+\big\langle I_v(u_m,v_m),\frac{\partial Z_{\bar{r}, \bar{y}'', \lambda}}{\partial \bar{y}_k''}\big\rangle
        =m\left( \frac{B_1}{\lambda^2}\frac{\partial V(\bar{r},\bar{y}'')}{\partial \bar{y_k}''}+O\bigg(\frac{1}{\lambda^{1+\varepsilon}}\bigg)\right),\,k=3,\cdots,N.
    \end{aligned}
\end{equation}

Next we use the local Pohozaev identities, generated from translation and scaling respectively, to estimate (\ref{31}) and (\ref{32}).
\begin{lemma}
    Equations (\ref{31}) and (\ref{32}) are equivalent to
    \begin{equation}\label{lem33_1}
    \begin{aligned}
        \int_{\textbf{B}_\rho}\frac{\partial V(r,y'')}{\partial y_i}u_m v_m=O\left(\int_{\partial \textbf{B}_\rho}|\nabla\phi|^2+|\phi|^2+|\phi|^{q+1}+|\nabla \psi|^2+|\psi|^2+|\psi|^{p+1}\right),\,i=3,\cdots,N,
    \end{aligned}
    \end{equation}
    and
    \begin{equation}\label{lem33_2}
        \begin{aligned}
         \int_{\textbf{B}_\rho}\frac{1}{r}\frac{\partial(r^2 V(r,y''))}{\partial r}u_mv_m
         =o\bigg(\frac{m}{\lambda^2}\bigg)+O\left(\int_{\partial \textbf{B}_\rho}|\nabla\phi|^2+|\phi|^2+|\phi|^{q+1}+|\nabla \psi|^2+|\psi|^2+|\psi|^{p+1}\right).
    \end{aligned}
    \end{equation}
\end{lemma}

\begin{proof}

    In the rest of the proof, we denote $(u_m,v_m)$ as $(u,v)$ for convenience. Let
    \[
        H(y,u,v)=-V(y)uv+\frac{|v|^{p+1}}{p+1}+\frac{|u|^{q+1}}{q+1},
    \]
    then (\ref{32}) is equivalent to
    \begin{equation}\label{poho_scal}
        \int_{\textbf{B}_\rho} \left( -\Delta u\frac{\partial v}{\partial y_i}-\Delta v\frac{\partial u}{\partial y_i} \right) = \int_{\textbf{B}_\rho} \left( \frac{\partial}{\partial y_i} H(y,u(y),v(y))-\frac{\partial V}{\partial y_i}(r,y'')uv \right).
    \end{equation}
    Integrating (\ref{poho_scal}) by parts yields
    \begin{equation*}
        \begin{aligned}
                \int_{\textbf{B}_\rho} \frac{\partial V}{\partial y_i}(r,y'')uv &=\int_{\partial \textbf{B}_{\rho}} \frac{\partial u}{ \pa\nu}\frac{\partial v}{ \pa y_i}+\frac{\partial v}{ \pa\nu}\frac{\partial u}{ \pa y_i}-\nabla u\cdot\nabla v \nu_i+\int_{\partial \textbf{B}_{\rho}}H(y,u,v)\nu_i\\
                &=O\left(\int_{\partial \textbf{B}_\rho}|\nabla\phi|^2+|\phi|^2+|\phi|^{q+1}+|\nabla \psi|^2+|\psi|^2+|\psi|^{p+1}\right),
        \end{aligned}
    \end{equation*}
    since $u=\phi$ on $\pa \textbf{B}_\rho$, which is the first equality in this Lemma.

    Similarly, integrating (\ref{31}) by parts yields
    \begin{equation}\label{poho_trans}
        \begin{aligned}
            & \quad (2-N)\int_{\textbf{B}_\rho} \nabla u\cdot \nabla v -\int_{\textbf{B}_\rho}\left( NV(y)+\langle y, \nabla V(y)\rangle\right)uv +\frac{N}{p+1}\int_{\textbf{B}_\rho}v^{p+1}+ \frac{N}{q+1} \int_{\textbf{B}_{\rho}}u^{q+1}\\
            &=\int_{\partial \textbf{B}_{\rho}} \bigg( \frac{\pa u}{\pa \nu}(y\cdot \nabla v)+\frac{\pa v}{\pa \nu}(y\cdot \nabla u) \bigg) -\int_{\partial \textbf{B}_{\rho}}(y\cdot\nu)\nabla u\cdot\nabla v+
            \int_{\partial \textbf{B}_{\rho}}(y\cdot\nu)H(y,u,v)\\
            &=O\left(\int_{\partial \textbf{B}_\rho}|\nabla\phi|^2+|\phi|^2+|\phi|^{q+1}+|\nabla \psi|^2+|\psi|^2+|\psi|^{p+1}\right).
        \end{aligned}
    \end{equation}
    From Proposition \ref{existence2}, we obtain
    \begin{align}
            & \quad \int_{\textbf{B}_\rho} \nabla u\cdot \nabla v + \int_{\textbf{B}_\rho}V(y)uv-\int_{\textbf{B}_\rho}v^{p+1}=\int_{\pa \textbf{B}_\rho} \frac{\pa u}{\pa \nu}v +\sum\limits_{l=1}^{N} c_l \sum\limits_{j=1}^{m} \int_{\R^N}pZ_{x_j,\lambda}^{p-1}Z_{j,l}v \label{lemma3_1} \\
            &=O\left( \int_{\pa \textbf{B}_{\rho}}|\nabla \phi|^2+\psi^2 \right)+O\left(\sum\limits_{l=1}^{N}c_l\cdot m\lambda^{n_l}||\psi||_{*,2} \right)+\sum\limits_{l=1}^{N} c_l \sum\limits_{j=1}^{m} \int_{\R^N}pZ_{x_j,\lambda}^{p-1}Z_{\bar{r}, \bar{y}'', \lambda}Z_{j,l}, \nonumber
    \end{align}
    and
    \begin{equation}\label{lemma3_2}
        \begin{aligned}
            \int_{\textbf{B}_\rho} \nabla u\cdot \nabla v&+ \int_{\textbf{B}_\rho}V(y)uv-\int_{\textbf{B}_\rho}u^{q+1}=\int_{\pa \textbf{B}_\rho} \frac{\pa v}{\pa \nu}u +\sum\limits_{l=1}^{N} c_l \sum\limits_{j=1}^{m} \int_{\R^N}qY_{x_j,\lambda}^{q-1}Y_{j,l}u\\
            &=O\left(\int_{\pa \textbf{B}_{\rho}}|\nabla \psi|^2+\phi^2 \right)
            +O\left(\sum\limits_{l=1}^{N}c_l\cdot m\lambda^{n_l}||\phi||_{*,1} \right)+\sum\limits_{l=1}^{N} c_l \sum\limits_{j=1}^{m} \int_{\R^N}qY_{x_j,\lambda}^{q-1}Y_{\bar{r}, \bar{y}'', \lambda}Y_{j,l},
        \end{aligned}
    \end{equation}
    where we have used
    \begin{equation*}
        \begin{aligned}
             \int\limits_{\R^N}  \sum\limits_{j=1}^{m} q Y_{x_j,\lambda}^{q-1} Y_{j,l} \phi
             = O\left( m\lambda^{n_l}||\phi||_{*,1}\int\limits_{\R^N} Y_{x_1,\lambda}^{q} \sum\limits_{j=1}^{m}\frac{\lambda^{\frac{N}{q+1}}}{(1+\lambda|y-x_{j}|)^{\frac{N}{q+1}+\tau}}\ \dy\right)
             =O(m\lambda^{n_l}||\phi||_{*,1}),
        \end{aligned}
    \end{equation*}
    and similarly,
    \[
        \int\limits_{\R^N} \sum\limits_{j=1}^{m} p Z_{x_j,\lambda}^{p-1} Z_{j,l}\psi=O(m\lambda^{n_l}||\psi||_{*,2}).
    \]

    Next, we estimate $c_l$ and subsequently the terms that involve $c_l$ in (\ref{lemma3_1}) and (\ref{lemma3_2}). Similarly to the argument in Proposition \ref{prop31}, using (\ref{linear-equation2}), \eqref{partial derivative_3}, \eqref{partial derivative_1} and \eqref{partial derivative_2}, we obtain the following algebra equations for $c_l$:
    \begin{equation*}
        \begin{aligned}
         \sum_{l=1}^{N}\sigma_{1l}' c_l=\big\langle I_u(u_m,v_m),\frac{\partial Y_{\bar{r}, \bar{y}'', \lambda}}{\partial \lambda}\big\rangle+\big\langle I_v(u_m,v_m),\frac{\partial Z_{\bar{r}, \bar{y}'', \lambda}}{\partial \lambda}\big\rangle=O(\frac{m}{\lambda^{3}}),\\
        \sum_{l=1}^{N}\sigma_{2l}' c_l=\big\langle I_u(u_m,v_m),\frac{\partial Y_{\bar{r}, \bar{y}'', \lambda}}{\partial \bar{r}}\big\rangle+\big\langle I_v(u_m,v_m),\frac{\partial Z_{\bar{r}, \bar{y}'', \lambda}}{\partial \bar{r}}\big\rangle=O(\frac{m}{\lambda^{1+\varepsilon}}),\\
             \sum_{l=1}^{N}\sigma_{kl}' c_l=\big\langle I_u(u_m,v_m),\frac{\partial Y_{\bar{r}, \bar{y}'', \lambda}}{\partial \bar{y}''_k}\big\rangle+\big\langle I_v(u_m,v_m),\frac{\partial Z_{\bar{r}, \bar{y}'', \lambda}}{\partial \bar{y}_k''}\big\rangle=O(\frac{m}{\lambda^{1+\varepsilon}}),\,\,k=3,\cdots,N,
        \end{aligned}
    \end{equation*}
    where
    $$\sigma_{kl}'=\Bigg( \int_{\R^N}\big(\sum\limits_{j=1}^{m} q Y_{x_j,\lambda}^{q-1} Y_{j,l},  \sum\limits_{j=1}^{m} p Z_{x_j,\lambda}^{p-1} Z_{j,l}\big)\cdot W_k' \Bigg)=\delta_{kl}m\lambda^{2n_k}\bar{a}_k+o(m\lambda^{n_l+n_k}),
    \,\bar{a}_k\neq 0,\, k=1,\cdots,N,$$
    and
    \begin{equation*}
        \begin{aligned}
           W_k'=\begin{cases}
                  (\frac{\partial Y_{\bar{r}, \bar{y}'', \lambda}}{\partial \lambda},\frac{\partial Z_{\bar{r}, \bar{y}'', \lambda}}{\partial \lambda}),\quad k=1,\\
                 (\frac{\partial Y_{\bar{r}, \bar{y}'', \lambda}}{\partial \bar{r}},\frac{\partial Z_{\bar{r}, \bar{y}'', \lambda}}{\partial \bar{r}}),\quad k=2,\\
                  (\frac{\partial Y_{\bar{r}, \bar{y}'', \lambda}}{\partial \bar{y}''_k},\frac{\partial Z_{\bar{r}, \bar{y}'', \lambda}}{\partial \bar{y}''_k}),\quad k=3,\cdots,N,
            \end{cases}
        \end{aligned}
    \end{equation*}
    which yields
    \begin{align}\label{3.211}
         c_1=\frac{1}{\lambda^{2n_1}}O\bigg(\frac{1}{\lambda^{3}}\bigg)=O\bigg(\frac{1}{\lambda}\bigg),
    \end{align}
    and
    \begin{align}\label{3.212}
        c_i=\frac{1}{\lambda^{2n_i}}O\bigg(\frac{1}{\lambda^{1+\varepsilon}}\bigg)=O\bigg(\frac{1}{\lambda^{3+\varepsilon}}\bigg),\quad i=2,\cdots,N.
    \end{align}
    Next, we can compute that
    \begin{equation}\label{cofficient_2}
        \begin{aligned}
            &\sum\limits_{j=1}^{m}\int_{\R^N}pZ_{x_j,\lambda}^{p-1}Z_{\bar{r}, \bar{y}'', \lambda}Z_{j,l}+qY_{x_j,\lambda}^{q-1}Y_{\bar{r}, \bar{y}'', \lambda}Y_{j,l}\\
            =&m\int_{\R^N}\left(\frac{p}{p+1}\frac{\partial Z^{p+1}_{x_1,\lambda}}{\partial \Box_l}+\frac{q}{q+1}\frac{\partial Y^{q+1}_{x_1,\lambda}}{\partial \Box_l} \right)+o(m\lambda^{n_l})\\
            =&o(m\lambda^{n_l}),
        \end{aligned}
    \end{equation}
    where $\Box_l$ denotes $\lambda$ if $l=1$, $\bar{r}$ if $l=2$ and $\bar{y}_l''$ if $l\geq3$.

    Inserting (\ref{lemma3_1})-(\ref{cofficient_2}) into (\ref{poho_trans}), we obtain
    \begin{align}\label{poho_trans_1}
        \int_{\textbf{B}_\rho}(2V(y)+\langle y,\nabla V(y)\rangle)uv
         =o\left(\frac{m}{\lambda^2}\right)+O\left(\int_{\partial \textbf{B}_\rho}|\nabla\phi|^2+|\phi|^2+|\phi|^{q+1}+|\nabla \psi|^2+|\psi|^2+|\psi|^{p+1}\right).
    \end{align}
    Note that
    \begin{equation*}
        \begin{aligned}
            \int_{\textbf{B}_\rho}\langle y'',\nabla_{y''} V(y)\rangle uv&=\sum\limits_{k=3}^N\bar{y}_k''\int_{\textbf{B}_\rho}\frac{\partial V(r,y'')}{\partial {y}_k''} uv+\int_{\textbf{B}_\rho}\langle y''-\bar{y}'',\nabla_{y''} V(y)\rangle uv\\
            &=O\left(\int_{\partial \textbf{B}_\rho}|\nabla\phi|^2+|\phi|^2+|\phi|^{q+1}+|\nabla \psi|^2+|\psi|^2+|\psi|^{p+1}\right)+o\left(\frac{m}{\lambda^2}\right),
        \end{aligned}
    \end{equation*}
    then we can rewrite (\ref{poho_trans_1}) as
    \[
        \int_{\textbf{B}_\rho}\left(2V(y)+r\frac{\partial V(r,y'')}{\partial r} \right)uv
         =o\left(\frac{m}{\lambda^2} \right)+O\left(\int_{\partial \textbf{B}_\rho}|\nabla\phi|^2+|\phi|^2+|\phi|^{q+1}+|\nabla \psi|^2+|\psi|^2+|\psi|^{p+1}\right),
    \]
    which is the second equality in this Lemma.
\end{proof}

Then we need to estimate boundary integral terms in \eqref{lem33_1} and \eqref{lem33_2}.
\begin{lemma}\label{par_1}
    $$\int_{\textbf{B}_{4\delta}\backslash \textbf{B}_{3\delta}}
    |\phi|^2+|\phi|^{q+1}+|\psi|^2 +|\psi|^{p+1}=o\left(\frac{m}{\lambda^2}\right).$$
\end{lemma}
\begin{proof}
From Proposition \ref{existence2}, we obtain that
    \begin{equation*}
        \begin{aligned}
            &\quad \int_{\textbf{B}_{4\delta}\backslash \textbf{B}_{3\delta}} |\psi|^{p+1} \leq \big|\big|\psi\big|\big|_{*}^{p+1}\int_{D_{4\delta}\backslash D_{3\delta}}
            \left( \sum\limits_{j=1}^{m}\dfrac{\lambda^{\frac{N}{p+1}}}{(1+\lambda|z-x_j|)^{\frac{N}{p+1}+\tau}} \right)^{p+1}\\
            &\leq m\big|\big|\psi\big|\big|_{*}^{p+1}\int_{\{\textbf{B}_{4\delta}\backslash \textbf{B}_{3\delta}\}\cap \Omega_1} \lambda^N\left( \dfrac{1}{(1+\lambda|z-x_1|)^{\frac{N}{p+1}+\tau}}+\dfrac{1}{(1+\lambda|z-x_1|)^{\frac{N}{p+1}}}\cdot \sum\limits_{j=1}^{m}\frac{1}{(\lambda|x_1-x_j|)^{\tau}}\right)^{p+1}\\
            &\leq C\cdot m\big|\big|\psi\big|\big|_{*}^{p+1}\int_{B_{4\delta \lambda}(0)\backslash B_{3\delta \lambda}(0)} \frac{1}{(1+|z|)^{N}} \ \dz \\
            &\leq O\left(\frac{m \log \lambda}{\lambda^{p+1+\varepsilon}}\right).
        \end{aligned}
    \end{equation*}
    Since $\textbf{B}_{4\delta}\backslash \textbf{B}_{3\delta}$ is a bounded domain, we have
    \[
        \int_{\textbf{B}_{4\delta}\backslash \textbf{B}_{3\delta}}|\psi|^2\leq C(\delta) \left(\int_{D_{4\delta}\backslash D_{3\delta}}|\psi|^{p+1}\right)^{\frac{2}{p+1}}=O\left(\frac{(m\log \lambda)^{\frac{2}{p+1}}}{\lambda^{2+\varepsilon}}\right)=o\left(\frac{m}{\lambda^2}\right).
    \]
    Similar estimates hold for $\phi$, this completes the proof.
\end{proof}

We also have the following point-wise estimate for the error term.
\begin{lemma}\label{par_2}
    $$\big|\big| (\phi,\psi) \big|\big|_{C^1(\textbf{B}_{4\delta}\backslash \textbf{B}_{3\delta})}=O\left(\frac{1}{\lambda^{1+\varepsilon}}\right).$$
\end{lemma}
\begin{proof}
    For any $z\in \textbf{B}_{4\delta}\backslash \textbf{B}_{3\delta}$, we obtain
    \begin{equation*}
        \begin{aligned}
            |\psi(z)|&\leq \big|\big|\psi\big|\big|_{*} \sum\limits_{j=1}^{m}\dfrac{\lambda^{\frac{N}{p+1}}}{(1+\lambda|z-x_j|)^{\frac{N}{p+1}+\tau}}\leq  C(\delta)\big|\big|\psi\big|\big|_{*}\sum\limits_{j=1}^{m}\dfrac{1}{(1+\lambda|z-x_j|)^{\tau}}\\
            &\leq C(\delta)\big|\big|\psi\big|\big|_{*}=O\left(\frac{1}{\lambda^{1+\varepsilon}}\right).
        \end{aligned}
    \end{equation*}
    Similarly, we have
    \[
        |\phi(z)|\leq O\left(\frac{1}{\lambda^{1+\varepsilon}} \right),\quad\text{in } \textbf{B}_{4\delta}\backslash \textbf{B}_{3\delta}.
    \]
    Since $(\phi,\psi)$ satisfies \eqref{linear-equation2}, combine Lemma \ref{N} Lemma \ref{l} and $L^p$ estimate, we could deduce that for any $r>1$,
    \begin{equation*}
        \begin{aligned}
            &\quad \big|\big| \phi \big|\big|_{W^{2,r}(\textbf{B}_{4\delta}\backslash \textbf{B}_{3\delta})}\\
            &\leq C \big|\big| \phi \big|\big|_{L^{\infty}(\textbf{B}_{4\delta}\backslash \textbf{B}_{3\delta})}
            +C\left|\left|
             p(Z_{\bar{r}, \bar{y}'', \lambda})^{p-1} \psi-V\phi+N_1(\psi)+l_1+\sum\limits_{l=1}^{N} c_l\sum\limits_{j=1}^{m} p Z_{x_j,\lambda}^{p-1} Z_{j,l}
            \right|\right|_{L^{\infty}(\textbf{B}_{4\delta}\backslash \textbf{B}_{3\delta})}\\
            &\leq C \big|\big| \phi \big|\big|_{L^{\infty}(\textbf{B}_{4\delta}\backslash \textbf{B}_{3\delta})}
            + C\left(||N(\phi,\psi)||_{**} +||(l_1, l_2)||_{**}\right)+O\left(\frac{1}{\lambda^{2}}\right)\\
            &=O\left(\frac{1}{\lambda^{1+\varepsilon}} \right).
        \end{aligned}
    \end{equation*}
    Similar estimate holds for $\psi$. Therefore, we conclude that
    $$\big|\big| (\phi,\psi) \big|\big|_{C^1(\textbf{B}_{4\delta}\backslash \textbf{B}_{3\delta})}=O\left(\frac{1}{\lambda^{1+\varepsilon}}\right).$$
\end{proof}

As a result, we can deduce from Lemma \ref{par_1} and Lemma \ref{par_2} that there exists a $\rho \in (3\delta,4\delta)$, such that
$$\int_{\partial \textbf{B}_\rho}|\nabla\phi|^2+|\phi|^2+|\phi|^{q+1}+|\nabla \psi|^2+|\psi|^2+|\psi|^{p+1}=O\left(\frac{m}{\lambda^{2+\varepsilon}}\right).$$

\begin{lemma}\label{lemma3.4}
    For any $C^1$ function $g(r,y'')$, it holds that
    \begin{equation*}
        \begin{aligned}
            \int_{\textbf{B}_\rho} g(r,y'')u_mv_m=m\left( \frac{1}{\lambda^2}g(\bar{r},\bar{y}'')\int_{\R^N}U_{0,1}V_{0,1}+o(\frac{1}{\lambda^2})\right).
        \end{aligned}
    \end{equation*}
\end{lemma}
\begin{proof}
 For simplicity, we drop the subscript $m$.
    Since $u=Y_{\bar{r}, \bar{y}'', \lambda}+\phi$, $v=Z_{\bar{r}, \bar{y}'', \lambda}+\psi$, we have
    \begin{equation*}
        \begin{aligned}
            \int_{\textbf{B}_\rho} g(r,y'')uv=
            \int_{\textbf{B}_\rho} g(r,y'')Y_{\bar{r}, \bar{y}'', \lambda}Z_{\bar{r}, \bar{y}'', \lambda}
            +\int_{\textbf{B}_\rho} g(r,y'')\big( Y_{\bar{r}, \bar{y}'', \lambda}\psi+Z_{\bar{r}, \bar{y}'', \lambda}\phi\big)
            +\int_{\textbf{B}_\rho} g(r,y'')\phi \psi.
        \end{aligned}
    \end{equation*}
    Similarly to the estimation of $I_1$ in Lemma \ref{expansion_lambda}, we have
     \begin{equation*}
         \begin{aligned}
             \int_{\textbf{B}_\rho} g(r,y'')Y_{\bar{r}, \bar{y}'', \lambda}Z_{\bar{r}, \bar{y}'', \lambda}&=\int_{\textbf{B}_\rho} g(\bar{r},\bar{y}'')Y_{\bar{r}, \bar{y}'', \lambda}Z_{\bar{r}, \bar{y}'', \lambda}
             +\int_{\textbf{B}_\rho} \big(g(r,y'')-g(\bar{r},\bar{y}'')\big)Y_{\bar{r}, \bar{y}'', \lambda}Z_{\bar{r}, \bar{y}'', \lambda}\\
             &=\frac{m}{\lambda^2}g(\bar{r},\bar{y}'')\int_{\R^N}U_{0,1}V_{0,1}+o\left(\frac{m}{\lambda^2} \right).
         \end{aligned}
     \end{equation*}
    From \eqref{H_2} and \eqref{K_2}, we get
    \begin{equation}\label{YZ}
        \begin{aligned}
            ||Y_{\bar{r}, \bar{y}'', \lambda}||_{**,1}\leq \frac{C}{\lambda^{1+\varepsilon}},\quad
            ||Z_{\bar{r}, \bar{y}'', \lambda}||_{**,2}\leq \frac{C}{\lambda^{1+\varepsilon}}
        \end{aligned}
    \end{equation}
    In addition, similar to \eqref{I_32_2}, we can use \eqref{phi_psi_c} and \eqref{YZ} to obtain that in $\Omega_1$,
     \begin{equation*}
         \begin{aligned}
             \left| \int_{\textbf{B}_\rho} g(r,y'') Z_{\bar{r}, \bar{y}'', \lambda}\phi\right|
             &\leq \frac{C||\phi||_{*,1}}{\lambda^{1+\varepsilon}}\int_{\textbf{B}_\rho} \sum\limits_{j=1}^{m} \frac{\lambda^{\frac{N}{q+1}+2}}{(1+\lambda|y-x_j|)^{\frac{N}{q+1}+\tau+2}}\sum\limits_{j=1}^{m} \frac{\lambda^{\frac{N}{p+1}}}{(1+\lambda|y-x_j|)^{\frac{N}{p+1}+\tau}}\\
             &\leq C\frac{m||\phi||_{*,1}}{\lambda^{1+\varepsilon}} \int_{\textbf{B}_\rho\cap \Omega_1} \frac{\lambda^{\frac{N}{q+1}+2}}{(1+\lambda|y-x_1|)^{\frac{N}{q+1}+2}}\cdot \frac{\lambda^{\frac{N}{p+1}}}{(1+\lambda|y-x_ 1|)^{\frac{N}{p+1}}}\\
             &=O\left(\frac{m||\phi||_{*,1}\log \lambda}{\lambda^{1+\varepsilon}} \right)
             =O\left(\frac{m}{\lambda^{2+\varepsilon}} \right).
         \end{aligned}
     \end{equation*}
     We can get a similar estimate of $\int_{\textbf{B}_\rho} g(r,y'') Y_{\bar{r}, \bar{y}'', \lambda}\psi$. Thus, we have
     \[
         \left|\int_{\textbf{B}_\rho} g(r,y'')\big( Y_{\bar{r}, \bar{y}'', \lambda}\psi+Z_{\bar{r}, \bar{y}'', \lambda}\phi\big)\right|\leq  O\left(\frac{m}{\lambda^{2+\varepsilon}}\right).
     \]
    In a similar way, we can prove
    \begin{equation*}
        \begin{aligned}
            \int_{\textbf{B}_\rho} g(r,y'')\phi \psi&\leq Cm||\phi||_{*,1}||\psi||_{*,2}\int_{D_\rho\cap \Omega_1} \frac{\lambda^{N-2}}{(1+\lambda|y-x_1|)^{N-2}}\\
            &=O(m||\phi||_{*,1}||\psi||_{*,2})=O\left(\frac{m}{\lambda^{2+\varepsilon}} \right).
        \end{aligned}
    \end{equation*}
     Combine Lemma \ref{par_1} and the above estimates, we have
     \begin{equation*}
         \begin{aligned}
             \int_{\textbf{B}_\rho} g(r,y'')u_mv_m=m\left( \frac{1}{\lambda^2}g(\bar{r},\bar{y}'')\int_{\R^N}U_{0,1}V_{0,1}+o(\frac{1}{\lambda^2})\right).
         \end{aligned}
     \end{equation*}
\end{proof}

\begin{proof}[\bf{Proof of Theorem \ref{Thm2}}]
Apply Lemma \ref{lemma3.4}, we obtain that \eqref{lem33_1} and \eqref{lem33_2} are equivalent to
 \begin{equation*}
    \begin{aligned}
        m\left(\frac{1}{\lambda^2}\frac{\partial V(\bar{r},\bar{y}'')}{\partial \bar{y}_i}\int_{\R^N}U_{0,1} V_{0,1}+o(\frac{1}{\lambda^2})\right)=o\left(\frac{m}{\lambda^2}\right),\,i=1,\cdots,N,
    \end{aligned}
    \end{equation*}
    and
    \begin{equation*}
        \begin{aligned}
         m\left(\frac{1}{\lambda^2}
         \frac{1}{\bar{r}}\frac{\partial(\bar{r}^2 V(\bar{r},\bar{y}''))}{\partial \bar{r}}
         \int_{\R^N}U_{0,1} V_{0,1}+o(\frac{1}{\lambda^2})\right)=o\left(\frac{m}{\lambda^2}\right).
    \end{aligned}
    \end{equation*}
    By virtue of Proposition \ref{prop31}, we need to find $(\lambda,\bar{r},\bar{y}'')$ such that
    \begin{equation}\label{equ_to_rylambda}
    \begin{cases}
        \dfrac{\partial(\bar{r}^2 V(\bar{r},\bar{y}''))}{\partial \bar{r}}=o(1), \\ 
        \dfrac{\partial V(\bar{r},\bar{y}'')}{\partial \bar{y}_i}=o(1),\hspace{1em}i=3,\cdots,N, \\ 
        -\dfrac{B_1}{\lambda^3}V(\bar{r},\bar{y}'')+\dfrac{B_3m^{N-2}}{\lambda^{N-1}}=O\left(\dfrac{1}{\lambda^{3+\varepsilon}}\right).
    \end{cases}
    \end{equation}

Let $\lambda=tm^{\frac{N-2}{N-4}}$, for $t\in [L_0,L_1]$. Then, the third equation of \eqref{equ_to_rylambda} is equivalent to
\begin{align}\label{equ_to_lambda_2}
-\frac{B_1}{t^3}V(\bar{r},\bar{y}'')+\frac{B_3}{t^{N-1}}=o(1),\hspace{1em}t\in [L_0,L_1].
\end{align}
Let
\[
    \bar{F}(t, \bar{r}, \bar{y}'')=\Big( \nabla_{\bar{r},\bar{y}''} \big(\bar{r}^2V(\bar{r},\bar{y}'')\big),-\frac{B_1}{t^3}V(\bar{r},\bar{y}'')+\frac{B_3}{t^{N-1}}\Big).
\]
Then
\[
    deg\Big(\bar{F}(t, \bar{r}, \bar{y}''),\,[L_0,L_1]\times B_{\theta}\big( (r_0,y''_0)\big)\Big) =-deg\Big(\nabla_{\bar{r},\bar{y}''} \big(\bar{r}^2V(\bar{r},\bar{y}'')\big),\, B_{\theta}\big( (r_0,y''_0)\big)\Big)\neq0.
\]
So, \eqref{equ_to_rylambda} have a solution $t_m\in [L_0,L_1],\, (\bar{r}_m,\bar{y}''_m)\in B_{\theta}\big( (r_0,y''_0)\big)$. This completes the proof.

\end{proof}

\appendix

\section{Asymptotic behavior and non-degeneracy of Ground State }

 The following sharp asymptotic behavior and the non-degeneracy of the ground state $(U_{0,1}, V_{0,1})$ of \eqref{lane-embden} are important.

\begin{lemma}[Hulshof and Van der Vorst \cite{HV}] \label{L1}
 Assume that $p\leq \frac{N+2}{N-2} \leq q.$ There exist some positive constants $a=a_{N,p}$ and $b=b_{N,p}$ depending only on $N$ and $p$ such that
$$ \lim\limits_{r\to \infty} r^{N-2} V_{0,1}(r) =b ;$$
while
\begin{equation*}
\begin{cases}
    \lim\limits_{r\to\infty} r^{(N-2)p-2}U_{0,1}(r) =a,  \;\; &\hbox{if } p<\frac{N}{N-2},\\
    \lim\limits_{r\to\infty} \frac{r^{N-2}}{\ln r}U_{0,1}(r) =a,  \;\; &\hbox{if } p=\frac{N}{N-2},\\
    \lim\limits_{r\to\infty} r^{N-2}U_{0,1}(r) =a,  \;\; &\hbox{if } p>\frac{N}{N-2}.
\end{cases}
\end{equation*}
Furthermore, in the last case, we have $b^p=a( (N-2)p-2  )(N-(N-2)p).$
\end{lemma}

\begin{lemma}[Kim and Moon \cite{KM}]\label{L3}
There exists a constant $C > 0$ depending only on $N$ and $p$ such that
\begin{equation}\label{V10est}
    \left|V_{0,1}(r) - \frac{b_{N,p}}{r^{N-2}}\right| \le \frac{C}{r^N}.
\end{equation}
Besides,
\begin{equation}\label{U10est}
\begin{cases}
    \displaystyle \left|U_{0,1}(r) - \frac{a_{N,p}}{r^{N-2}}\right| \le \frac{C}{r^{N-2+\kappa_0}} &\text{if } p \in (\frac{N}{N-2}, \frac{N+2}{N-2}], \\
    \displaystyle \left|U_{0,1}(r) - \frac{a_{N,p} \log r}{r^{N-2}}\right| \le \frac{C}{r^{N-2}} &\text{if } p = \frac{N}{N-2}, \\
    \displaystyle \left|U_{0,1}(r) - \frac{a_{N,p}}{r^{p(N-2)-2}}\right| \le \frac{C}{r^{p(N-2)-2+\kappa_1}} &\text{if } p \in (\frac{2}{N-2}, \frac{N}{N-2}),
\end{cases}
\end{equation}
where $\kappa_0 := p(N-2)-N > 0$ and $\kappa_1$ is any number in $(0, \min\{N-p(N-2),2(p+1)\})$.
\end{lemma}

In fact, this asymptotic behavior of $(U_{0,1},V_{0,1})$ has been refined by the first and third author recently, see \cite{GHPY}.

\begin{lemma}[Frank, Kim and Pistoia \cite{FKP21}]\label{L2}
Set
   $$(\Psi_{0,1}^0,\Phi_{0,1}^0) = \left(  y \cdot \nabla U_{0,1} +\frac{NU_{0,1}}{q+1},\; y \cdot \nabla V_{0,1}+\frac{NV_{0,1}}{p+1}  \right)$$
and
   $$ (\Psi_{0,1}^l, \Phi_{0,1}^l) = (\partial_l U_{0,1}, \partial_l V_{0,1} ),\;\; \hbox{for }\ \ l=1,\cdots,N. $$
Then the space of solutions to the linear system
   \begin{equation}\label{8}
   \begin{cases}
   -\Delta \Psi =pV_{0,1}^{p-1} \Phi,\;\;\; \hbox{in } \mathbb R^N,\\
   -\Delta \Phi =qU_{0,1}^{q-1} \Psi,\;\;\; \hbox{in } \mathbb R^N,\\
   (\Psi,\Phi)\in \dot{W}^{2,\frac{p+1}{p}}(\mathbb R^N) \times \dot{W}^{2,\frac{q+1}{q}}(\mathbb R^N),
   \end{cases}
   \end{equation}
 is spanned by
   $$ \left\{  (\Psi_{0,1}^0,\Phi_{0,1}^0), (\Psi_{0,1}^1,\Phi_{0,1}^1)  ,\cdots, (\Psi_{0,1}^N,\Phi_{0,1}^N)    \right\} .$$
   \end{lemma}

\section{Estimate of the approximate solution}

In this section, we estimate $Y_{\bar{r},\bar{y}'',\lambda}^*(y)$, which is the unique solution of \eqref{ansatzes_2}. The basic idea is the same as \cite{GKPY}. However, since we are at a different setting, we give a complete proof for the reader's convenience. Recall that
\[
    \varphi (y) = Y_{\bar{r},\bar{y}'',\lambda}^*(y) - \sum_{j=1}^m U_{x_j, \lambda}(y).
\]
Then, we have the following expansion of $\varphi$:

\begin{lemma}
\label{nnl2-19-1}
Assume that $N \ge 5$. For $m \in \N$ large enough, there is a positive constant $B_0>0$, such that
\begin{equation}\label{nn1-19-1}
    \varphi(y)=\sum\limits_{j=2}^m\frac{B_{0}\lambda^{\frac{N}{q+1}}}{(\lambda|x_j-x_1|)^{N-2}}w(\lambda(y-x_1))+O\left(\frac{\lambda^{\frac{N}{q+1}}}{\lambda^{2+\varepsilon}}\right), \quad y \in \Omega_1
\end{equation}
uniformly in $\{(\bar{r},\bar{y}'',t):t\in [L_0,L_1],(\bar{r},\bar{y}'')\in B_{\theta}\big( (r_0,y''_0)\}$, where $\lambda=tm^{\frac{N-2}{N-4}}$, $\varepsilon>0$ is a sufficiently small number, $\Omega_1$ is defined as \eqref{Omega} and $w$ is the unique solution of the following equation:
\begin{equation}\label{def-w}
    \begin{cases}
        -\Delta w = V_{0,1}^{p-1}, \ \ \text{in} \ \ \R^N, \\
        |w| \to 0, \ \ \text{as} \ \ |x| \to +\infty.
    \end{cases}
\end{equation}
Moreover, we have the following estimate
\begin{equation}\label{varphi-h}
    \frac{\partial}{\partial \Box_h}\varphi(y)=O\left(\frac{\lambda^{\frac{N}{q+1}+n_h}}{\lambda^{2+\varepsilon}}\right),\quad y\in\Omega_1.
\end{equation}
uniformly in the parameter space above for $h=1,2,\cdots,N$ and $\Box_h = \lambda$ if $h=1$, $\Box_h = \bar{r}$ if $h=2$, $\Box_h=\bar{y}_h''$ if $h=3,\cdots,N$.
\end{lemma}

\begin{proof}
From the definition of $\varphi(y)$, we have the following representation of $\varphi(y)$:
\begin{equation}\label{nn2-19-1}
    \varphi(y) = \int_{\R^N}\frac{C(N)}{|y-z|^{N-2}} \left( \bigg(\sum_{j=1}^m V_{x_j, \lambda}(z)\bigg)^p- \sum_{j=1}^m V_{x_j,\lambda}^p(z) \right)\ \dz > 0, \quad y \in \R^N.
\end{equation}
Next, we estimate the right hand side of \eqref{nn2-19-1} when $y \in \Omega_1$. Since $\tau =\frac{N-4}{N-2} \in (0,1)$, we can choose $\mu \in (\tau,1)$ , which is slightly larger than $\tau$. Then, we can split the integral on the right hand side of \ref{nn2-19-1} into three parts.

\textbf{Estimate the integral in $B_{_{\lambda^{-\mu}}}(x_1)$.} Since we choose $\mu$ slightly larger than $\tau$, we have
\begin{equation}\label{Vj2}
\begin{aligned}
    \sum_{j=2}^m V_{x_j,\lambda}(y) &\le  \sum_{j=2}^m \frac{C\lambda^{\frac{N}{p+1}}}{(\lambda|y-x_j|)^{N-2}}
    \le \sum_{j=2}^m \frac{C\lambda^{\frac{N}{p+1}}}{(\lambda|x_1-x_j|)^{N-2}} \leq \frac{C\lambda^{\frac{N}{p+1}}m^{N-2}}{\lambda^{N-2}}\leq CV_{x_1,\lambda}(y),
\end{aligned}
\end{equation}
for $y \in B_{_{\lambda^{-\mu}}}(x_1)$. Therefore, for small $\delta>0$ such that $0 < \delta < p - \frac{N}{N-2}$, we have
\begin{align}
        &\quad \int_{B_{\lambda^{-\mu}}(x_1)} \frac{C(N)}{|y-z|^{N-2}} \left( \bigg(\sum_{j=1}^m V_{x_j,\lambda}(z)\bigg)^p- \sum_{j=1}^m V_{x_j,\lambda}^p(z) \right) \ \dz \nonumber \\
        &= p\int_{B_{\lambda^{-\mu}}(x_1)} \frac{C(N)}{|y-z|^{N-2}} \bigg(V_{x_1,\lambda}^{p-1}(z)\sum_{j=2}^m V_{x_j,\lambda}(z)\bigg)\ \dz - \int_{B_{\lambda^{-\mu}}(x_1)}\frac{C(N)}{|y-z|^{N-2}} \bigg(\sum_{j=2}^m V_{x_j,\lambda}^p(z)\bigg)\ \dz \label{1-29-7} \\
        &\quad+ O\Bigg(\int_{B_{\lambda^{-\mu}}(x_1)}\frac{1}{|y-z|^{N-2}} \Bigg( V_{x_1,\lambda}^{p-1-\delta}(z)\bigg(\sum_{j=2}^m V_{x_j,\lambda}(z)\bigg)^{1+\delta}\Bigg)\ \dz\Bigg). \nonumber
\end{align}
By Lemma \ref{L3}, we have for some small $\sigma>0$ that
\begin{align}
    &\quad p\int_{B_{\lambda^{-\mu}}(x_1)} \frac{C(N)}{|y-z|^{N-2}} \bigg(V_{x_1,\lambda}^{p-1}(z)\sum_{j=2}^m V_{x_j,\lambda}(z)\bigg)\ \dz \nonumber \\
    &= p\left( \lambda^{\frac{N}{p+1}} \sum_{j=2}^m \frac{b_{N,p}}{(\lambda|x_j-x_1|)^{N-2}} + O\bigg(\frac{m^{N-2}}{\lambda^{\frac{N}{q+1}+\sigma}}\bigg)\right) \int_{B_{\lambda^{-\mu}}(x_1)}\frac{C(N)}{|y-z|^{N-2}} V_{x_1,\lambda}^{p-1}(z)\ \dz  \label{2-29-7}.
\end{align}
In addition, we have
\begin{equation}\label{3-29-7}
\begin{aligned}
    & \quad\int_{B_{\lambda^{-\mu}}(x_1)}\frac{C(N)}{|y-z|^{N-2}} V_{x_1,\lambda}^{p-1}(z)\ \dz = \frac1{\lambda^{2- \frac{(p-1)N}{p+1} }} \int_{B_{\lambda^{1-\mu}}(0)}\frac{C(N)}{ |z-\lambda(y-x_1)|^{N-2}} V_{0,1}^{p-1}(z)\ \dz\\
    &= \frac{w(\lambda(y-x_1))}{\lambda^{2- \frac{(p-1)N}{p+1} }} - \frac1{\lambda^{2- \frac{(p-1)N}{p+1} }}  \int_{\R^N\setminus B_{\lambda^{1-\mu}}(0)}\frac{C(N)}{ |z-\lambda(y-x_1)|^{N-2}} V_{0,1}^{p-1}(z)\ \dz.
\end{aligned}
\end{equation}
If $\lambda|y-x_1|\le \frac{1}{2}\lambda^{1-\mu}$, then
\begin{equation}\label{5-29-7}
    \int_{\R^N\setminus B_{\lambda^{1-\mu}}(0)}\frac{C(N)}{ |z-\lambda(y-x_1)|^{N-2}} V_{0,1}^{p-1}(z)\ \dz
    \le C\int_{\R^N\setminus B_{\lambda^{1-\mu}}(0)}\frac 1{|z|^{p(N-2)}} \le \frac{C}{\lambda^{\sigma}}.
\end{equation}
If $\lambda|y-x_1|> \frac{1}{2}\lambda^{1-\mu}$, then we can use Lemma \ref{B3} to get
\begin{equation}\label{6-29-7}
    \int_{\R^N\setminus B_{\lambda^{1-\mu}}(0)}\frac{1}{ |z-\lambda(y-x_1)|^{N-2}} V_{0,1}^{p-1}(z)\ \dz \le \frac{C}{(1+\lambda |y-x_1|)^{(p-1)(N-2)-2}} \le \frac{C}{\lambda^{\sigma}}.
\end{equation}
Combine \eqref{2-29-7}--\eqref{6-29-7}, we obtain
\begin{multline}\label{7-29-7}
    p\int_{B_{\lambda^{-\mu}}(x_1)} \frac{C(N)}{|y-z|^{N-2}} \bigg(V_{x_1,\lambda}^{p-1}(z)\sum_{j=2}^m V_{x_j,\lambda}(z)\bigg)\, \dz = \sum\limits_{j=2}^m\frac{B_{11}\lambda^{\frac{N}{q+1}}}{(\lambda|x_j-x_1|)^{N-2}}w(\lambda(y-x_1))+O\bigg(\frac{\lambda^{\frac{N}{q+1}}}{\lambda^{2+\varepsilon}}\bigg).
\end{multline}
Recall that $0 < \delta < p - \frac{N}{N-2}$, we have
\begin{align}
        &\quad \int_{B_{\lambda^{-\mu}}(x_1)}\frac{1}{|y-z|^{N-2}} \Bigg( V_{x_1,\lambda}^{p-1-\delta}(z)\bigg(\sum_{j=2}^m V_{x_j,\lambda}(z)\bigg)^{1+\delta}\Bigg)\  \dz  \nonumber\\
        &\le C\left( \sum_{j=2}^m \frac{\lambda^{\frac{N}{p+1}}} {(\lambda |x_j-x_1|)^{N-2}}\right)^{1+\delta} \int_{B_{\lambda^{-\mu}}(x_1)}\frac{1}{|y-z|^{N-2}} V_{x_1,\lambda}^{p-1-\delta}(z)\  \dz \label{8-29-7}\\
       &\le C\left( \sum_{j=2}^m \frac1{(\lambda |x_j-x_1|)^{N-2}}\right)^{1+\delta} \int_{\R^N}\frac{\lambda^{\frac{pN}{p+1}}}{|y-z|^{N-2}} V_{0,1}^{p-1-\delta}(\lambda(z-x_1))\ \dz \nonumber \\
        &\le C \lambda^{\frac{pN}{p+1}-2} \left( \sum_{j=2}^m \frac{1} {(\lambda |x_j-x_1|)^{N-2}}\right)^{1+\delta} \le \frac{C\lambda^{\frac{N}{q+1}}}{\lambda^{2+\varepsilon}} ,\nonumber
\end{align}
and
\begin{equation}\label{9-29-7}
    \int_{B_{\lambda^{-\mu}}(x_1)}\frac{C(N)}{|y-z|^{N-2}} \bigg(\sum_{j=2}^m V_{x_j,\lambda}^p(z)\bigg)\ \dz \le \int_{B_{\lambda^{-\mu}}(x_1)}\frac{C}{|y-z|^{N-2}} \bigg(V_{x_1,\lambda}^{p-1-\delta}\sum_{j=2}^m V_{x_j,\lambda}^{1+\delta}(z)\bigg)\ \dz  \le \frac{C\lambda^{\frac{N}{q+1}}}{\lambda^{2+\varepsilon}}.
\end{equation}
Thus, combining \eqref{1-29-7} and \eqref{7-29-7}-\eqref{9-29-7}, we obtain
\begin{equation}\label{2-30-7}
    \begin{split}
        & \quad \int_{B_{\lambda^{-\mu}}(x_1)} \frac{C(N)}{|y-z|^{N-2}} \left(\bigg(\sum_{j=1}^m V_{x_j, \lambda}(z)\bigg)^p- \sum_{j=1}^m V_{x_j,\lambda}^p(z) \right)\ \dz \\
        & = \sum\limits_{j=2}^m\frac{B_{11}\lambda^{\frac{N}{q+1}}}{(\lambda|x_j-x_1|)^{N-2}}w(\lambda(y-x_1))+O\bigg(\frac{\lambda^{\frac{N}{q+1}}}{\lambda^{2+\varepsilon}}\bigg).
    \end{split}
\end{equation}

\textbf{Estimate the integral in $\Omega_1 \backslash B_{\lambda^{-\mu}}(x_1)$.} Using Lemma \ref{A1}, \ref{B1} and \ref{B2}, we have for any $y\in \Omega_1$ that
\begin{align}
        & \quad \int_{\Omega_1\setminus B_{\lambda^{-\mu}}(x_1)}\frac{1}{|y-z|^{N-2}} \bigg(\sum_{j=2}^m V_{x_j,\lambda}(z)\bigg)^p\ \dz  \nonumber \\
        & \le \frac{C}{\lambda^{\frac{pN}{q+1}}}\int_{\Omega_1\setminus B_{\lambda^{-\mu}}(x_1)}\frac{1}{|y-z|^{N-2}} \bigg( \sum_{j=2}^m \frac{1}{|z-x_1|^{N-3-\theta}|x_1 - x_j|^{1+\theta}}\bigg)^p\ \dz \label{1-29-9} \\
        & =  \frac{C}{\lambda^{\frac{pN}{q+1}}}\int_{\Omega_1\setminus B_{\lambda^{-\mu}}(x_1)}\frac{1}{|y-z|^{N-2}}  \frac{m^{(1+\theta)p}}{|z-x_1|^{(N-3-\theta)p}}\ \dz \nonumber \\
        &\le \frac{Cm^{(1+\theta)p}}{\lambda^{\frac{pN}{q+1}}} \lambda^{\mu((N-3-\theta)p-2)} \le \frac{C\lambda^{\frac{N}{q+1}}}{\lambda^{2+\varepsilon}},\nonumber
\end{align}
where we choose $\theta > 0$ small enough such that $p(N-3-\theta) > 2$, and the last inequality follows from $p(N-2) > N$. Hence, using \eqref{1-29-9} and the same method as \eqref{5-29-7}--\eqref{6-29-7}, we obtain
\begin{align}
        0 &< \int_{\Omega_1\setminus B_{\lambda^{-\mu}}(x_1)} \frac{C(N)}{|y-z|^{N-2}} \left(\bigg(\sum_{j=1}^m V_{x_j,\lambda}(z)\bigg)^p- \sum_{j=1}^m V_{x_j,\lambda}^p(z) \right)\ \dz \nonumber  \\
        &\le \int_{\Omega_1\setminus B_{\lambda^{-\mu}}(x_1)} \frac{C(N)}{|y-z|^{N-2}} \left( \bigg(\sum_{j=1}^m V_{x_j,\lambda}(z)\bigg)^p- V_{x_1,\lambda}^p(z) \right)\ \dz \label{3-30-7}\\
        &\le C\int_{\Omega_1\setminus B_{\lambda^{-\mu}}(x_1)}\frac{1}{|y-z|^{N-2}} \left( V_{x_1,\lambda}^{p-1}(z)\sum_{j=2}^m V_{x_j,\lambda}(z) + \bigg(\sum_{j=2}^m V_{x_j,\lambda}(z)\bigg)^p\right)\ \dz \le \frac{C\lambda^{\frac{N}{q+1}}}{\lambda^{2+\varepsilon}}. \nonumber
\end{align}
Thus, it follows that
\begin{equation}\label{4-30-7}
    \int_{\Omega_1 \backslash B_{\lambda^{-\mu}}(x_1)} \frac{C(N)}{|y-z|^{N-2}} \left( \bigg(\sum_{j=1}^m V_{x_j,\lambda}(z)\bigg)^p- \sum_{j=1}^m V_{x_j,\lambda}^p(z) \right)\ \dz =  O\bigg(\frac{\lambda^{\frac{N}{q+1}}}{\lambda^{2+\varepsilon}}\bigg).
\end{equation}

\textbf{Estimate the integral in $\R^N \backslash \Omega_1 = \cup_{i=2}^m \Omega_i$}. For $ z \in \Omega_i$, we can use an argument similar to \eqref{Vj2} to obtain
\[
     \sum_{j\neq i} V_{x_j,\lambda}(z) \leq CV_{x_i,\lambda}(z).
\]
Therefore, we have
\begin{align*}
    &\quad \sum_{i=2}^m \int_{\Omega_i} \frac{C}{|y-z|^{N-2}} \left( V_{x_i,\lambda}^{p-1}(z) \sum_{j \ne i} V_{x_j,\lambda}(z) + \bigg(\sum_{j \ne i} V_{x_j,\lambda}(z)\bigg)^p\right)\ \dz \\
    & \displaystyle \le \frac{C}{\lambda^{\frac{pN}{q+1}}} \sum_{i=2}^m \int_{\Omega_i} \frac{1}{|y-z|^{N-2}} \left(\frac{1}{|z-x_i|^{(p-1)(N-2)}} \sum_{j \ne i} \frac{1}{|z-x_j|^{N-2}} + \bigg(\sum_{j \ne i} \frac{1}{|z-x_j|^{N-2}}\bigg)^p\right) \ \dz \\
    &= \frac{Cm^{p(N-2)-2}}{\lambda^{\frac{pN}{q+1}}} \sum_{i=2}^m \int_{\Omega_i} \frac{1}{|my-z|^{N-2}}\frac{1}{|z-mx_i|^{(p-1)(N-2)}} \sum_{j \ne i} \frac{1}{|z-mx_j|^{N-2}} \ \dz \\
    &\quad + \frac{Cm^{p(N-2)-2}}{\lambda^{\frac{pN}{q+1}}} \sum_{i=2}^m \int_{\Omega_i} \frac{1}{|my-z|^{N-2}} \bigg(\sum_{j \ne i} \frac{1}{|z-mx_j|^{N-2}}\bigg)^p \  \dz \\
    & : = K_1 + K_2.
    &
\end{align*}
Note that $y \in \Omega_1$, we can choose $\theta>0$ small enough such that
\begin{equation}\label{K1}
    \begin{split}
        K_1 & \leq \frac{Cm^{p(N-2)-2}}{\lambda^{\frac{pN}{q+1}}} \sum_{i=2}^m \int_{\Omega_i} \frac{1}{|my-z|^{N-2}}\frac{1}{|z-mx_i|^{(p-1)(N-2)+ N - 3 - \theta}} \sum_{j \ne i} \frac{1}{|mx_1-mx_j|^{1+\theta}} \ \dz \\
        & \leq \frac{Cm^{p(N-2)-2}}{\lambda^{\frac{pN}{q+1}}} \sum_{i=2}^m \frac{1}{|my - mx_i|^{(p-1)(N-2) + N-5-\theta}} = \dfrac{Cm^{p(N-2)-2}}{\lambda^{\frac{pN}{q+1}}} = O\bigg(\frac{\lambda^{\frac{N}{q+1}}}{\lambda^{2+\varepsilon}}\bigg),
    \end{split}
\end{equation}
where we have used Lemma \ref{A1}, \ref{B1}, $(p-1)(N-2) + N - 5 > 1$ and $\frac{(p+1)N}{q+1} - \tau(p(N-2)-2) = 2p-\frac{2N}{N-2} >0$. Similarly, we have
\begin{equation}\label{K2}
    \begin{split}
        K_2 & \leq \frac{Cm^{p(N-2)-2}}{\lambda^{\frac{pN}{q+1}}} \sum_{i=2}^m \int_{\Omega_i} \frac{1}{|my-z|^{N-2}}\frac{1}{|z-mx_i|^{N-2+ (N - 3 - \theta)p}} \sum_{j \ne i} \bigg( \frac{1}{|mx_1-mx_j|^{1+\theta}}\bigg)^p \ \dz \\
        & \leq \frac{Cm^{p(N-2)-2}}{\lambda^{\frac{pN}{q+1}}} \sum_{i=2}^m \frac{1}{|my - mx_i|^{N-4+(N-3-\theta)p}} = \dfrac{Cm^{p(N-2)-2}}{\lambda^{\frac{pN}{q+1}}} = O\bigg(\frac{\lambda^{\frac{N}{q+1}}}{\lambda^{2+\varepsilon}}\bigg).
    \end{split}
\end{equation}
Combine \eqref{K1} and \eqref{K2}, we obtain
\begin{equation}\label{5-30-7}
    \int_{\R^N \backslash \Omega_1} \frac{C(N)}{|y-z|^{N-2}} \left[\bigg(\sum_{j=1}^m V_{x_j,\lambda}(z)\bigg)^p- \sum_{j=1}^m V_{x_j,\lambda}^p(z) \right]\ \dz = O\bigg(\frac{\lambda^{\frac{N}{q+1}}}{\lambda^{2+\varepsilon}}\bigg).
\end{equation}
Then, \eqref{nn1-19-1} follows from \eqref{2-30-7}, \eqref{4-30-7} and \eqref{5-30-7}. Applying $\frac{\pa}{\pa \Box_h}$ on both sides of \eqref{nn2-19-1} and using the same method, we can obtain \eqref{varphi-h}.
\end{proof}

Lemma \ref{nnl2-19-1} is useful when the integral involving $Y_{\bar{r},\bar{y}'',\lambda}$ produce the main order term, see the estimate of $I_2$ in Lemma \ref{expansion_lambda}. In other cases, we may only need the upper bound of $Y_{\bar{r},\bar{y}'',\lambda}$, which is given by the next lemma.

\begin{lemma}
\label{lemma:U}
Assume that $N \ge 5$ and $p \in (\frac{N}{N-2},\frac{N+2}{N-2}]$. Then there is a small constant $\theta_0$, such that for any $\theta \in (0, \theta_0)$,  it holds that
\[
    Y_{\bar{r},\bar{y}'',\lambda}^*(y) \le C \sum_{i=1}^m \frac{\lambda^{\frac{N}{q+1}}}{(1+\lambda|y-x_i|)^{N-2}} + \frac{C}{\lambda^{\frac{pN}{q+1}}} \sum_{i=1}^m \frac{m^{p(N-2)-2}}{(1+m|y-x_i|)^{\min\{N-2,p(N-3-\theta)-2\}}}
\]
for $y \in \R^N$.
\end{lemma}

\begin{proof}
 Since $Y_{\bar{r},\bar{y}'',\lambda}^*(y)$ is the unique solution of \eqref{ansatzes_2}, we have
\begin{align}\label{U1}
    Y_{\bar{r},\bar{y}'',\lambda}^*(y) & = \int_{\R^N} \frac{C(N)}{|y-z|^{N-2}} \left( \sum_{j=1}^m V_{x_j,\lambda}(z) \right)^p  \ \ dz= \sum_{i=1}^m \int_{\Omega_i} \frac{C(N)}{|y-z|^{N-2}} \left( \sum_{j=1}^m V_{x_j,\lambda} (z) \right)^p  \ \dz \\
    & \leq C \sum_{i=1}^m \int_{\Omega_i} \frac{1}{|y-z|^{N-2}} \left( V_{x_i,\lambda}^p(z)+ \bigg(\sum_{j \neq i} V_{x_j,\lambda} (z)\bigg)^p\right) \ \dz.
\end{align}
For $i=1, \cdots, m$, we have
\begin{equation}\label{U2}
    \int_{\Omega_i} \frac{1}{|y-z|^{N-2}} V_{x_i,\lambda}^p(z) dz \le \frac{C\lambda^{\frac{pN}{p+1}-2}}{(1+\lambda|y-x_i|)^{N-2}} = \frac{C\lambda^{\frac{N}{q+1}}}{(1+\lambda|y-x_i|)^{N-2}} .
\end{equation}
For $j \neq i$, we have $|z - m x_j| \geq 1 + |z-mx_i|$ for $z \in \Omega_i$. Therefore, we have
\begin{align}\label{U3}
    & \quad \int_{\Omega_i} \frac{C(N)}{|y-z|^{N-2}} \left( \sum_{j \neq i} V_{x_j,\lambda}(z) \right)^p \ \dz  \le \frac{C m^{p(N-2)-2}}{\lambda^{\frac{pN}{q+1}}} \int_{\Omega_i} \frac{1}{|my-z|^{N-2}} \left( \sum_{j \neq i} \frac{1}{|z-mx_j|^{N-2}} \right)^p \ \dz \\
    &\le \frac{C m^{p(N-2)-2}}{\lambda^{\frac{pN}{q+1}}} \int_{\Omega_i} \frac{1}{|my-z|^{N-2}} \frac{1}{(1+|z-mx_i|)^{p(N-3-\theta)}} \left( \sum_{j \neq i} \frac{1}{|mx_i-mx_j|^{1+\theta}} \right)^p \ \dz \\
    &\le \frac{C m^{p(N-2)-2}}{\lambda^{\frac{pN}{q+1}}} \frac{1}{(1+m|y-x_i|)^{\min\{ N-2, p(N-3-\theta)-2\}}}.
\end{align}
Then the desired result follows from \eqref{U1}-\eqref{U3}.
\end{proof}

Lemma \ref{lemma:U} and Lemma \ref{nnl2-19-1} provide a good estimate of $Y_{\bar{r},\bar{y}'',\lambda}^*$. However, due to their complex forms, we sometimes use a simpler but rougher estimate , which is also sufficient in some cases.
\begin{lemma}
\label{l1-23-4}Suppose that $N \ge 5$, $p \in (\frac{N}{N-2},\frac{N+2}{N-2}]$ and \eqref{critical hyperbola} hold. If $N = 5$, we also assume that $p \in (2,\frac{7}{3}]$. Let $\tau = \frac{N-4}{N-2} \in (0,1)$. Then there is a constant $\theta_0>$, such that for any $\theta \in (0, \theta_0)$, it holds that
\begin{equation}\label{10-23-4}
     Y_{\bar{r},\bar{y}'',\lambda}^{*}(y)\le C\sum_{j=1}^m \frac{ \lambda^{\frac{N}{q+1}}  }{ (1+\lambda |y-x_j|)^{  \frac{N}{q+1}+\tau+\theta } }, \ \ y \in \R^N.
\end{equation}
\end{lemma}
\begin{proof}
    Without loss of generality, we may assume $y \in \Omega_1$. Recall that $S$ is defined in \eqref{eq:S}. If $y \in S$, we can use Lemma \ref{nnl2-19-1}, \ref{A1}, the definition of $w$ and $$1+\lambda|y-x_1| \leq 1+\lambda/(mr_0) \leq C\lambda^{1-\tau} = C\lambda^{\frac{2}{N-2}}, \ \  \text{when} \ \ y \in S, \ \ j \neq 1,$$ to obtain
    \[
        \begin{split}
            \sum_{j=2}^m U_{x_j,\lambda}(y) + \varphi(y) & \leq \sum_{j=2}^m \frac{C\lambda^{\frac{N}{q+1}}}{(1+\lambda|y-x_j|)^{N-2}} +  \frac{C\lambda^{\frac{N}{q+1}}}{\lambda^2(1+\lambda|y-x_1|)^{(N-2)(p-1)-2}} + \frac{C\lambda^{\frac{N}{q+1}}}{\lambda^{2+\varepsilon}} \\
            & \leq U_{x_1,\lambda}(y)\bigg( \sum_{j=2}^m \frac{\lambda^2}{|\lambda x_1 - \lambda x_j |^{N-2}} + \lambda^{\frac{2}{N-2}(N - (N-2)p)-2} + \frac{1}{\lambda^{\varepsilon}} \bigg) \\
            & \leq U_{x_1,\lambda}(y) \leq \frac{ \lambda^{\frac{N}{q+1}}  }{ (1+\lambda |y-x_1|)^{  \frac{N}{q+1}+\tau+\theta } }.
        \end{split}
    \]
    If $y \in \Omega_1 \backslash S$ and $p(N-3) > N$, we have
    \[
        \begin{split}
            \frac{1}{\lambda^{\frac{pN}{q+1}}} \sum_{i=1}^m \frac{m^{p(N-2)-2}}{(1+m|y-x_i|)^{N-2}} & \leq \bigg( \frac{m}{\lambda} \bigg)^{p(N-2)-2} \bigg( \frac{\lambda}{m} \bigg)^{N-2} \sum_{i=1}^m \frac{\lambda^{\frac{N}{q+1}}}{|\lambda y - \lambda x_i|^{N-2}} \\
            & \leq C\bigg(\frac{m}{\lambda} \bigg)^{p(N-2)-N} \sum_{i=1}^m \frac{\lambda^{\frac{N}{q+1}}}{(1+\lambda| y - x_i|)^{N-2}}  \\
            & \leq  C\sum_{i=1}^m \frac{\lambda^{\frac{N}{q+1}}}{(1+\lambda| y - x_i|)^{\frac{N}{q+1}+\tau + \theta}}.
        \end{split}
    \]
    If $p(N-3)< N$, using $(1+\lambda|y-x_i|) \approx \lambda|y-x_i|$ and $(1+m|y-x_i|) \approx m|y-x_i|$, we only need to prove
    \[
        \frac{1}{\lambda^{\frac{pN}{q+1}}} \frac{m^{p(N-2)-2}}{(m|y-x_i|)^{p(N-3-\theta)-2}} \leq C  \frac{\lambda^{\frac{N}{q+1}}}{(\lambda| y - x_i|)^{\frac{N}{q+1}+\tau + \theta}}, \ \ i=1,2,\cdots,m,
    \]
    which is equivalent to
    \[
        \bigg( \frac{m}{\lambda} \bigg)^{p(1+\theta)} \leq C (\lambda|y-x_i|)^{p(N-3-\theta) - 2 - \frac{N}{q+1} - \theta - \tau}.
    \]
    The above inequality holds since $\lambda|y-x_i| \geq \lambda/m \to +\infty$ and
    \[
        p(N-3) - 2 - \frac{N}{q+1} - \theta - \tau > 0,
    \]
    when $\frac{N}{N-2}< p< \frac{N}{N-3}$ and $\theta >0$ can be chosen sufficiently small. This completes the proof.
\end{proof}

\section{The energy expansion}
In this section, we will give some estimates of the energy expansion for the approximate solutions. Recall that
\[
    I(u,v) = \int_{\R^N}(\nabla u \nabla v + V(|y'|,y'')uv ) \ \dy - \dfrac{1}{p+1} \int_{\R^N} |v|^{p+1} \ \dy - \dfrac{1}{q+1} \int_{\R^N} |u|^{q+1} \ \dy.
\]

\begin{lemma}\label{expansion_lambda}
    If $N\geq 5$, then there is a small constant $\varepsilon>0$ such that
    $$
        \frac{\partial I(Y_{\bar{r}, \bar{y}'', \lambda},Z_{\bar{r}, \bar{y}'', \lambda})}{\partial \lambda}=m\left( -\frac{B_1}{\lambda^3}V(\bar{r},\bar{y}'')+\sum\limits_{j=2}^m\frac{B_2}{\lambda^{N-1}|x_1-x_j|^{N-2}}+O\left(\frac{1}{\lambda^{3+\varepsilon}}\right)\right),
    $$
    for some positive constants $B_1$ and $B_2$.
\end{lemma}
\begin{proof}
    First, we split $\frac{\partial I(Y_{\bar{r}, \bar{y}'', \lambda},Z_{\bar{r}, \bar{y}'', \lambda})}{\partial \lambda}$ as follows:
     \begin{equation*}
         \begin{aligned}
             &\quad \frac{\partial I(Y_{\bar{r}, \bar{y}'', \lambda},Z_{\bar{r}, \bar{y}'', \lambda})}{\partial \lambda}\\
             &=\left\langle -\Delta Y_{\bar{r}, \bar{y}'', \lambda}+V(y)Y_{\bar{r}, \bar{y}'', \lambda}-Z_{\bar{r}, \bar{y}'', \lambda}^p,\frac{\partial Z_{\bar{r}, \bar{y}'', \lambda}}{\partial \lambda}\right\rangle
             +\left\langle -\Delta Z_{\bar{r}, \bar{y}'', \lambda}+V(y)Z_{\bar{r}, \bar{y}'', \lambda}-Y_{\bar{r}, \bar{y}'', \lambda}^q,\frac{\partial Y_{\bar{r}, \bar{y}'', \lambda}}{\partial \lambda}\right\rangle\\
             &:=\frac{\partial I(Y^*_{\bar{r}, \bar{y}'', \lambda},Z^*_{\bar{r}, \bar{y}'', \lambda})}{\partial \lambda}+E_1+E_2-E_3,
         \end{aligned}
     \end{equation*}
     where
     $$E_1=\int_{\R^N}(\xi^2-1)V(y)\frac{\partial}{\partial\lambda}(Y^*_{\bar{r}, \bar{y}'',\lambda}\cdot Z^*_{\bar{r}, \bar{y}'', \lambda}),$$
     \begin{equation*}
     \begin{aligned}
         E_2=\int_{\R^N}\left[(1-\xi^{p+1})(Z^*_{\bar{r}, \bar{y}'', \lambda})^p +(1-\xi^2)\Delta Y^*_{\bar{r}, \bar{y}'', \lambda}\right]\frac{\partial}{\partial\lambda}Z^*_{\bar{r}, \bar{y}'', \lambda}\\
     +\int_{\R^N}\left[(1-\xi^{q+1})(Y^*_{\bar{r}, \bar{y}'', \lambda})^q +(1-\xi^2)\Delta Z^*_{\bar{r}, \bar{y}'', \lambda}\right]\frac{\partial}{\partial\lambda}Y^*_{\bar{r}, \bar{y}'', \lambda},
     \end{aligned}
     \end{equation*}
     and
     $$E_3=\int_{\R^N}\left( 2\xi \nabla\xi\cdot\nabla Y^*_{\bar{r}, \bar{y}'', \lambda}+\xi\Delta\xi Y^*_{\bar{r}, \bar{y}'', \lambda}\right)\frac{\partial}{\partial\lambda}Z^*_{\bar{r}, \bar{y}'', \lambda}+(2\xi \nabla\xi\cdot\nabla Z^*_{\bar{r}, \bar{y}'', \lambda}+\xi\Delta\xi Z^*_{\bar{r}, \bar{y}'', \lambda})\frac{\partial}{\partial\lambda}Y^*_{\bar{r}, \bar{y}'', \lambda}.$$
     Next, we can divide $\frac{\partial I(Y^*_{\bar{r}, \bar{y}'', \lambda},Z^*_{\bar{r}, \bar{y}'', \lambda})}{\partial \lambda}$ as
     \begin{equation*}
        \begin{aligned}
        \frac{\partial I(Y^*_{\bar{r}, \bar{y}'', \lambda},Z^*_{\bar{r}, \bar{y}'', \lambda})}{\partial \lambda} & = \int_{\R^N}V(y)\frac{\partial}{\partial \lambda}(Y^*_{\bar{r}, \bar{y}'', \lambda}\cdot Z^*_{\bar{r}, \bar{y}'', \lambda})
        +\left( \int_{\R^N} \big(-\Delta Z^*_{\bar{r}, \bar{y}'', \lambda} -(Y^*_{\bar{r}, \bar{y}'', \lambda})^{q}\big)\frac{\partial}{\partial \lambda}Y^*_{\bar{r}, \bar{y}'', \lambda} \right)\\
        &\hspace{1em}
        +\quad \left( \int_{\R^N} \big(-\Delta Y^*_{\bar{r}, \bar{y}'', \lambda} -(Z^*_{\bar{r}, \bar{y}'', \lambda})^{p}\big)\frac{\partial}{\partial \lambda}Z^*_{\bar{r}, \bar{y}'', \lambda} \right)\\
        &=\int_{\R^N}V(y)\frac{\partial}{\partial \lambda}(Y^*_{\bar{r}, \bar{y}'', \lambda}\cdot Z^*_{\bar{r}, \bar{y}'', \lambda})
        + \int_{\R^N} \bigg(\sum\limits_{j=1}^m U^q_{x_j,\lambda}-(Y^*_{\bar{r}, \bar{y}'', \lambda})^{q}\bigg)\frac{\partial}{\partial \lambda}Y^*_{\bar{r}, \bar{y}'', \lambda} \\
        &:=I_1+I_2.
        \end{aligned}
    \end{equation*}
    Therefore, we have
    \[
        \frac{\partial I(Y_{\bar{r}, \bar{y}'', \lambda},Z_{\bar{r}, \bar{y}'', \lambda})}{\partial \lambda} = I_1 + I_2 + E_1 + E_2 - E_3.
    \]

\textbf{Estimate $I_1$}. We divide $\R^N$ into $\R^N = (\R^N \backslash S) \cup S$, where $S$ is defined in \eqref{eq:S}, then
\begin{align*}
        I_1&=m\int_{\R^N} V(y)\frac{\partial}{\partial \lambda}(Y^*_{\bar{r}, \bar{y}'', \lambda}(y)V_{x_1,\lambda}(y))\dy \\
        &=m\int_{S}V(y)\frac{\partial}{\partial \lambda}(U_{x_1,\lambda}(y)V_{x_1,\lambda}(y))\dy+m\int_{S}V(y)\frac{\partial}{\partial \lambda}\Big(\big(Y^*_{\bar{r}, \bar{y}'', \lambda}(y)-U_{x_1,\lambda}(y)\big)V_{x_1,\lambda}(y)\Big)\dy\\
        &\quad+m\int_{\R^N\backslash S}V(\bar{r},\bar{y}'')\frac{\partial}{\partial \lambda}\left(Y^*_{\bar{r}, \bar{y}'', \lambda}(y)V_{x_1,\lambda}(y)\right)\dy\\
        &:=J_{10}+J_{11}+J_{12}.
\end{align*}
Next, we estimate $J_{10}$, $J_{11}$ and $J_{12}$.
\begin{align*}
    J_{10}&=mV(\bar{r},\bar{y}'') \frac{\partial}{\partial \lambda} \left( \int_{\R^N} - \int_{\R^N \backslash S}\right) U_{x_1,\lambda}V_{x_1,\lambda} + m\frac{\partial}{\partial \lambda}\int_{S}(V(\bar{r},\bar{y}'') - V(y)) U_{x_1,\lambda}V_{x_1,\lambda} \\
    & = mV(\bar{r},\bar{y}'') \frac{\partial}{\partial \lambda}  \int_{\R^N}U_{x_1,\lambda}V_{x_1,\lambda} + \frac{m}{\lambda} O\left(\int_{\R^N\backslash S}V(\bar{r},\bar{y}'')U_{x_1,\lambda}V_{x_1,\lambda} +\int_{S}|V(y)-V(\bar{r},\bar{y}'')|U_{x_1,\lambda}V_{x_1,\lambda}\right)\\
    &=-\frac{2mV(\bar{r},\bar{y}'')}{\lambda^3}\int_{\R^N}U_{0,1}V_{0,1}+O\left(\frac{m}{\lambda^{3+\varepsilon}}\right).
\end{align*}
Next, we estimate $J_{11}$. Recall that $Y^*_{\bar{r}, \bar{y}'', \lambda} = \sum_{j=1}^m U_{x_j,\lambda} + \varphi$, we can use Lemma \ref{nnl2-19-1} to obtain
\begin{align*}
     J_{11}&\leq C\frac{m}{\lambda}\int_{S} V_{x_1,\lambda}(y)\left(\sum\limits_{j=2}^m U_{x_j,\lambda}(y)+\varphi(y)\right) \ \dy\\
    &=Cm\lambda^{N-3}\int_{S} V_{0,1}(\lambda(y-x_1)) \sum\limits_{j=2}^m \left( U_{0,1}\big(\lambda(y-x_j)) + \frac{B_{11}}{(\lambda|x_j-x_1|)^{N-2}}w(\lambda(y-x_1))\right) + O\left( \frac{1}{m\lambda^{3+\varepsilon}} \right)\\
    &=O\left(\frac{m}{\lambda^3}\cdot \left(\frac{m}{\lambda} \right)^{N-4}\right)= O\left(\frac{m}{\lambda^{3+\varepsilon}}\right).
\end{align*}
Regarding $J_{12}$, we can use Lemma \ref{lemma:U} to get
\begin{equation}\label{J12}
    \begin{aligned}
        J_{12}&\leq C\frac{m}{\lambda^3}\int_{\R^N\backslash S}\frac{\lambda^N}{(1+\lambda|y-x_1|)^{N-2}}\sum\limits_{j=1}^{m}\frac{1}{(1+\lambda|y-x_j|)^{N-2}} \ \dy\\
        &\quad + C\frac{m}{\lambda^3} \left(\frac{m}{\lambda}\right)^{p(N-2)-2} \int_{\R^N\backslash S}\frac{\lambda^N}{(1+\lambda|y-x_1|)^{N-2}} \sum\limits_{j=1}^{m}\frac{1}{(1+m|y-x_j|)^{\min \{N-2,p(N-3-\theta)-2\}}} \ \dy\\
        & : =J_{121} + J_{122}.
    \end{aligned}
\end{equation}
If $p(N-3)>N$, then as shown in Lemma \ref{l}, then $J_{122}$ can be bounded by $J_{121}$, therefore, we can employ Lemma \ref{B2} to obtain
\begin{equation*}
    \begin{aligned}
        J_{12}&\leq C\frac{m}{\lambda^3}\int_{\R^N\backslash S}\frac{\lambda^N \ \dy}{(1+\lambda|y-x_1|)^{2N-4}}+ C\frac{m}{\lambda^3}\sum\limits_{j=2}^{m}\frac{1}{(\lambda|x_j-x_1|)^{\tau}}\int_{\R^N \backslash S} \frac{\lambda^N \ \dy}{(1+\lambda|y-x_1|)^{2N-4-\tau}}  \\
        &\leq O\left(\frac{m}{\lambda^{3+\varepsilon}}\right).
    \end{aligned}
\end{equation*}
On the other hand, if $p(N-3)\leq N$, $J_{121}$ will be bounded by $J_{122}$. We consider the cases $N \geq 6$, $p \in (\frac{N}{N-2}, \frac{N+2}{N-2})$ and $N =5$, $p \in (2, \frac{7}{3})$ separately. If $N\geq6$, we have
\begin{align}\label{J_12_1}
    p(N-3)-4-\tau>0,
\end{align}
then, similarly, we have
\begin{align*}
    J_{12} & \leq C \frac{m}{\lambda^3} \left(\frac{m}{\lambda}\right)^{p(N-2)-2}\int_{\R^N\backslash S}\frac{\lambda^N}{(1+\lambda|y-x_1|)^{N-2}}\sum\limits_{j=1}^{m}\frac{1}{(1+m|y-x_j|)^{p(N-3-\theta)-2}} \ \dy\\
    &\leq C\frac{m^{1+p(1+\theta)}}{\lambda^{3+{p(1+\theta)}}} \int_{\R^N\backslash S} \frac{\lambda^N}{(1+\lambda|y-x_1|)^{N-2}}\cdot \sum\limits_{j=1}^{m}\frac{1}{(1+\lambda|y-x_j|)^{p(N-3-\theta)-2}} \ \dy\\
    &\leq C\frac{m^{1+p(1+\theta)}}{\lambda^{3+{p(1+\theta)}}} \int_{\R^N\backslash S}   \frac{\lambda^{N}}{(1+\lambda|y-x_1|)^{p(N-3-\theta)+N-4-\tau}}  \ \dy\\
    &=O\left(\frac{m}{\lambda^{3+\varepsilon}} \right).
\end{align*}
If $N=5$, \eqref{J_12_1} may not hold and we split the integral domain into $(\R^N \backslash S) \cap \textbf{B}_{R_0} = \textbf{B}_{R_0} \backslash S$ and $\textbf{B}_{R_0}^c$ for some large $R_0$. In $\textbf{B}_{R_0} \backslash S$, we can use the same method in the case $N\geq 6$. In $\textbf{B}_{R_0}^c$, we have $|y-x_j| \approx |y-x_1| > \frac{R_0}{2}$ if we choose $R_0$ large enough. Then, choosing $\theta = \theta(p) > 0$ small enough, such that $p(N-3 - \theta) - 4 > 0$, we have
\begin{equation*}
    \begin{aligned}
        J_{12} & \leq C\frac{m^{1+p(1+\theta)}}{\lambda^{3+{p(1+\theta)}}} \left( \int_{\textbf{B}_{R_0}\backslash S}   \frac{\lambda^{N} \ \dy}{(1+\lambda|y-x_1|)^{p(N-3-\theta)+N-4-\tau}}  + \int_{ \textbf{B}_{R_0}^{c}} \frac{ m \lambda^{N} \ \dy}{(1+\lambda|y-x_1|)^{p(N-3-\theta)+N-4}}   \right) \\
        & \leq C\frac{m^{1+p(1+\theta)}}{\lambda^{3+{p(1+\theta)} -4 - \tau +p(N-3-\theta)}} + C\frac{m^{2+p(1+\theta)}}{\lambda^{3+{p(1+\theta)} }}  = O\left(\frac{m}{\lambda^{3+\varepsilon}} \right).
    \end{aligned}
\end{equation*}
Thus, we have
\begin{align}\label{I1}
    I_1 = J_{10} + J_{11} + J_{12}= -\frac{2mV(\bar{r},\bar{y}'')}{\lambda^3}\int_{\R^N}U_{0,1}V_{0,1} + O\left(\frac{m}{\lambda^{3+\varepsilon}}\right).
\end{align}

\textbf{Estimate $I_2$.} Next we estimate the integral $I_2$. Dividing $\R^N$ into $\cup_{i=1}^m \Omega_i$, we have
\begin{align*}
        I_2&=m\int_{\Omega_1} \left( \sum\limits_{j=1}^m U^q_{x_j,\lambda}-(Y^*_{\bar{r}, \bar{y}'', \lambda})^{q}  \right)\frac{\partial}{\partial \lambda}Y^*_{\bar{r}, \bar{y}'', \lambda} \\
        &=m\int_{S \cap \Omega_1} \left( \sum\limits_{j=1}^m U^q_{x_j,\lambda}-(Y^*_{\bar{r}, \bar{y}'', \lambda})^{q}  \right) \frac{\partial}{\partial \lambda}U_{x_1,\lambda}+ m\int_{S \cap \Omega_1} \left( \sum\limits_{j=1}^m U^q_{x_j,\lambda}-(Y^*_{\bar{r}, \bar{y}'', \lambda})^{q}  \right) \frac{\partial}{\partial \lambda}\left(\sum\limits_{j=2}^m U_{x_j,\lambda}+\varphi\right) \ \\
        &\quad +m\int_{\Omega_1\backslash S} \left( \sum\limits_{j=1}^m U^q_{x_j,\lambda}-(Y^*_{\bar{r}, \bar{y}'', \lambda})^{q}  \right)\frac{\partial}{\partial \lambda}Y^*_{\bar{r}, \bar{y}'', \lambda}\\
        &:=J_{20}+J_{21}+J_{22}.
\end{align*}
Recall that for any $y\in S$,
$$\sum\limits_{j=2}^m U_{x_j,\lambda}(y)+\varphi(y) \leq C\leq CU_{x_1,\lambda}(y).$$
We have for any $\delta \in (0, q-1)$,
\begin{align}
    J_{20}&=m\int_{S} \left( \sum\limits_{j=1}^m U^q_{x_j,\lambda}-(Y^*_{\bar{r}, \bar{y}'', \lambda})^{q}  \right)\frac{\partial}{\partial \lambda}U_{x_1,\lambda} \nonumber \\
    &=m\int_{S} \left( -qU_{x_1,\lambda}^{q-1} \left(\sum\limits_{j=2}^m U_{x_j,\lambda}+\varphi \right)  \right)\frac{\partial}{\partial \lambda}U_{x_1,\lambda}+O\left(\frac{m}{\lambda}\int_{S}U_{x_1,\lambda}^{q-\delta} (\sum\limits_{j=2}^m U_{x_j,\lambda}+\varphi)^{1+\delta}\right) \label{J20} \\
    &\hspace{3em}+O\left(\frac{m}{\lambda}\int_{S}U_{x_1,\lambda} \sum\limits_{j=2}^m U_{x_j,\lambda}^{q}\right). \nonumber
\end{align}
We can deduce from Lemma 2.5 in \cite{GHPY} that
\begin{align}
    &\quad \int_{S}  -qU_{x_1,\lambda}^{q-1} \left( \sum\limits_{j=2}^m U_{x_j,\lambda} \right) \frac{\partial}{\partial \lambda}U_{x_1,\lambda} \nonumber \\
    &=\frac{1}{\lambda}\int_{\lambda(S - \{x_1 \})} -qU_{0,1}(z)^{q-1}\left(\frac{N}{q+1}U_{0,1}(z)+z\cdot \nabla U_{0,1}(z)\right) \left(\sum\limits_{j=2}^m U_{0,1}(z+\lambda(x_1-x_j))\right) \label{J201} \\
    & = \frac{1}{\lambda}\int_{\lambda(S - \{x_1 \})} -qU_{0,1}(z)^{q-1}\left(\frac{N}{q+1}U_{0,1}(z)+z\cdot \nabla U_{0,1}(z)\right) \sum\limits_{j=2}^m \left(\frac{a_{N,p}}{(\lambda|x_1-x_j|)^{N-2}} + \frac{Cz^2}{(\lambda|x_1-x_j|)^N}\right) \nonumber \\
    &=\frac{1}{\lambda}\sum\limits_{j=2}^m\frac{B_{21}}{(\lambda|x_1-x_j|)^{N-2}}+O\left(\frac{m^{N-2}}{\lambda^{N-1+\varepsilon}}\right)=\frac{1}{\lambda}\sum\limits_{j=2}^m\frac{B_{21}}{(\lambda|x_1-x_j|)^{N-2}}+O\left(\frac{1}{\lambda^{3+\varepsilon}}\right), \nonumber
\end{align}
where
\[
    \begin{split}
        B_{21}& =a_{N,p}\int_{\R^N} -qU_{0,1}(z)^{q-1}\left(\frac{N}{q+1}U_{0,1}(z)+z\cdot \nabla U_{0,1}(z)\right)\ \dz \\
        &  = -a_{N,p}\int_{\R^N} \left( \frac{qN}{q+1}U_{0,1}^q(z)+z\cdot \nabla (U_{0,1}^q(z))\right) \ \dz  = \frac{a_{N,p}N}{q+1} \int_{\R^N} U_{0,1}^q(z) \ \dz > 0.
    \end{split}
\]
Similarly, from Lemma \ref{nnl2-19-1}, we have
\begin{equation}\label{J202}
    \begin{aligned}
        \int_{S}  -qU_{x_1,\lambda}^{q-1}\frac{\partial}{\partial \lambda}U_{x_1,\lambda}\cdot \varphi=\frac{1}{\lambda}\sum\limits_{j=2}^m\frac{B_{22}}{(\lambda|x_1-x_j|)^{N-2}}+O\left(\frac{1}{\lambda^{3+\varepsilon}}\right),
    \end{aligned}
\end{equation}
where
\begin{align*}
    B_{22} & =-B_{0}\int_{\R^N} qU_{0,1}(z)^{q-1}\left(\frac{N}{q+1}U_{0,1}(z)+z\cdot \nabla U_{0,1}(z)\right)\cdot w(z) \ \dz \\
    & = -B_{0}\int_{\R^N} \left(\frac{qN}{q+1}U_{0,1}^q(z)+z\cdot \nabla (U_{0,1}^q(z))\right)\cdot w(z) \ \dz \\
    & = B_{0}\int_{\R^N} -\Delta V_{0,1}(z) \left(\frac{N}{q+1}w(z)+z\cdot \nabla w(z)\right) \ \dz  \\
    & = B_0 \int_{\R^N}  \left(\frac{N(q+2)}{q+1}V_{0,1}^p(z) - \frac{1}{p}z\cdot \nabla (V_{0,1}^p(z))\right) \ \dz  \\
    & = B_0 \int_{\R^N}  \left(\frac{N(q+2)}{q+1} + \frac{N}{p} \right)V_{0,1}^p(z) \ \dz > 0.
\end{align*}
Here we used the equation for $w$: $-\Delta w = V_{0,1}^{p-1}$ in $\R^N$. In addition, since we can choose $\delta$ small such that $(q-\delta)(N-2)>N$, we can obtain
\begin{equation}\label{J203}
    \int_{S}U_{x_1,\lambda}^{q-\delta} \left(\sum\limits_{j=2}^m U_{x_j,\lambda}+\varphi \right)^{1+\delta}\leq C\left(\frac{m^{N-2}}{\lambda^{N-2}}\right)^{1+\delta}\int_{S}U_{x_1,\lambda}^{q-\delta}\leq O\left(\frac{1}{\lambda^{2+\varepsilon}} \right).
\end{equation}
Similarly, we have
\begin{align}\label{J204}
    \int_{S}U_{x_1,\lambda} \sum\limits_{j=2}^m U_{x_j,\lambda}^{q}\leq C \int_{S}U_{x_1,\lambda}^{q-\delta} \left(\sum\limits_{j=2}^m U_{x_j,\lambda}\right)^{1+\delta}=O\left(\frac{1}{\lambda^{2+\varepsilon}}\right).
\end{align}
Combine \eqref{J20}-\eqref{J204}, we have
\begin{equation}\label{J205}
    J_{20}=m\sum\limits_{j=2}^m\frac{B_2}{\lambda^{N-1}|x_1-x_j|^{N-2}}+O\left(\frac{m}{\lambda^{3+\varepsilon}}\right),
\end{equation}
where $B_2=B_{21}+B_{22}>0.$

Regarding $J_{21}$, we can use \eqref{J203} to get
\begin{align}\label{J21}
    J_{21} & =m\int_{S} \left( \sum\limits_{j=1}^m U^q_{x_j,\lambda}-(Y^*_{\bar{r}, \bar{y}'', \lambda})^{q}  \right)\frac{\partial}{\partial \lambda}\left(\sum\limits_{j=2}^m U_{x_j,\lambda}+\varphi\right)\nonumber\\
    &\leq \frac{Cm}{\lambda}\int_{S} U_{x_1,\lambda}^{q-1}\left(\sum\limits_{j=2}^m U_{x_j,\lambda}+\varphi\right)^2+U_{x_1,\lambda}^{q-1-\delta}\left(\sum\limits_{j=2}^m U_{x_j,\lambda}+\varphi \right)^{2+\delta}+\sum\limits_{j=2}^m U_{x_j,\lambda}^{q}\left(\sum\limits_{j=2}^m U_{x_j,\lambda}+\varphi \right)\nonumber\\
    &\leq \frac{Cm}{\lambda} \int_{S}U_{x_1,\lambda}^{q-\delta} \left( \sum\limits_{j=2}^m U_{x_j,\lambda}+\varphi \right)^{1+\delta}=O\left(\frac{m}{\lambda^{3+\varepsilon}}\right).
\end{align}

Next, we estimate $J_{22}$. From Lemma \ref{lemma:U}, we have
\begin{align}
        J_{22} & \leq \frac{Cm}{\lambda}\int_{\Omega_1\backslash S} (Y^*_{\bar{r}, \bar{y}'', \lambda})^{q+1} \nonumber \\
        & \le Cm\lambda^{N-1} \int_{\Omega_1 \setminus S} \left(\sum_{j=1}^m \frac{1}{|\lambda(y-x_j)|^{N-2}}\right)^{q+1} \dy \label{J_22_1} \\
        &\ + Cm\lambda^{N-1} \int_{\Omega_1 \setminus S} \left(\sum_{j=1}^m \left(\frac{m}{\lambda} \right)^{p(N-2)-2} \frac{1}{(1+m|y-x_j|)^{\min\{p(N-3-\theta)-2, N-2\}}}\right)^{q+1}\dy  \nonumber \\
        & : = J_{221} + J_{222}. \nonumber
\end{align}
We first estimate $J_{221}$:
\begin{equation*}
    J_{221}\leq C \frac{m}{\lambda} \int_{\Omega_1 \setminus S} \frac{\lambda^N dy}{|\lambda(y-x_1)|^{(N-2-\tau)(q+1)}}
    \le C\left(\frac{m}{\lambda}\right)^{(q+1)(N-2-\tau)-N+1} = O\left(\frac{m}{\lambda^{3+\varepsilon}}\right),
\end{equation*}
where we have used $(q+1)(N-2-\tau)-N>N-2$.

Next, we estimate $J_{222}$. If $p(N-3)>N$, $J_{222}$ can be bounded by $J_{221}$. Hence, we only need to estimate $J_{222}$ when $p(N-3)\le N$. We follow the same strategy for $J_{12}.$ If $N \ge 6$, then
\begin{equation}\label{1-25-3}
    [p(N-3)-2-\tau](q+1)>N.
\end{equation}
Hence, we have
\begin{align*}
     J_{222} \leq & C\left(\frac{m}{\lambda}\right)^{p(q+1)(1+\theta)+1}\int_{\Omega_1 \setminus S} \frac{\lambda^{N}\dy}{(1+\lambda|y-x_1|)^{(p(N-3-\theta)-2-\tau)(q+1)}} \\
     &\leq C\left(\frac{m}{\lambda}\right)^{(q+1)(p(N-2)-2-\tau)-N+1} = O\left(\frac{m}{\lambda^{3+\varepsilon}}\right),
\end{align*}
since $(q+1)(p(N-2)-2-\tau)-N>(N-2-\tau)(q+1)-N>N-2$.

If $N = 5$, then \eqref{1-25-3} may not hold. Reasoning as in the case of $J_{12}$, we choose $R_0$ large enough. Then it holds
\begin{equation}
    \begin{split}
        |J_{222}| & \leq Cm\lambda^{N-1} \left( \int_{(\Omega_1 \cap \textbf{B}_{R_0}) \backslash S} + \int_{\Omega_1 \backslash \textbf{B}_{2\delta}}\right) \left( \sum_{j=1}^m \left(\frac{m}{\lambda}\right)^{p(N-2)-2} \frac{1}{(1+m|y-x_j|)^{p(N-3-\theta)-2}}\right)^{q+1}\dy \\
        & \leq C m\lambda^{N-1} \int_{\Omega_1 \backslash \textbf{B}_{2\delta}} \left( \left(\frac{m}{\lambda}\right)^{p(1+\theta)} \frac{m}{(1+\lambda|y-x_j|)^{p(N-3-\theta)-2}} \right)^{q+1} \ \dy + O\left(\frac{m}{\lambda^{3 + \varepsilon}}\right) \\
        & \leq C\left( \frac{m}{\lambda} \right)^{1+p(1+\theta)(q+1)} \frac{m^{q+1}}{\lambda^{(q+1)(p(N-3-\theta)-2)-N}} + O\left(\frac{m}{\lambda^{3 + \varepsilon}}\right) \\
        & = O\left(\frac{m}{\lambda^{3 + \varepsilon}}\right),
    \end{split}
\end{equation}
where we have used $[p(N-3)-2](q+1)>5$ when $N = 5$ and $p>2$. Thus, we have
\begin{align}\label{J22}
    J_{22} = J_{221} + J_{222} = O\left(\frac{m}{\lambda^{3+\varepsilon}}\right).
\end{align}
Combine \eqref{J205}, \eqref{J21} and \eqref{J22}, we have
\begin{equation}\label{I2}
    I_2 = J_{20} + J_{21} + J_{22} = m\sum\limits_{j=2}^m\frac{B_2}{\lambda^{N-1}|x_1-x_j|^{N-2}}+O\left(\frac{m}{\lambda^{3+\varepsilon}}\right).
\end{equation}

 \textbf{Estimate $E_1$, $E_2$ and $E_3$.} Recall that $\text{supp} \ \xi = \textbf{B}_{2 \delta}= \{ y = (y',y''):|(y',y'') - (r_0,y_0'')| <2\delta \}$, and $\xi=1$ in $\textbf{B}_{\delta}$. Through a similar computation as $I_1$, we have
     \begin{equation}\label{E_1}
         \begin{aligned}
             E_1&=O\left( \int_{\textbf{B}_{\delta}^c}V(y)\frac{\partial}{\partial \lambda}(Y^*_{\bar{r}, \bar{y}'', \lambda}\cdot Z^*_{\bar{r}, \bar{y}'', \lambda})\right)=O\left(\frac{1}{\lambda^{\varepsilon}} \right)I_1=O\left(\frac{m}{\lambda^{3+\varepsilon}}\right).
         \end{aligned}
     \end{equation}
    Then, through a similar computation as in $J_{22}$, we have
    \begin{equation}\label{E_2}
        \begin{aligned}
            E_2&=O\left(m\int_{\Omega_1\cap B^c}\frac{\partial}{\partial \lambda}(Z^*_{\bar{r}, \bar{y}'', \lambda})^{p+1}+\frac{\partial}{\partial \lambda}(Y^*_{\bar{r}, \bar{y}'', \lambda})^{q+1}\right)\\
            &\leq O\left(m\int_{\Omega_1\cap B^c}\frac{\partial}{\partial \lambda}(Z^*_{\bar{r}, \bar{y}'', \lambda})^{p+1}\right)+O(J_{22})=O\left(\frac{m}{\lambda^{3+\varepsilon}}\right).
        \end{aligned}
    \end{equation}
    Finally, we estimate ${E_3}$. Since $\nabla Y^*_{\bar{r}, \bar{y}'', \lambda}(y)=O\big(Y^*_{\bar{r}, \bar{y}'', \lambda}(y)\big)$ and $\nabla Z^*_{\bar{r}, \bar{y}'', \lambda}(y)=O\big(Z^*_{\bar{r}, \bar{y}'', \lambda}(y)\big)$ in $A=\{ y = (y',y''): \delta < |(y',y'') - (r_0,y_0'')| <2\delta \}$, we can estimate each term in $E_3$ in a similar way to $I_1$ to obtain
     \begin{align}\label{E_3}
         E_3=O\left(\frac{1}{\lambda}\int_{A}Y^*_{\bar{r}, \bar{y}'', \lambda}Z^*_{\bar{r}, \bar{y}'', \lambda}\right)= O\left(\frac{m}{\lambda^{3+\varepsilon}}\right).
     \end{align}
     Combine \eqref{I1}, \eqref{I2}, \eqref{E_1}, \eqref{E_2} and \eqref{E_3}, we can get the desired expansion of $\frac{\partial I (Y^*_{\bar{r}, \bar{y}'', \lambda},Z^*_{\bar{r}, \bar{y}'', \lambda})}{\partial \lambda}$.
\end{proof}

Following similar arguments as in Lemma \ref{expansion_lambda}, we can prove the following lemmas.
\begin{lemma}\label{expansion_2}
    If $N\geq5$, then it holds that
    \[
        \frac{\partial I(Y_{\bar{r}, \bar{y}'', \lambda},Z_{\bar{r}, \bar{y}'', \lambda})}{\partial \bar{r}}=m\left( \frac{B_1}{\lambda^2}\frac{\partial V(\bar{r},\bar{y}'')}{\partial \bar{r}}+\sum\limits_{j=2}^m\frac{B_2}{\bar{r}\lambda^{N-1}|x_1-x_j|^{N-2}}+O\left(\frac{1}{\lambda^{1+\varepsilon}}\right)\right),
    \]
    and
    \[
        \frac{\partial I(Y_{\bar{r}, \bar{y}'', \lambda},Z_{\bar{r}, \bar{y}'', \lambda})}{\partial \bar{y_k}''}=m\left( \frac{B_1}{\lambda^2}\frac{\partial V(\bar{r},\bar{y}'')}{\partial \bar{y_k}''}+O\left(\frac{1}{\lambda^{1+\varepsilon}}\right)\right),\,k=3,\cdots,N,
    \]
    where $B_1$ and $B_2$ are the same positive constants in Lemma \ref{expansion_lambda}.
\end{lemma}

\section{Some technical Estimates}
In this section, we present some technical estimates used in the previous sections. The proof of Lemma D.1-D.4 are essential and could be found in \cite{WY}.
\begin{lemma}\label{A1}
    It holds that
    \begin{equation*}
    \sum\limits_{j=2}^{m} \dfrac{1}{|x_j-x_1|^{\alpha}} =
        \begin{cases}
            O(m^{\alpha}/\bar{r}^{\alpha}), \;\; \alpha>1; \\
            O(m^{\alpha} \log m/\bar{r}^{\alpha}), \;\; \alpha=1; \\
            O(m/\bar{r}^{\alpha}), \;\; 0<\alpha<1; \\
        \end{cases}
    \end{equation*}
\end{lemma}

\begin{lemma}\label{B1}
   For any $\alpha > 0$, we have
   \[
      \sum\limits_{j=1}^{m} \dfrac{1}{(1+|y-x_j|)^{\alpha}} \leq C\left( 1 + \sum\limits_{j=2}^{m}\dfrac{1}{|x_1 - x_j|^{\alpha}} \right).
   \]
   Here, the constant $C>0$ does not depend on $k$.
\end{lemma}

\begin{lemma}\label{B2}
 Suppose $\alpha > 1$ and $\beta > 1$ and $i \neq j$. Then, for any $\sigma \in [0, min (\alpha, \beta)]$, we have
   \[
      \dfrac{1}{(1+|y-x_i|)^{\alpha}}\dfrac{1}{(1+|y-x_j|)^{\beta}} \leq \dfrac{C}{|x_i - x_j|^{\sigma}} \left( \dfrac{1}{(1+|y-x_i|)^{\alpha+\beta-\sigma}} + \dfrac{1}{(1+|y-x_j|)^{\alpha+\beta-\sigma}} \right),
   \]
   where $C$ is a positive constant.
\end{lemma}

\begin{lemma}\label{B3}
   If $\sigma \in (0,N-2)$, we have
   \[
      \int_{\R^N} \dfrac{1}{|y-z|^{N-2}} \dfrac{1}{(1+|z|)^{2+\sigma}} dz \leq \dfrac{C}{(1+|y|)^{\sigma}}.
   \]
   If $\sigma > N-2$, we have
   \[
      \int_{\R^N} \dfrac{1}{|y-z|^{N-2}} \dfrac{1}{(1+|z|)^{2+\sigma}} dz \leq \dfrac{C}{(1+|y|)^{N-2}}.
   \]
\end{lemma}

\begin{lemma}[Lemma B.3, \cite{GKPY}]\label{B1_1}
Suppose that $p > 1$ and \eqref{critical hyperbola} holds. Suppose $\tau = \frac{N-4}{N-2}$ and $m \approx \lambda^{\tau} $. Then for any $\tau' \ge \tau$, we have
\begin{equation}\label{0-6-2}
    \left(\sum_{j=1}^m \frac{1}{(1+\lambda|y-x_j|)^{\frac{N}{p+1}+\tau'}}\right)^{p} \le C\sum_{j=1}^m \frac{ 1 }{ (1+\lambda |y-x_j| )^{  \frac{N}{q+1}+2+\tau' } }
\end{equation}
and
\[
    \left(\sum_{j=1}^m \frac{1}{(1+\lambda|y-x_j|)^{\frac{N}{q+1}+\tau'}}\right)^{q} \le C\sum_{j=1}^m \frac{ 1 }{ (1+\lambda |y-x_j| )^{  \frac{N}{p+1}+2+\tau' } }.
\]
\end{lemma}

\begin{lemma}\label{B7}
    It holds that
    \begin{equation*}
        \begin{aligned}
            \left| \left\langle -\Delta Z_{1,h}  - q (Y_{\bar{r}, \bar{y}'',\lambda})^{q-1}Y_{1,h}, \phi \right\rangle\right|=O\left(\frac{\lambda^{n_h}||(\phi,\psi)||_{*}}{\lambda^{1+\varepsilon}} \right),
        \end{aligned}
    \end{equation*}
    and
    \begin{equation*}
    \begin{aligned}
        \left|\left\langle -\Delta Y_{1,h}  - p (Z_{\bar{r}, \bar{y}'',\lambda})^{p-1}Z_{1,h}, \psi \right\rangle \right|=O\left( \frac{\lambda^{n_h}||(\phi,\psi)||_{*}}{\lambda^{\varepsilon}} \right),
    \end{aligned}
\end{equation*}
where $(\phi,\psi)$ is the same as in \eqref{estimate_1}.
\end{lemma}
\begin{proof}From the definition of $Y_{1,h}$, $Z_{1,h}$ and $Y_{\bar{r},\bar{y}'',\lambda}$, we have
    \begin{equation}\label{cZ}
    \begin{aligned}
         &\quad\left\langle -\Delta Z_{1,h}  - q (Y_{\bar{r}, \bar{y}'',\lambda})^{q-1}Y_{1,h}, \phi \right\rangle \\
        &= \int_{\R^N} \phi(y) \left( - \Delta \dfrac{\partial (\xi V_{x_1,\lambda})}{\partial \Box_h} + \Delta \dfrac{\partial (V_{x_1,\lambda})}{\partial \Box_h} \right) \ \dy\\
        & \quad+ \int_{\R^N} \phi(y) \left(  q (Y_{\bar{r}, \bar{y}'',\lambda}^*)^{q-1} \dfrac{\partial U_{x_1,\lambda}}{\partial \Box_h} - q ( Y_{\bar{r}, \bar{y}'',\lambda} )^{q-1}\dfrac{\partial (\xi U_{x_1,\lambda})}{\partial \Box_h} \right) \ \dy \\
        & \quad+ \left( \int_{S} + \int_{\Omega_1 \backslash S} + \int_{\R^N \backslash \Omega_1} \right)\phi(y) \left( - \Delta \dfrac{\partial (V_{x_1,\lambda})}{\partial \Box_h} - q (Y_{\bar{r}, \bar{y}'',\lambda}^*)^{q-1} \dfrac{\partial U_{x_1,\lambda}}{\partial \Box_h} \right) \ \dy \\
        & : = \bar{I}_1+\bar{I}_2+\bar{I}_3 + \bar{I}_4 + \bar{I}_5.
    \end{aligned}
    \end{equation}

\textbf{Estimate $\bar{I}_1$}. Recall that $A = \{(r,y'')| \ \delta <|(r,y'') - (r_0, y_0'') | \leq 2 \delta\}$, we have
\begin{equation*}
\begin{aligned}
    \left| - \Delta \dfrac{\partial (\xi V_{x_1,\lambda})}{\partial \Box_h} + \Delta \dfrac{\partial (V_{x_1,\lambda})}{\partial \Box_h} \right| & = \left| -\Delta \xi \cdot \dfrac{\partial V_{x_1,\lambda}}{\partial \Box_h} - 2 \nabla \xi \cdot \nabla \dfrac{\partial V_{x_1,\lambda}}{\partial \Box_h} + (1-\xi) \dfrac{\partial (\Delta  V_{x_1,\lambda})}{\partial \Box_h} \right| \\
    & \leq C \chi_{A}\lambda^{n_h}(V_{x_1,\lambda} +  |\nabla V_{x_1,\lambda}|) + C \lambda^{n_h}(1-\xi)U_{x_1,\lambda}^q \\
    & \leq C  \dfrac{\chi_A\lambda^{n_h+\frac{N}{p+1}}}{(1+\lambda|y-x_1|)^{N-2}} + C (1- \xi) \dfrac{\lambda^{n_h + \frac{Nq}{q+1}}}{(1+\lambda|y-x_1|)^{(N-2)q}},
\end{aligned}
\end{equation*}
where $\chi_A$ is the characteristic function in $A$. Recall that $\textbf{B}_{\delta}= \{y=(r,y''):|(r,y'') - (r_0,y_0'')| < \delta\}$, we have
\begin{equation}\label{bar_I_1}
\begin{split}
    |\bar{I}_1| & \leq C \int_{\R^N} |\phi(y)| \left( \chi_A \dfrac{\lambda^{n_h +\frac{N}{p+1}}}{(1+\lambda|y-x_1|)^{N-2}} + C (1- \xi) \dfrac{\lambda^{n_h + \frac{Nq}{q+1}}}{(1+\lambda|y-x_1|)^{(N-2)q}}  \right)  \ \dy\\
    & \leq C ||\phi||_{*,1} \int_{A} \dfrac{\lambda^{n_h +\frac{N}{p+1}}}{(1+\lambda|y-x_1|)^{N-2}} \left( \sum\limits_{j=1}^m \dfrac{\lambda^{\frac{N}{q+1}}}{(1+\lambda|y-x_j|)^{\frac{N}{q+1}+\tau}} \right)  \ \dy\\
    & \quad + C ||\phi||_{*,1} \int_{\textbf{B}_{\delta}^c} \dfrac{\lambda^{n_h + \frac{Nq}{q+1}}}{(1+\lambda|y-x_1|)^{(N-2)q}}  \left( \sum\limits_{j=1}^m \dfrac{\lambda^{\frac{N}{q+1}}}{(1+\lambda|y-x_j|)^{\frac{N}{q+1}+\tau}} \right)  \ \dy \\
    & \leq \dfrac{C \lambda^{n_h}||\phi||_{*,1} }{\lambda^{\frac{N}{q+1}}}=O\left(\dfrac{ \lambda^{n_h}||\phi||_{*,1} }{\lambda^{1+\varepsilon}} \right) .
\end{split}
\end{equation}

\textbf{Estimate $\bar{I}_2$}. From Lemma \ref{l1-23-4} we can compute that
\begin{equation}\label{bar_I_2}
    \begin{aligned}
        |\bar{I}_2|  &\leq C\lambda^{n_h}\int_{\textbf{B}_{\delta}^c} \phi(y)(Y_{\bar{r}, \bar{y}'',\lambda}^*)^{q-1}U_{x_1,\lambda} \ \dy\\
        &\leq C\lambda^{n_h - \varepsilon}||\phi||_{*,1}\int_{\textbf{B}_{\delta}^c} \frac{\lambda^N}{(1+\lambda|y-x_1|)^{N-2}}
        \left(\sum\limits_{j=1}^m \dfrac{1}{(1+\lambda|y-x_j|)^{\frac{N}{q+1}+\tau}}\right)^{q}  \ \dy\\
        & =O\left(\dfrac{ \lambda^{n_h}||\phi||_{*,1} }{\lambda^{1+\varepsilon}} \right).
    \end{aligned}
\end{equation}

\textbf{Estimate $\bar{I}_3$}. From \ref{nnl2-19-1}, we have
\[
    \sum_{j=2}^m U_{x_j,\lambda} + \varphi \leq C \min \left\{ U_{x_1,\lambda}, \frac{m^{N-2}}{\lambda^{N-2}} \right\}, \ \ \text{in} \ \ S.
\]
Therefore,

\begin{equation}\label{bar_I_3}
    \begin{aligned}
        |\bar{I}_3| &= \int_{S} \phi(y) \left( qU_{x_1,\lambda}^{q-1} \dfrac{\partial U_{x_1,\lambda}}{\partial \Box_h}- q (Y_{\bar{r}, \bar{y}'',\lambda}^*)^{q-1} \dfrac{\partial U_{x_1,\lambda}}{\partial \Box_h}\right) \ \dy \\
        & \leq C\lambda^{n_h}\int_{S}\phi(y)U_{x_1,\lambda}^{q-1}\left(\sum\limits_{j=2}^m U_{x_j,\lambda}+\varphi\right) \ \dy\\
        &\leq C\lambda^{n_h} ||\phi||_{*,1}\left(\frac{m}{\lambda}\right)^{N-2}\int_{B_{C\lambda/m}(0)}\frac{1}{(1+|z|)^{(N-2)(q-1)+\frac{N}{q+1}}} \ \dy \\
        & = C\lambda^{n_h}||\phi||_{*,1} \left( \frac{m}{\lambda} \right)^{(N-2)(q-1)+\frac{N}{q+1} -2} = O\left(\dfrac{ \lambda^{n_h}||\phi||_{*,1} } {\lambda^{1+\varepsilon}}\right),
    \end{aligned}
\end{equation}
where the final inequality holds due to $(N-2)(q-1)+\frac{N}{q+1}-2 = \frac{qN}{p+1} > \frac{N+2}{2}$.

\textbf{Estimate $\bar{I}_4$}. Using Lemma \ref{lemma:U}, we have
\begin{equation*}
    \begin{aligned}
        |\bar{I}_4 |& \leq C \lambda^{n_h}  \int_{\Omega_1\backslash S} |\phi(y)|(Y_{\bar{r}, \bar{y}'',\lambda}^*)^{q-1}U_{x_1,\lambda} \ \dy\\
        &\leq C\lambda^{n_h}||\phi||_{*,1} \int_{\Omega_1\backslash S} \frac{\lambda^N}{(1+\lambda|y-x_1|)^{N-2}}\sum\limits_{j=1}^m\frac{1}{(1+\lambda|y-x_j|)^{\frac{N}{q+1}+\tau}}\\
        &\quad\times\left( \left(\sum_{j=1}^m \frac{1}{(1+\lambda|y-x_j|)^{N-2}} \right)^{q-1}+ \left( \left(\frac{m}{\lambda}\right)^{p(N-2)-2} \sum_{j=1}^m \frac{1}{(1+m|y-x_j|)^{\min\{N-2,p(N-3-\theta)-2\}}}\right)^{q-1}\right) \ \dy\\
        &:=\bar{J}_1+\bar{J}_2.
    \end{aligned}
\end{equation*}

Thus, from \eqref{I_32_2} we can obtain
\begin{equation}\label{I_32_J_1}
    \begin{aligned}
        \bar{J}_1&\leq C\lambda^{n_h}||\phi||_{*,1}\int_{\Omega_1\backslash S}\frac{\lambda^N}{(1+\lambda|y-x_1|)^{N-2+\frac{N}{q+1}+(N-2-\tau)(q-1)}} \ \dy\\
        &\leq C \lambda^{n_h}||\phi||_{*,1}\cdot \left(\frac{m}{\lambda} \right)^{\frac{N}{q+1} -2 + (N-2-\tau)(q-1)}=O\left(\frac{\lambda^{n_h}||\phi||_{*,1}}{\lambda^{1+\varepsilon}}\right),
    \end{aligned}
\end{equation}
since $\frac{N}{q+1} -2 + (N-2-\tau)(q-1)=\frac{qN}{p+1}-\tau(q-1)>\frac{N-2}{2}$ when $p\in (\frac{N}{N-2},\frac{N+2}{N-2}]$.

Next, we estimate $\bar{J}_2$. If $p(N-3)>N$, then $\bar{J}_2$ is bounded by $\bar{J}_1$. We next estimate $\bar{J}_2$ for the case $p(N-3)\leq N$.
Similar to \eqref{I_32_J_1}, we have
\begin{equation}\label{I_32_J_2}
    \begin{aligned}
        \bar{J}_2&\leq C\lambda^{n_h}||\phi||_{*,1} \int_{\Omega_1\backslash S} \frac{\lambda^N}{(1+\lambda|y-x_1|)^{N-2+\frac{N}{q+1}}}\left( \left( \frac{m}{\lambda} \right)^{p(1+\theta)} \sum_{i=1}^m \frac{1}{(1+\lambda|y-x_i|)^{p(N-3-\theta)-2}}\right)^{q-1} \ \dy\\
        &\leq C\lambda^{n_h}||\phi||_{*,1} \left(\frac{m}{\lambda}\right)^{p(q-1)(1+\theta)} \int_{\Omega_1\backslash S}\frac{\lambda^N}{(1+\lambda|y-x_1|)^{N-2+\frac{N}{q+1}+(q-1)(p(N-3-\theta)-2-\tau)}} \ \dy\\
        &\leq C\lambda^{n_h}||\phi||_{*,1}\left(\frac{m}{\lambda}\right)^{-2+\frac{N}{q+1}+(q-1)(p(N-2)-2-\tau)}=O\left(\frac{\lambda^{n_h}||\phi||_{*,1}}{\lambda^{1+\varepsilon}}\right),
    \end{aligned}
\end{equation}
where we have used $-2+\frac{N}{q+1}+(q-1)(p(N-2)-2-\tau)>\frac{N-2}{2}$ if $p\in (\frac{N}{N-2},\frac{N}{N-3}]$. Combine \eqref{I_32_J_1} and \eqref{I_32_J_2}, we have
\begin{equation}\label{bar_I_4}
    |\bar{I}_4| \leq \bar{J}_1 +\bar{J}_2 = O\left(\frac{\lambda^{n_h}||\phi||_{*,1}}{\lambda^{1+\varepsilon}}\right).
\end{equation}

\textbf{Estimate $\bar{I}_5$.} Split $\R^N \backslash \Omega_1$ into $\cup_{i=2}^m \Omega_i$, and choose two constants $\sigma_1,\sigma_2$ such that
\[
    \begin{gathered}
        \frac{N-2}{2} < \sigma_1 < \beta_1:= \frac{N}{q+1}+(q-1)(N-2-\tau)-2 , \\
        \frac{N-2}{2} - p(q-1)(1+\theta) < \sigma_2 < \beta_2 := \frac{N}{q+1}+(q-1)(p(N-3-\theta)-2-\tau)-2.
    \end{gathered}
\]
Then similar to the estimate of $\bar{I}_4$, we have
\begin{equation*}
    \begin{aligned}
        | \bar{I}_5| & \leq C\lambda^{n_h}\sum\limits_{i=2}^m\int_{\Omega_i} |\phi(y)|\cdot (Y_{\bar{r}, \bar{y}'',\lambda}^*)^{q-1} U_{x_1,\lambda} \ \dy\\
        &\leq C\lambda^{n_h}||\phi||_{*,1}\sum\limits_{i=2}^m\int_{\Omega_i}
        \frac{\lambda^N}{(1+\lambda|y-x_1|)^{N-2}}\left( \frac{1}{(1+\lambda|y-x_i|)^{\beta_1 + 2}} \right) \ \dy \\
        & \quad + C\lambda^{n_h}||\phi||_{*,1}\sum\limits_{i=2}^m\int_{\Omega_i}
        \frac{\lambda^N}{(1+\lambda|y-x_1|)^{N-2}} \left( \left(\frac{m}{\lambda} \right)^{p(q-1)(1+\theta)}\frac{\chi_{\{p(N-3)\leq N\}}}{(1+\lambda|y-x_i|)^{\beta_2 + 2 }}\right) \ \dy\\
        &\leq C\lambda^{n_h}||\phi||_{*,1}\sum\limits_{i=2}^m
        \frac{1}{(\lambda|x_1-x_i|)^{\sigma_1}} \int_{\Omega_i} \frac{\lambda^N}{(1+\lambda|y-x_i|)^{\beta_1+N - \sigma_1}} \ \dy\\
        & \quad + C\lambda^{n_h}||\phi||_{*,1} \left(\frac{m}{\lambda}\right)^{p(q-1)(1+\theta)}\sum\limits_{i=2}^m
        \frac{\chi_{\{p(N-3)\leq N\}}}{(\lambda|x_1-x_i|)^{\sigma_2}} \int_{\Omega_i} \frac{\lambda^N}{(1+\lambda|y-x_i|)^{\beta_2+N - \sigma_2}} \ \dy\\
        & = C\lambda^{n_h}||\phi||_{*,1} \left( \left(\frac{m}{\lambda} \right)^{\sigma_1} + \chi_{\{p(N-3)\leq N\}}\left(\frac{m}{\lambda} \right)^{p(q-1)(1+\theta) + \sigma} \right) \\
        &=O\left(\frac{\lambda^{n_h}||\phi||_{*,1}}{\lambda^{1+\varepsilon}}\right) .
    \end{aligned}
\end{equation*}
    Thus, we prove
    \begin{align}\label{bar_I_5}
        |\bar{I}_5|=O\left(\dfrac{ \lambda^{n_h}||\phi||_{*,1} }{\lambda^{1+\varepsilon}} \right) .
    \end{align}
Combine \eqref{cZ} \eqref{bar_I_1} \eqref{bar_I_2}, \eqref{bar_I_3}, \eqref{bar_I_4} and \eqref{bar_I_5}, we have
\begin{equation}\label{c-linearize1}
    \begin{split}
        \left| \left\langle -\Delta Z_{1,h}  - q (Y_{\bar{r}, \bar{y}'',\lambda})^{q-1}Y_{1,h}, \phi \right\rangle \right| = O\left(\dfrac{ \lambda^{n_h}||\phi||_{*,1} }{\lambda^{1+\varepsilon}}\right) .
\end{split}
\end{equation}
Similarly, we can prove that if $N\geq 5$, $p\in (\frac{N}{N-2},\frac{N+2}{N-2}]$, then
\begin{equation}\label{c-linearize2}
    \begin{split}
        \left| \left\langle -\Delta Y_{1,h}  - p (Z_{\bar{r}, \bar{y}'',\lambda})^{p-1}Z_{1,h}, \psi \right\rangle \right| = O\left(\dfrac{ \lambda^{n_h}||\psi||_{*,1} }{\lambda^{\varepsilon}}\right) .
\end{split}
\end{equation}
Note that the estimate for the left-hand side of \eqref{c-linearize2} could be improved to $O\left(\dfrac{ \lambda^{n_h}||\psi||_{*,1} }{\lambda^{1+\varepsilon}}\right)$ by imposing suitable conditions on $p$, such as $p\in (\frac{N+1}{N-2},\frac{N+2}{N-2}]$, $N\geq 5$.

Combine \eqref{c-linearize1} and \eqref{c-linearize2}, we can get the desired result.

\end{proof}

\section*{Acknowledgments}
Yuxia Guo was supported by the National Natural Science Foundation of China(No. 12271283) and  National Key R\&D Program (2023YFA1010002).

\section*{Statements and Declarations}

The authors confirm that there are no relevant financial or non-financial competing interests to report.
	
\section*{Data Availability Statements}

All data generated or analyzed during this study are included in this article.

\end{document}